\documentclass[sn-mathphys-num]{sn-jnl}% Math and Physical Sciences Numbered Reference Style 
\usepackage{graphicx}%
\usepackage{multirow}%
\usepackage{amsmath,amssymb,amsfonts}%
\usepackage{amsthm}%
\usepackage{mathrsfs}%
\usepackage[title]{appendix}%
\usepackage{xcolor}%
\usepackage{textcomp}%
\usepackage{manyfoot}%
\usepackage{booktabs}%
\usepackage{algorithm}%
\usepackage{algorithmicx}%
\usepackage{algpseudocode}%
\usepackage{listings}%
\usepackage{enumitem}
\theoremstyle{thmstyleone}%
\newtheorem{theorem}{Theorem}[section]% meant for sectionwise numbers
\newtheorem{proposition}[theorem]{Proposition}% 
\newtheorem{lemma}[theorem]{Lemma}
\newtheorem{corollary}[theorem]{Corollary}
\newtheorem{assumption}[theorem]{Assumption}

\newtheorem{definition}[theorem]{Definition}

\theoremstyle{thmstyletwo}%

\theoremstyle{thmstylethree}
\newtheorem{remark}[theorem]{Remark}

\newcommand{\R}{\mathbb{R}}
\newcommand{\EX}{\mathbb{E}^X}
\newcommand{\E}{\mathbb{E}}
\newcommand{\D}{D[0,\infty)}

\makeatletter
	\@addtoreset{equation}{section}\def \theequation{\thesection.\arabic{equation}}

\allowdisplaybreaks[4]

\begin{document}

\title[Convergence of processes time-changed by GMC]{Convergence of processes time-changed by Gaussian multiplicative chaos}


\author*[1]{\fnm{Takumu} \sur{Ooi}}\email{ooitaku@rs.tus.ac.jp}

\affil*[1]{\orgdiv{Department of Mathematics, Faculty of Science and Technology}, \orgname{Tokyo University of Science}, \orgaddress{\street{2641 Yamazaki}, \city{Noda-shi}, \postcode{2788510}, \state{Chiba-prefecture}, \country{Japan}}}


\abstract{As represented by the Liouville measure, Gaussian multiplicative chaos is a random measure constructed from a Gaussian field. Under certain technical assumptions, we prove the convergence of a process time-changed by Gaussian multiplicative chaos in the case the latter object is square integrable (the \(L^2\)-regime). As examples of the main result, we prove that, in the whole \(L^2\)-regime, the scaling limit of the Liouville simple random walk on \(\mathbb{Z}^2\) is Liouville Brownian motion and, as \(\alpha \to 1\), Liouville \(\alpha\)-stable processes on \(\mathbb{R}\) converge weakly to the Liouville Cauchy process.}


\maketitle

\section{Introduction}\label{intro}
There has been intensive research on Gaussian free fields, the Liouville measure and Liouville Brownian motion, due to their close relationship with important concepts in areas such as statistical mechanics, conformal field theory, random geometry and quantum gravity. Here, a massive Gaussian free field \(X=\{X_f\}_{f\in \mathcal{S}(\R^d)}\) on \(\R^d\) is a centred Gaussian system on the Schwartz space \(\mathcal{S}(\R^d)\) whose covariance kernel is \(\pi\) times the \(\lambda\)-order Green's function of Brownian motion on \(\R^d\). 

For the massive Gaussian free field \(X\) on \(\R^2\) and \(\gamma \in (0,2)\), a random measure formally represented by
\begin{equation}
\exp(\gamma X(x)-\frac{\gamma^2}{2}\E[X(x)^2])dx
\label{eq:I1}
\end{equation}
is called the Liouville measure on \(\R^2\). This was first constructed rigorously by Kahane \cite{K} and, since then, has been studied in many other articles, see \cite{B,S}, for example. It is known that the Liouville measure is singular with respect to the Lebesgue measure for \(\gamma\in (0,2)\), see \cite{K}. Moreover, it corresponds to the Riemann metric tensor of two-dimensional Liouville quantum gravity. See \cite{BP, DS} for details.

Liouville Brownian motion is a time-changed process of Brownian motion on \(\R^2\) by the Liouville measure. This was constructed in \cite{GRV} and similar results were proved in \cite{B2} simultaneously. We also refer the reader to \cite[Appendix A]{AK} for the construction of Liouville Brownian motion, or Appendix \ref{AppPCAF} of this article for the definition. Liouville Brownian motion is the canonical diffusion process under Liouville quantum gravity, and it is conjectured that Liouville Brownian motion is the scaling limit of simple random walks on various random planar maps, see \cite{DS, GRV}. In \cite{BeG, GMS}, it is proved that the scaling limit of random walks on random graphs called mated-CRT planar maps is Liouville Brownian motion. 

As a 1-dimensional counterpart of Liouville Brownian motion, the Liouville Cauchy process is considered in \cite{Bav} from the point of view of the trace process of Liouville Brownian motion. 
One of the reasons for the various distinct properties of Liouville Brownian motion is that the Green's function of Brownian motion on \(\R^2\) diverges logarithmically. Similarly, the Green's function of the Cauchy process on \(\R\) diverges logarithmically. So, when developing techniques for time-changed processes with a log-divergent Green's function, in the 1-dimensional case, it makes sense to consider the Liouville Cauchy process on \(\R\).

More general models of Gaussian free fields and the Liouville measure have been studied. For example, fractional Gaussian fields are Gaussian fields whose covariance kernels can be represented in terms of fractional order Laplacians, and properties of these fields are summarized in \cite{LSSW}. Gaussian multiplicative chaos is a random measure formally represented as \((\ref{eq:I1})\), with the Gaussian free field replaced by a more general Gaussian field. Berestycki \cite{B} constructed Gaussian multiplicative chaos for log-correlated Gaussian fields using an elementary approach.  Shamov \cite{S} constructed Gaussian multiplicative chaos for Gaussian fields indexed by elements of Hilbert spaces, and also proved the convergence of Gaussian multiplicative chaos under some assumptions by introducing a randomized shift and a new definition of Gaussian multiplicative chaos. Hager and Neuman \cite{HN} proved the convergence of Gaussian multiplicative chaos for fractional Brownian fields by using Berestycki's method \cite{B}.
In summary, convergence of Gaussian fields and Gaussian multiplicative chaos has been proved in various specific situations. This leads to the following general question.

\textit{If a sequence of processes and Gaussian multiplicative chaoses converge respectively, then is it the case that the associated time-changed processes also converge?}

In this paper, we answer this question under some assumptions on the Green's functions of the approximating processes on subsets of \(\R\) and \(\R^2\). In short, under Assumption \ref{assumption} below, we prove the process time-changed by a corresponding Gaussian multiplicative chaos on \(\R^2\) converges to Liouville Brownian motion on \(\R^2\) and the process time-changed by a corresponding Gaussian multiplicative chaos on \(\R\) converges to Liouville Cauchy process on \(\R\). In \cite{CHK}, the convergence of time-changed processes is studied in general cases where the limit process has positive capacity of a point. The novelty of our result is to establish general conditions for the convergence of time-changed processes by Gaussian multiplicative chaos to Liouville Brownian motion and Liouville Cauchy process, which are the case where the capacity of each point for the limit process is \(0\).

The rest of this paper is organized as follows. In Section \ref{main}, we introduce the settings and assumptions in detail, and then describe the main result concerning the weak convergence of time-changed processes of Hunt processes by Gaussian multiplicative chaos. Hunt processes are special cases of strong Markov processes and, since we consider time-changed processes in the sense of the theory of Dirichlet forms, we consider Hunt processes. The main theorem, Theorems \ref{mainconvergence}, includes cases of some limiting processes on both discrete and continuum spaces. Theorem \ref{mainconvergence} covers the 2-dimensional case about convergence to Liouville Brownian motion and the 1-dimensional case about convergence to the Liouville Cauchy process. We prove the main results in Section \ref{mainproof}, with a proof of one of the intermediate propositions postponed to Section \ref{secPCAFlem}. To prove the main results on the convergence of time-changed processes, we prove the convergence of Gaussian multiplicative chaos, the tightness of the pairs of original processes and PCAFs (positive continuous additive functionals), and convergence of finite-dimensional distributions. This strategy is basic, but one of the novelties lies in the fact that we regard PCAFs as Gaussian multiplicative chaoses of occupation measures of the relevant processes and apply methods for Gaussian multiplicative chaos to PCAFs. In Section \ref{examples}, we present two examples. The first one is the convergence of Liouville \(\alpha\)-stable processes on \(\R\) and the second one is the convergence of Liouville simple random walk on \(\frac{1}{\sqrt{n}}\mathbb{Z}^2\) to Liouville Brownian motion. We also include two appendices. Appendix \ref{AppPCAF} contains definitions and properties of a PCAF. Appendix \ref{AppSko} contains definitions and properties of Skorokhod's topology, in particular, properties for the weak convergence with respect to Skorokhod's topologies.

\section{Setting and Main result}\label{main}
In this section, we present the main results, in which we examine weak convergence of processes time-changed by Gaussian multiplicative chaos on both continuum and discrete spaces.

Throughout this paper, we fix \(d=1\) or \(d=2\). When \(d=2\), we consider convergence to Liouville Brownian motion on \(\R^2\), which is a time-changed process of Brownian motion on \(\mathbb{R}^2\) by Liouville measure, and, when \(d=1\), we consider convergence to Liouville Cauchy process on \(\R\), which is a time-changed process of the Cauchy process on \(\mathbb{R}\) by the corresponding Gaussian multiplicative chaos. Both Brownian motion on \(\mathbb{R}^2\) and the Cauchy process on \(\mathbb{R}\) have Green's functions that diverge logarithmically (see Lemma \ref{Green'slemma}, for example), which means the Liouville Cauchy process is a natural counterpart of Liouville Brownian motion when \(d=1.\) See \cite{Bav} for further background on the Liouville Cauchy process. For \(d\geq 3,\) there do not exist symmetric \(\alpha\)-stable processes on \(\R^d\) having Green's functions that diverge logarithmically, and so we cannot construct corresponding counterpart processes of Liouville Brownian motion in theses dimensions. See Example \ref{example1}, Appendix \ref{AppPCAF} or \cite{LSSW} for details. This is the reason why we only consider the cases \(d=1,2\).

We now state the setting and main theorems that includes both cases \(d=2\) and \(d=1\). Denote by \(m\) the Lebesgue measure on \((\R^d, \mathcal{B}(\R^d))\) and by \(Z^{\infty}=(\{Z_t^{\infty}\}_{t\geq 0}, \{\mathbb{P}_x^{Z^{\infty}}\}_{x\in \R^d})\) \(m\)-symmetric \(d\)-stable process on \(\R^d\). We remark that, when \(d=2\), \(Z^{\infty}\) is Brownian motion on \(\R^2\) and, when \(d=1\), \(Z^{\infty}\) is Cauchy process on \(\R\). For the heat kernel \(p^{\infty}\) of \(Z^{\infty}\) on \(\mathbb{R}^d\) and any fixed \(\lambda >0,\) we define the \(\lambda\)-order Green's function \(g_{\lambda}^{\infty}\) by setting
\[g_{\lambda}^{\infty}(x,y):=\int_0^{\infty} e^{-\lambda t} p^{\infty}(t,x,y)dt\]
for \(x,y \in \R^d\).

For \(n\in \mathbb{N}\), suppose \(Z^n=(\{Z_t^n\}_{t\geq 0}, \{\mathbb{P}_x^{Z^n}\}_{x\in \R^d})\) is a strongly recurrent \(m\)-symmetric Hunt process on \(\R^d\). Here, \(Z^n\) is called strongly recurrent if there exist local times at \(x\) for any \(x\in \R^d\). See \cite{CHK} for details of this term.

We consider the situation where one of the following conditions hold.

\begin{enumerate}[label=(\Alph*)]
  \item For any \(n\in \mathbb{N}\), \(Z^n\) has a continuous transition probability density function \(p^n(t,x,y)\) with respect to \(m\).
  \item For any \(n\in \mathbb{N}\), there exist a discrete subset \(D_n \subset \R^d\) such that \(\mathbb{P}_x^{Z^n}(Z^n_t-x \in D_n \text{\ for\ any\ }t)=1\), and a measurable map \(i_n:\R^d \to D_n\) satisfying \(\lim_{n\to \infty}|i_n(x)-x|=0\) for any \(x\in \R^d\), \(i_n(x)=x\) and \(C_n:=1/m(i^{-1}_n(\{x\}))\) is positive constant for any \(x\in D_n\). We also represent \(i_n(x)\) as \(x_n\). In this case, we define \(p^n(t,x,y):=\mathbb{P}_{x_n}^{Z^n}(Z^n_t =y_n)\).
\end{enumerate}

The case (A) corresponds to the case where the state space of each \(Z^n\) is the continuum space \(\mathbb{R}^d\), while the case (B) corresponds to the case where the state spaces of \(Z^n\) are discrete subsets of \(\mathbb{R}^d\). Example \ref{example1} corresponds to the case of (A) and Example \ref{example2} to the case of (B). We will use the situations discribed in (A) and (B) to construct Gaussian fields corresponding to \(Z^n\) which are the analogies of Gaussian free fields.

\begin{remark}
In case (B), the processes \(Z^n\) are essentially processes on the discrete sets \(D_n\). We have simply formulated  our setting so that it is easier to consider convergence to a process on a continuum space. In this setting, we will be able to treat discrete Gaussian fields and Gaussian multiplicative chaoses on \(D_n\) as Gaussian fields and Gaussian multiplicative chaoses on \(\R^d\), respectively, and so we can consider convergence of them on this larger common space.
\end{remark}

The quantity \(C_n\) captures the scaling of the measure. For example, it holds that \(C_n \int_{\mathbb{R}^d} p^n(t,x,y)dm(y)=1\) for any \(t>0\) and \(x\in \mathbb{R}^d\). By setting \(C_n=1\) and \(x_n=x\) for any \(x\in \R^d\) in the case of (A), for \(\lambda>0\), we define the \(\lambda\)-order Green's kernel \(g_{\lambda}^n\) of \(Z^n\) in both cases (A) and (B) by
\[g_{\lambda}^n(x,y):=C_n\int_0^{\infty} e^{-\lambda t} p^n(t,x_n,y_n)dt\]
for \(x,y \in \R^d\).

\begin{remark}\label{capG}
Since the processes \(Z^n\) are strongly recurrent for each \(n\in \mathbb{N}\), by using \cite[Lemma 4.1.5]{CF}, it holds that \(g_{\lambda}^n(x,x)<\infty\) for any \(n\in \mathbb{N}\) and \(x\in \R^d\). On the other hand, since \(Z^{\infty}\) is not strongly recurrent, but recurrent, \(g_{\lambda}^{\infty}(x,x)=\infty\) for any \(x\in \R^d.\)
\end{remark}

Next, we define a Gaussian field \(X^n\) with covariance kernel \(\pi g_{\lambda}^n\) as follows.
\begin{definition}\label{GFdef}
For \(n\in \mathbb{N}\cup \{\infty\}\), we define \(X^n=\{X^n_f\}_{f\in \mathcal{S}(\R^d)}\) as the Gaussian system on a probability space \((\Omega^{X^n}, \mathcal{M}^{X^{n}}, \mathbb{P}^{X^n})\) satisfying \(\mathbb{E}^{X^n}[X^n_f]=0\) and \[\mathbb{E}^{X^n}[X^n_fX_g^n]=\pi \int_{\R^d} \int_{\R^d} g_{\lambda}^n(x_n,y_n)f(x)g(y)dm(x)dm(y) \] for \(f,g \in \mathcal{S}(\R^d).\)
\end{definition}

\begin{remark}
For \(n\in \mathbb{N}\cup \{\infty\}\), let \((\mathcal{E}^n, \mathcal{F}^n)\) be Dirichlet forms associated with \(Z^n\), then the above fields \(X^n\) can be realized as Gaussian fields associated with Dirichlet forms \((\mathcal{E}^n, \mathcal{F}^n)\). See \cite{FO,O,Ro} for the existence and properties of  Gaussian fields associated with Dirichlet forms. In particular, \(X^{\infty}\) is called (massive) Gaussian free field on \(\mathbb{R}^2\).
\end{remark}

Throughout this paper, we make the following assumption. Condition (1) is a natural one on the convergence of the original processes and Gaussian fields. Condition (2) is an estimate on the Green's functions that will be used to guarantee the uniformly integrability for the corresponding Gaussian multiplicative chaos. The uniformly integrability is often used when considering the convergence of Gaussian multiplicative chaos in a specific situation.

\begin{assumption}\label{assumption}\  \\
\((1)\) The pair \((Z^n, X^n)\) converges to \((Z^{\infty}, X^{\infty})\) as \(n\to \infty\) in the sense that the processes \(Z^n\) under \(\mathbb{P}_{x_n}^{Z^n}\) converge weakly to \(Z^{\infty}\) under \(\mathbb{P}_x^{Z^{\infty}}\) in the Skorokhod space \(D[0,\infty)\) equipped with \(J_1\)-topology for any \(x\in \mathbb{R}^d\), and the covariance kernels \(\pi g_{\lambda}^{n}\) converge pointwise to \(\pi g_{\lambda}^{\infty}\).\\
\((2)\) There exist positive constants \(C^*\) and \(C\) such that, for any \(n \in \mathbb{N}\) and \(x, y \in \R^d\), it holds that \[g_{\lambda}^n(x,y) \leq C^*g_{\lambda}^{\infty}(x,y) +C.\]
\end{assumption}

See Appendix \ref{AppSko} for the definition of Skorokhod's \(J_1\)-topology. We remark that a positive constant \(C\) may change from line to line, but \(C^*\) will be fixed throughout.

By Remark \ref{capG}, for \(n\in \mathbb{N}\), we can define \(X^n\) not only as a random distribution, but also as a random function. In case (A), for each \(n\in \mathbb{N}\), we obtain that \(\lim_{|x-y|\to 0}\mathbb{E}^{X^n}[|X^n(x)-X^n(y)|^2]=0 \) by the continuity of \(p^n\). So, by \cite[Proposition 2.1.12]{GN}, there exists a Borel measurable version of \(X^n\) for \(n\in \mathbb{N}\). In case (B), \(X^n\) is a simple function almost surely for \(n\in \mathbb{N}\). Hence, we can define Gaussian multiplicative chaos (GMC in abbreviation) \(\mu^n\), a random measure on \(\R^d\), by setting
\begin{equation}
d\mu^n(x):= \exp{(\gamma X^n(x)-\frac{\gamma^2}{2}\mathbb{E}^{X^n}[X^n(x)^2])}dm(x), \label{eq:GMCdef}
\end{equation}
for \(n\in \mathbb{N}.\) We remark that \(X^n(x)=X^n(y)\) for \(x,y\in \R^d\) with \(x_n=y_n\).

By Lemma \ref{Green'slemma} below, there exists a positive constant \(C\) such that 
\[\pi g_{\lambda}^{\infty}(x,y) \leq \log_+{\frac{1}{|x-y|}}+C\]
for any \(x,y \in \R^d,\) where \(\log_+{\frac{1}{|x-y|}}:=\max{\{\log{\frac{1}{|x-y|}}, 0\}}\). So we can define GMC \(\mu^{\infty}\) of \(X^{\infty}\) for \(\gamma<\sqrt{2d}\) by applying results of \cite{K}, \cite{B} or \cite{S}. Formally, \(\mu^{\infty}\) can be represented as the replacement of \(n\) by \(\infty\) in \((\ref{eq:GMCdef})\). For \(d=2\), this GMC \(\mu^{\infty}\) is also called \textit{Liouville measure} on \(\mathbb{R}^2\).

Since \(\mu^n\) is a smooth measure of \(Z^n\) for \(n\in \mathbb{N}\), there exists a positive continuous additive functional (PCAF in abbreviation) \(A^n=\{A_t^n\}_t\) in the strict sense such that \(A^n\) corresponds to \(\mu^n.\) Details on smooth measures and PCAFs and what it means that a PCAF in the strict sense are presented in Appendix \ref{AppPCAF}. Since, for \(n\in \mathbb{N}\), \(\mu^n\) is absolutely continuous with respect to \(m\), we have the representation
\begin{equation}
A_t^n=\int_0^{t} \exp{(\gamma X^n(Z^n_s)-\frac{\gamma^2}{2}\mathbb{E}^{X^n}[X^n(Z^n_s)^2])}ds, \label{eq:PCAFdef}
\end{equation}
for \(t\geq 0\) and \(n\in \mathbb{N}.\)

Furthermore, by \cite{GRV} and \cite[Appendix A]{AK} for the case \(d=2\) and by \cite{Bav} or similarly to \cite{GRV, AK} for the case \(d=1\), GMC \(\mu^{\infty}\) is a smooth measure on \(\R^d\) and there exist \(A^{\infty}\) and \(\Lambda \in \mathcal{M}^{X^{\infty}}\otimes \mathcal{F}_{\infty}^{\infty}\) such that the following conditions hold.
\begin{enumerate}
  \item For \(\mathbb{P}^{X^{\infty}}\)-almost every \(\omega \in \Omega^{X^{\infty}}\), \(\mathbb{P}_x^{Z^{\infty}}(\Lambda^{\omega})=1\) holds for any \(x\in \R^d\), where \(\Lambda^{\omega}:=\{\omega'\in \Omega^{Z^{\infty}} : (\omega,\omega')\in \Lambda\}.\)
  \item For \(\mathbb{P}^{X^\infty}\)-almost every \(\omega \in \Omega^{X^{\infty}}\), \(A^{\infty}(\omega, \cdot)\) is a PCAF in the strict sense with defining set \(\Lambda^{\omega}\) such that \(A^{\infty}(\omega,\cdot)\) corresponds to \(\mu^{\infty}(\omega)\). 
\end{enumerate}

Here, \(\{\mathcal{F}_{t}^{\infty}\}\) is a filtration to which \(Z^{\infty}\) is adapted. Formally, \(A^{\infty}_t\) can be represented as the replacement of \(n\) by \(\infty\) in \((\ref{eq:PCAFdef})\).

We define the time-changed process \(\hat{Z}^n_t:=Z^n_{(A^n)^{-1}_t}\), where
\[(A^n)^{-1}_t:=\inf\{s>0: A_s^n>t\}\]
for \(n\in \mathbb{N}\cup\{\infty\}.\)
We say that \(\hat{Z}^n=\{\hat{Z}^n_t\}_t\) is the time-changed process of \(Z^n\) by \(\mu^n\). By \cite[Theorem 5.2.13]{CF} this process \(\hat{Z}^n\) is also a Hunt process. For \(d=2\), \(\hat{Z}^{\infty}\) is called \textit{Liouville Brownian motion} on \(\R^2\), and, for \(d=1\), we call \(\hat{Z}^{\infty}\) the \textit{Liouville Cauchy process} on \(\R\).

The main theorem of this paper is stated as follows:

\begin{theorem}\label{mainconvergence}
Suppose Assumption \ref{assumption} holds. Then, for any \(x\in \R^d\) and \(\gamma \in (0,\sqrt{d/C^*})\), as \(n\to \infty\), the law of the time-changed process \(\hat{Z}^n\) under \(\mathbb{P}_{x_n}^{Z^n} \otimes \mathbb{P}^{X^n}\) converges to the law of \(\hat{Z}^{\infty}\) under \(\mathbb{P}_{x}^{Z^{\infty}} \otimes \mathbb{P}^{X^{\infty}}\) as probability measures on the Skorokhod space \(D[0,\infty)\) equipped with Skorokhod's \(J_1\)-topology.
\end{theorem}

\begin{remark}
When the limit function is continuous, convergence with Skorokhod's \(J_1\)-topology implies convergence with local uniform topology, so, for \(d=2\), Theorem \ref{mainconvergence} means the weak convergence of \(\hat{Z}^n\) to Liouville Brownian motion with local uniform topology.
\end{remark}

\begin{remark}
 Since the Green's function of the Cauchy process on \(\R\) satisfies the same estimates as that of Brownian motion on \(\R^2\), it is natural that we consider the Liouville Cauchy process on \(\R\) when \(d=1\), not Liouville Brownian motion on \(\R\). Moreover, due to this, Gaussian fields \(X^{\infty}\) and Gaussian multiplicative chaos \(\mu^{\infty}\) corresponding to the Cauchy process on \(\R\) have similar properties as the 2-dimensional Gaussian free field and Liouville measure. See Example \ref{example1} for details.

The Liouville-Cauchy process on the unit circle \(\mathbb{S}^1\) is constructed in \cite{Bav} in the same way as \cite{GRV}. In \cite{Bav}, the Liouville-Cauchy process on \(\mathbb{S}^1\) is used to consider the trace process of Liouville Brownian motion on the unit ball \(\mathbb{D}^2\), the analogous to Spitzer's embedding theorem \cite{Spi}. We recall that, by Spitzer's embedding theorem, the Cauchy process on \(\R\) (resp. \(\mathbb{S}^1\)) is the trace process of Brownian motion on \(\R^2\) (resp. \(\mathbb{D}^2\)). Thus, there is no essential difference between the Liouville Cauchy process on \(\R\) in this paper and the Liouville-Cauchy process on \(\mathbb{S}^1\) in \cite{Bav}.
\end{remark}

\section{Proof of the main result}\label{mainproof}
In this section, we prove Theorem \ref{mainconvergence} with the following strategy. First, we prove convergence of GMC \(\mu^n\) to \(\mu^{\infty}\) in Subsection \ref{sec:3-2}. Convergence of GMC will be used to prove corresponding PCAFs. By considering continuities under Skorokhod's topologies(see Subsection \ref{seccjd} for details), to prove convergence of time-changed process of \(Z^n\) by \(A^n\), in our case, it is enough to show that \((Z^n, A^n)\) converges weakly to \((Z^{\infty}, A^{\infty})\) with respect to \(J_1\times U\)-topology, where \(U\)-topology is the local uniform topology. It is known that the weak convergence is equivalent to the tightness and convergence in finite-dimensional distribution, so we prove the tightness of the collection of the distributions of \((Z^n, A^n)\) in Subsection \ref{sectight} and the finite-dimensional distributional convergence of \((Z^n, A^n)\) to \((Z^{\infty}, A^{\infty})\) in Subsection \ref{seccfdd}.

Throughout this paper, a positive constant \(C\) may change line to line, but the \(C^*\) that appeared in Assumption \ref{assumption} (2) does not change.
\subsection{Convergence of GMC}\label{sec:3-2}
We defined \(g_{\lambda}^{\infty}(x,y)\) to be the \(\lambda\)-order Green's functions of \(Z^{\infty}.\) These \(\lambda\)-order Green's functions have logarithmically divergence as follows.
\begin{lemma}\label{Green'slemma}
Let \(d=1\) or \(d=2.\) There exists a bounded continuous function \(h\) such that, for any \(x,y\in \R^d\),
\[\pi g_{\lambda}^{\infty}(x,y)=\log_+{\frac{1}{|x-y|}}+h(|x-y|).\]
\end{lemma}
\begin{proof}
For \(d=2,\) we have
\begin{eqnarray*}
\lefteqn{g_{\lambda}^{\infty}(x,y)}\\
&=&\int_0^{\infty} e^{-\lambda t} \frac{e^{-|x-y|^2/2t}}{2\pi t} dt\\
&=&\int_0^{\infty} e^{-\lambda t} \frac{e^{-|x-y|^2/2t}}{2\pi t} {\bf 1}_{\{|x-y|\leq 1\}}dt + \int_0^{\infty} e^{-\lambda t} \frac{e^{-|x-y|^2/2t}}{2\pi t} {\bf 1}_{\{|x-y|> 1\}}dt\\
&=&\int_0^{\infty} \frac{e^{-|x-y|^2/2t}-e^{-1/2t}}{2\pi t} {\bf 1}_{\{|x-y|\leq 1\}} dt +\int_0^{\infty}e^{-\lambda t} \frac{e^{-1/2t}}{2\pi t}{\bf 1}_{\{|x-y|\leq 1\}} dt \\
&&+\int_0^{\infty}(e^{-\lambda t}-1)\frac{e^{-|x-y|^2/2t}-e^{-1/2t}}{2\pi t} {\bf 1}_{\{|x-y|\leq 1\}}dt\\
&& \hspace{10mm} + \int_0^{\infty} e^{-\lambda t} \frac{e^{-|x-y|^2/2t}}{2\pi t} {\bf 1}_{\{|x-y|> 1\}}dt\\
&=& \frac{{\bf 1}_{\{|x-y|\leq 1\}}}{2\pi} \int_0^{\infty} \int_{\frac{|x-y|^2}{2t}}^{\frac{1}{2t}}\frac{1}{t} e^{-s}ds dt +\frac{h(|x-y|)}{\pi}\\
&=& \frac{{\bf 1}_{\{|x-y|\leq 1\}}}{2\pi} \int_0^{\infty} \int_{\frac{|x-y|^2}{2s}}^{\frac{1}{2s}} \frac{1}{t} e^{-s}dtds +\frac{h(|x-y|)}{\pi}\\
&=& \frac{1}{\pi} \log_+{\frac{1}{|x-y|}} +\frac{h(|x-y|)}{\pi}.
\end{eqnarray*}
Here, we defined a bounded continuous function \(h\) by
\begin{eqnarray*}
h(w)&:=&\pi \int_0^{\infty}e^{-\lambda t} \frac{e^{-1/2t}}{2\pi t} {\bf 1}_{\{|w|\leq 1\}}dt \\
&&+\pi \int_0^{\infty}(e^{-\lambda t}-1)\frac{e^{-|w|^2/2t}-e^{-1/2t}}{2\pi t} {\bf 1}_{\{|w|\leq 1\}}dt\\
&& \hspace{5mm} + \pi \int_0^{\infty} e^{-\lambda t} \frac{e^{-|w|^2/2t}}{2\pi t} {\bf 1}_{\{|w|> 1\}}dt.
\end{eqnarray*}
By using \(\lim_{t\to 0} (e^{-\lambda t}-1)/t = -\lambda\), we can check that \(h\) is bounded.

For \(d=1,\) we have
\begin{eqnarray*}
\lefteqn{g_{\lambda}^{\infty}(x,y)}\\
&= &\int_0^{\infty} \frac{e^{-\lambda t}}{\pi} \frac{t}{|x-y|^2+t^2} dt\\
&= &\int_0^{\infty} \frac{e^{-\lambda |x-y|t}}{\pi} \frac{t}{1+t^2} dt\\
&= &\int_0^{\infty} \frac{e^{-\lambda |x-y|t}-e^{-\lambda t}}{\pi} \frac{{\bf 1}_{\{|x-y|\leq 1\}}}{t} dt+ \int_0^{\infty} \frac{e^{-\lambda t}}{\pi} \frac{t}{1+t^2}{\bf 1}_{\{|x-y|\leq 1\}}  dt\\
&&+\int_0^{\infty} \frac{e^{-\lambda |x-y|t}-e^{-\lambda t}}{\pi} \left(\frac{t}{1+t^2}-\frac{1}{t}\right) {\bf 1}_{\{|x-y|\leq 1\}} dt\\
&& \hspace{10mm} + \int_0^{\infty} \frac{e^{-\lambda |x-y|t}}{\pi} \frac{t}{1+t^2} {\bf 1}_{\{|x-y|> 1\}}dt \\
&= &\int_0^{\infty} \int_{\lambda |x-y|t}^{\lambda t} \frac{e^{-s}}{\pi} \frac{1}{t}{\bf 1}_{\{|x-y|\leq 1\}} dsdt+ \frac{h(|x-y|)}{\pi}\\
&= &\int_0^{\infty} \int^{\frac{s}{\lambda |x-y|}}_{\frac{s}{\lambda }} \frac{e^{-s}}{\pi} \frac{1}{t}{\bf 1}_{\{|x-y|\leq 1\}}dt ds+ \frac{h(|x-y|)}{\pi}\\
&=& \frac{1}{\pi} \log_+{\frac{1}{|x-y|}}+\frac{h(|x-y|)}{\pi}.
\end{eqnarray*}
Here, we defined a bounded continuous function \(h\) by
\begin{eqnarray*}
h(w)&:=&\pi \int_0^{\infty} \frac{e^{-\lambda t}}{\pi} \frac{t}{1+t^2} {\bf 1}_{\{|w|\leq 1\}} dt\\
&&+\pi \int_0^{\infty} \frac{e^{-\lambda |w|t}-e^{-\lambda t}}{\pi} \left(\frac{t}{1+t^2}-\frac{1}{t}\right) {\bf 1}_{\{|w|\leq 1\}} dt\\
&&\hspace{5mm} +\pi  \int_0^{\infty} \frac{e^{-\lambda |w|t}}{\pi} \frac{t}{1+t^2} {\bf 1}_{\{|w|> 1\}}dt.
\end{eqnarray*}
By using \(\lim_{t\to 0} (e^{-\lambda t}-1)/t = -\lambda\), we can check that \(h\) is bounded.
\end{proof}

In this subsection, we prove convergence of Gaussian multiplicative chaos \(\mu^n\). By Lemma \ref{Green'slemma}, Gaussian multiplicative chaos \(\mu^{\infty}\) cannot be represented by the form (\ref{eq:I1}) rigorously, and so we need to use Shamov's general result \cite{S} concerning convergence of Gaussian multiplicative chaos in order to prove convergence of \(\mu^n\) to \(\mu^{\infty}\).

First, we will construct \(X^n\) for any \(n\in \mathbb{N}\cup \{\infty\}\) on a common probability space \((\Omega^X, \mathscr{F}^X, \mathbb{P}^X)\). Let \(\mathcal{H}\) be \(L^2([0,\infty)\times \R^d ; dt\otimes dm)\) and \(L^0(\R^d;dm(x))\) be the collection of all measurable functions on \(\R^d\) which is equipped with a topology of the convergence in measure \(m\) locally. For each \(n\in \mathbb{N}\cup\{\infty\}\), we define continuous linear operators \(Y^{n}:\mathcal{H}\to L^0(\R^d;dm(x))\) by setting
\[Y^n(\xi)(x):=\pi^{1/2} \int_0^{\infty}\int_{\R^d}C_n(z)e^{-\lambda t/2}p^{n}(t/2, x_n, z_n)\xi(t,z)dm(z)dt\]
for \(\xi \in \mathcal{H}\) and \(x \in \R^d.\) Indeed, \(Y^n:\mathcal{H}\cap L^{\infty}([0,\infty)\times \R^d)\to L^0(\R^d;dm(x))\) is well-defined since \(|Y^n(\xi)(x)|\leq 2\pi^{1/2}\|\xi\|_{\infty}/\lambda \) for any \(\xi \in \mathcal{H} \cap L^{\infty}([0,\infty)\times \R^d)\), and we can construct the continuous linear operator \(Y^{n}:\mathcal{H}\to L^0(\R^d;dm(x))\) because it holds that
\begin{eqnarray}
\nonumber \lefteqn{ \int_{\R^d}|Y^n(\xi)-Y^n(\eta)|^2(x)dm(x)}\\
\nonumber &=& \pi \int_{\R^d} \left( \int_0^{\infty}\int_{\R^d}C_n e^{-\lambda t/2}p^n(t/2,x_n,z_n)(\xi-\eta)(t,z)dm(z)dt  \right)^2  dm(x)\\
\nonumber &  \leq &  \pi \int_{\R^d}\left \{ \left( \int_0^{\infty} \int_{\R^d}C_n e^{-\lambda t/2} p^n(t/2,x_n,z_n)|\xi-\eta|^2(t,z)dm(z)dt\right)\right.\\
\nonumber &&\left. \hspace{10mm} \times \left(\int_0^{\infty} \int_{\R^d} C_n e^{-\lambda t/2} p^n(t/2,x_n,z_n)dm(z)dt\right)\right \}dm(x)  \\
\nonumber & = &\frac{2\pi }{\lambda} \int_{\R^d} \int_0^{\infty} \int_{\R^d}C_n e^{-\lambda t/2} p^n(t/2,x_n,z_n)|\xi-\eta|^2(t,z)dm(z)dtdm(x)\\
\nonumber &=& \frac{2\pi }{\lambda} \int_0^{\infty} \int_{\R^d} e^{-\lambda t/2} |\xi-\eta|^2(t,z)dm(z)dt\\
& \leq & \frac{2\pi }{\lambda} \|\xi-\eta \|_{\mathcal{H}}^2 \label{eq:GMCconst1}
\end{eqnarray}
for any \(\xi, \eta\in \mathcal{H}\). In the first inequality above, we used the Cauchy-Schwarz inequality.

We remark that \(Y^n(x):=\pi^{1/2} C_n(z)e^{-\lambda t/2}p^n(t/2,x_n,z_n)\) belongs to \(\mathcal{H}\) for \(n \in \mathbb{N}\), and it holds that \(Y^n(\xi)(x)=\langle Y^n(x), \xi \rangle_{\mathcal{H}}\) for \(n\in \mathbb{N}\). However, \(Y^{\infty}(x):=\pi^{1/2} C_n(z)e^{-\lambda t/2}p^{\infty}(t/2,x,z)\) does not belongs to \(\mathcal{H}\), so the above constructions of operators \(Y^{n}\) were necessary.

We can construct a Gaussian field \(X\) indexed by elements in \(\mathcal{H}\) on a probability space \((\Omega^X, \mathscr{F}^X, \mathbb{P}^X)\) by \cite[Theorem 12.1.4]{D}. More precisely, for any \(k\in \mathbb{N}\) and \(\xi_1, \cdots, \xi_k \in \mathcal{H}\), the random vector \((X_{\xi_1},\cdots , X_{\xi_k})\) has a multivariate Gaussian distribution on \((\Omega^X, \mathscr{F}^X, \mathbb{P}^X)\) satisfying \(\EX [X_{\xi_i}]=0\) and \(\EX [X_{\xi_i}X_{\xi_j}]=\langle \xi_i, \xi_j \rangle_{\mathcal{H}} \) for \(1\leq i, j\leq k.\) For \(n\in \mathbb{N}\cup\{\infty\}\), \(\xi \in \mathcal{H}\) and \(f\in L^2(\R^d;dm)\), the Cauchy-Schwarz inequality and (\ref{eq:GMCconst1}) yield that
\[\int_{\R^d}|Y^n(\xi)(x)f(x)|dm(x) \leq \sqrt{\frac{2\pi}{\lambda}} \|\xi\|_{\mathcal{H}}\cdot \|f\|_{L^2}.\]
So the adjoint operator 
\[(Y^n)^*:L^2(\R^d;m)\ni f \mapsto \pi^{1/2}\int_{\R^d}e^{-\lambda t/2}C_n(y)p^n(t/2,x_n,y_n)f(x)dm(x) \in \mathcal{H}\] is continuous. Moreover, for \(f,g\in L^2(\R^d;dm)\), it holds that
\[\left\langle (Y^n)^*f, \xi \right\rangle_{\mathcal{H}}=\int_{\R^d}Y^n(\xi)(x)f(x)dm(x)\]
and  
\[\left\langle (Y^n)^*f, (Y^n)^*g\right\rangle_{\mathcal{H}}=\int_{\R^d}\int_{\R^d}\langle Y^n(x), Y^n(y)\rangle f(x)g(y)dm(x)dm(y),\]
where \(\langle Y^n(x), Y^n(y) \rangle =\pi g_{\lambda}^n(x,y)\). This means that \(X_{(Y^n)^*}\) can be identified with \(X^n\) in distribution, so we have realized \(X^n\) for any \(n\in \mathbb{N}\cup \{\infty\}\) on the same probability space \(\Omega^X\). By \cite[Theorem 17]{S}, GMC \(\mu^n\) is realized on \(\Omega^X\) for all \(n\in \mathbb{N}\cup\{\infty\}.\) In this paper, we use the above realizations of \(X^n\) and \(\mu^n\) without loss of generality since the main result is stated with respect to the weak convergence.

By checking convergence of \(Y^n\) to \(Y^{\infty}\) in the sense of \cite{S}, the following theorem concerning convergence of GMC holds.

\begin{theorem}\label{Shamov}
For any \(\gamma \in (0,\sqrt{2d/C^*}),\) GMC \(\mu^n\) converges vaguely to \(\mu^{\infty}\) in \(L^1(\mathbb{P}^X)\) as \(n \to \infty\) in the sense that, for \(f \in C_c(\R^d),\) 
\[\lim_{n \to \infty}\EX\left|\int_{\R^d}fd\mu^n-\int_{\R^d}fd\mu^{\infty}\right|=0. \]
\end{theorem}
\begin{proof}
We check the three assumptions of \cite[Theorem 25]{S} to apply to this theorem. For any \(f\in C_c(\R^d)\), let \(D\) be the support of \(f\). 

Firstly, by Kahane's convexity inequality (see \cite[Theorem 28]{S} or \cite{K}) and Assumption \ref{assumption} (2), \(\{\mu^n(D)\}_{n}\) is uniformly integrable.

Secondly, we will prove that \(Y^n\) converges to \(Y^{\infty}\) as \(n \to \infty\) in the sense that \(Y^n(\xi)\) converges to \(Y^{\infty}(\xi)\) in measure \(m\) for any \(\xi \in \mathcal{H}.\) For any \(\xi \in \mathcal{H}\) and \(\varepsilon>0\), if there exists \(\xi_{\varepsilon} \in \mathcal{H}\) such that \(\|\xi -\xi_{\varepsilon}\|_{\mathcal{H}}<\varepsilon\), then \((\ref{eq:GMCconst1})\) yields that
\begin{eqnarray}
\nonumber \lefteqn{ m(\{x\in \R^d : |Y^n(\xi)(x)-Y^{\infty}(\xi)(x)|>\delta\})}\\
\nonumber &\leq & m(\{x\in \R^d : |Y^n(\xi)(x)-Y^{n}(\xi_{\varepsilon})(x)|>\delta/3\})\\
\nonumber &&\hspace{15mm}+m(\{x\in \R^d : |Y^n(\xi_{\varepsilon})(x)-Y^{\infty}(\xi_{\varepsilon})(x)|>\delta/3\})\\
\nonumber &&\hspace{25mm}+m(\{x\in \R^d : |Y^{\infty}(\xi_{\varepsilon})(x)-Y^{\infty}(\xi)(x)|>\delta/3\})\\
&\leq & \frac{3}{\delta}2\sqrt{\frac{2\pi}{\lambda}}\varepsilon + m(\{x\in \R^d : |Y^n(\xi_{\varepsilon})(x)-Y^{\infty}(\xi_{\varepsilon})(x)|>\delta/3\}) \label{eq:Shamov1}
\end{eqnarray}
for any \(\delta>0.\) By \cite[Theorem 3.3.14]{R}, \(C_c([0,\infty)\times \R^d)\) is dense in \(\mathcal{H}\), so we may assume that \(\xi \in C_c([0,\infty)\times \R^d)\). Then we have
\begin{eqnarray}
\nonumber Y^n(\xi)(x)&=&\pi^{1/2} \int_0^{\infty}\int_{\R^d}C_n(z)e^{-\lambda t/2}p^{n}(t/2, x_n, z_n)\xi(t,z)dm(z)dt\\
\nonumber &=& \pi^{1/2} \int_0^{\infty}\int_{\R^d}C_n(z)e^{-\lambda t/2}p^{n}(t/2, x_n, z_n)(\xi(t,z)-\xi(t,z_n))dm(z)dt\\
&&+\pi^{1/2} \int_0^{\infty}e^{-\lambda t/2}\mathbb{E}_{x_n}^{Z^n}[\xi(t, Z^n_{t/2})]dt\label{eq:Shamov2}
\end{eqnarray}
By uniform continuity of \(\xi\), the first term of \((\ref{eq:Shamov2})\) converges to \(0\) as \(n\to \infty\). By Lebesgue's convergence theorem and Assumption \ref{assumption} (1), the second term of \((\ref{eq:Shamov2})\) converges to \(Y^{\infty}(\xi)(x)\) for any \(x\), and so does in measure \(m\).

Finally, by Assumption \ref{assumption}, the kernel \(\langle Y^n(x),Y^n(y) \rangle  =\pi g_{\lambda}^n(x,y)\) converges to \(\langle Y^{\infty}(x),Y^{\infty}(y) \rangle= \pi g_{\lambda}^{\infty}(x,y)\) pointwisely, and so does in measure \(dm \times dm\) on \(D\times D.\)

Thus we can apply \cite[Theorem 25]{S} to our case, which completes the proof. 
\end{proof}

By Theorem \ref{Shamov}, \(\int fd\mu^n\) converges in law to \(\int fd\mu^{\infty}\) for any \(f\in C_c(\R^d)\). Hence, by \cite[Theorem 23.16]{Ka}, the following statement holds.
\begin{corollary}
For \(\gamma \in (0,\sqrt{2d/C^*}),\) GMC \(\mu^n\) under \(\mathbb{P}^{X^n}\) converges in law to \(\mu^{\infty}\) under \(\mathbb{P}^{X^{\infty}}\) in the space of Radon measures with vague topology as \(n\to \infty.\)
\end{corollary}

\subsection{Tightness}\label{sectight}
The goal of this subsection is to prove the following theorem concerning the tightness of the collection of the distributions of \((Z^n, A^n)\). We remind that \(U\)-topology means the local uniform topology. Throughout of this subsection, we take  \(\gamma \in (0,\sqrt{d/C^*})\).
\begin{theorem}\label{maintight}
The collection of the distributions of \((Z^n, A^n)\) under \(\mathbb{P}_{x_n}^{Z^n}\otimes \mathbb{P}^{X^n}\) is tight with respect to \(J_1\times U\)-topology for any starting point \(x\in \R^d\).
\end{theorem}
This theorem follows from the following and Assumption \ref{assumption} (1) immediately.
\begin{theorem}\label{PCAFconvergence}
\((1)\) \(A^n\) converges weakly to \(A^{\infty}\) with respect to \(U\)-topology.\\
\((2)\) \((A^n)^{-1}\) converges weakly to \((A^{\infty})^{-1}\) with respect to \(U\)-topology.
\end{theorem}

So, we prove Theorem \ref{PCAFconvergence} in this subsection. Before beginning the proof, we state the following corollary.
\begin{corollary}\label{cor4.3}
For any \(x\in \R^d\), the collection of distributions of time-changed processes \(\{\hat{Z}^n\}_{n}\) under \(\mathbb{P}_{x_n}^{Z^n} \otimes \mathbb{P}^{X^n}\) is tight with respect to \(J_1\)-topology.
\end{corollary}
\begin{proof}
Since \((A^n)^{-1}\) is continuous \(\mathbb{P}_{x_n}^{Z^n}\otimes \mathbb{P}^{X^n}\)-almost surely, and the family of continuous functions is the closed set of the family of all c\`{a}dl\`{a}g functions with respect to \(J_1\)-topology, Corollary \ref{cor4.3} follows from Proposition \ref{J1UJ1}, Assumption \ref{assumption} (1) and Theorem \ref{PCAFconvergence} (2).
\end{proof}

We fix \(x\in \R^d\) and \(T>0\). For \(n\in \mathbb{N}\), we define the PCAF \(A^{n,\infty}\) corresponding to \(\mu^n\) with respect to the Dirichlet form of \(Z^{\infty}\) by 
\[A^{n,\infty}_T:= \int_0^T e^{\gamma X^n(Z_s^{\infty})-\frac{\gamma^2}{2} \E^{X^n}[X^n(Z^{\infty}_s)^2]}ds. \] 

To prove Theorem \ref{PCAFconvergence}, we will prove convergence of \(A^{n, \infty}\) to \(A^{\infty}\) and \(|A^{n, \infty}-A^n|\) to \(0\). We need the following lemma.
\begin{lemma} \label{selfsimlemma}
For any \(t,s \geq 0\), \(d=1,2\), \(x\in \R^d\) and \(r \in (0,d)\), it holds that
\begin{eqnarray*}
\E^{Z^{\infty}}_x \left[\frac{1}{|Z_t^{\infty}-Z_s^{\infty}|^{r}}\right] \leq \frac{C}{|t-s|^{r/d}}.
\end{eqnarray*}
\end{lemma}
\begin{proof}
By the stationary increments and the self-similarity of \(Z^{\infty}\), we have
\begin{eqnarray*}
\E^{Z^{\infty}}_x \left[ \frac{1}{|Z_t^{\infty}-Z_s^{\infty}|^{r}}\right] =
\E^{Z^{\infty}}_0\left[ \frac{1}{|Z_{t-s}^{\infty}|^{r}}\right] =\frac{1}{|t-s|^{r/d}} \E^{Z^{\infty}}_0 \left[\frac{1}{|Z_{1}^{\infty}|^{r}}\right]
\end{eqnarray*}
and, by calculation using heat kernels, \(\E^{Z^{\infty}}_0\left[|Z_{1}^{\infty}|^{-r}\right]\) is constant for \(r<d\).
\end{proof}

\begin{lemma}\label{PCAFcor}
For any \(T>0,\) \(A^{n, \infty}_T\) converges to some \(\tilde{A}_T\) in \(L^1(\mathbb{P}^X)\) as \(n \to \infty\) \(\mathbb{P}_x^{Z^{\infty}}\)-almost surely.
\end{lemma}
\begin{proof}
We prove this lemma in a similar way to Theorem \ref{Shamov}. Let \(\nu^{\infty}_T\) be an occupation measure of \(Z^{\infty}\) at \(T.\) Then we have
\[A^{n,\infty}_T= \int_{\R^d} e^{\gamma X^n(x)-\frac{\gamma^2}{2} \E^{X^n}[X^n(x)^2]}d\nu^{\infty}_T(x)\]
for \(n\in \mathbb{N}\), so we may consider \(A^{n,\infty}_T\) as GMC of \(X^n\) under \(\nu^{\infty}_T\).

By Lemma \ref{selfsimlemma}, it holds that
\begin{eqnarray*}
\int_0^T \int_0^T \mathbb{E}^{Z^{\infty}}_x \left[\frac{1}{|Z^{\infty}_t-Z^{\infty}_s|^{1/2}}\right]dtds &\leq & \int_0^T \int_0^T \frac{C}{|t-s|^{1/2d}}dtds < \infty,
\end{eqnarray*}
so 
\begin{equation}
\int_0^T \int_0^T \frac{1}{|Z^{\infty}_t(\omega)-Z^{\infty}_s(\omega)|^{1/2}}dtds <\infty.
\label{eq:new1}
\end{equation}
holds \(\mathbb{P}_x^{Z^{\infty}}\)-almost surely. Fix \(\omega \in \Omega^{Z^{\infty}}\) satisfying \((\ref{eq:new1})\).

Let \(\mathcal{H}:=L^2([0,\infty)\times \R^d; dt\otimes dm(w))\). For \(n\in \mathbb{N}\cup \{\infty\}\), we define the linear operator \(B^n:L^2(\R^d;d\nu_T^{\infty})\to \mathcal{H}\) by
\[B^n(f)(t,w):= \int_{\R^d} e^{-\lambda t/2}C_n p^n(t/2,x,w)f(x)d\nu_T^{\infty}(x)\]
for \(f\in L^2(\R^d;d\nu_T^{\infty}).\) Then, for any \(f\in L^2(\R^d;d\nu_T^{\infty}),\) we have
\begin{eqnarray*}
\lefteqn{\|B^n(f)\|_{\mathcal{H}}^2 } \\
&= & \int_0^{\infty}\int_{\R^d}\int_{\R^d}\int_{\R^d} e^{-\lambda t}C_n^2 p^n(t/2,x,w)p^n(t/2,y,w)f(x)f(y)d\nu_T^{\infty}(x)d\nu_T^{\infty}(y)dm(w) dt\\
&= & \int_0^{\infty}\int_{\R^d}\int_{\R^d} e^{-\lambda t}C_n p^n(t,x,y)f(x)f(y)d\nu_T^{\infty}(x)d\nu_T^{\infty}(y) dt\\
&=&  \int_{\R^d}\int_{\R^d} g_{\lambda}^n(x,y)f(x)f(y)d\nu_T^{\infty}(x)d\nu_T^{\infty}(y)\\
&\leq & \left(\int_{\R^d}\int_{\R^d} g_{\lambda}^n(x,y)^2 d\nu_T^{\infty}(x)d\nu_T^{\infty}(y) \right)^{1/2} \times \|f\|_{L^2(d\nu^{\infty}_T)}\\
&\leq & \left(\int_{\R^d}\int_{\R^d} C\left(\log_+{\frac{1}{|x-y|}}\right)^2 d\nu_T^{\infty}(x)d\nu_T^{\infty}(y) +C \right)^{1/2} \times \|f\|_{L^2(d\nu^{\infty}_T)}\\
&\leq & \left(\int_{\R^d}\int_{\R^d} C\frac{1}{|x-y|^{1/2}} d\nu_T^{\infty}(x)d\nu_T^{\infty}(y) +C \right)^{1/2} \times \|f\|_{L^2(d\nu^{\infty}_T)}\\
&\leq & \left(C\int_0^T \int_0^T \frac{1}{|Z^{\infty}_t-Z^{\infty}_s|^{1/2}}dtds +C \right)^{1/2} \times \|f\|_{L^2(d\nu^{\infty}_T)}.
\end{eqnarray*}
In the second equality, we used the Markov property for \(Z^n\) and, in the second inequality, we used Assumption \ref{assumption} (2). Hence, by \((\ref{eq:new1})\), \(B^n\) is a continuous linear operator. So, the adjoint operator \(Y^n:=(B^n)^*:\mathcal{H}\to L^2(\R^d;d\nu_T^{\infty})\) of \(B^n\) is continuous. For \(\xi \in \mathcal{H}\), we can represent \(Y^n(\xi)\) by
\[Y^n(\xi)(x)= \int_0^{\infty}\int_{\R^d}e^{-\lambda t/2}C_n p^n(t/2,x,w)\xi(t,w)dm(w)dt. \]
By Kahane's convexity theorem \cite{K}, \(\{A_T^{n,\infty}(\omega)\}_n\) is uniformly integrable. By the same way as the proof of Theorem \ref{Shamov}, \(Y^n(\xi)\) converges to \(Y^{\infty}(\xi)\) in measure \(\nu_T^{\infty}\) for any \(\xi \in \mathcal{H}\). The covariance kernel \(\pi g_{\lambda}^n\) converges to \(\pi g_{\lambda}^{\infty}\) by Assumption \ref{assumption} (1). Hence, by \cite[Theorem 25]{S}, \(A^{n,\infty}_T(\omega)\) converges to \(\tilde{A}_T^{\infty}(\omega)\) in \(L^1(\mathbb{P}^X)\) for \(\mathbb{P}_x^{Z^{\infty}}\)-almost sure \(\omega\), where \(\tilde{A}_T^{\infty}(\omega)\) is GMC of \(X^{\infty}\) under \(\nu^{\infty}_T\).
\end{proof}

Next, we prove the difference between \(A^n_T\) and \(A^{n,\infty}_T\) converges to 0. By Skorokhod's representation theorem, there exist a probability space \((\Omega^Z, \mathbb{P}_x^Z)\) and a process \(\tilde{Z}^n\) for \(n\in \mathbb{N}\cup \{\infty\}\) such that \(\tilde{Z}^n\) has the same distribution as that of \(Z^n\), and \(\tilde{Z}^n(\omega)\) converges to \(\tilde{Z}^{\infty}(\omega)\) with respect to \(J_1\)-topology as \(n \to \infty\) for any \(\omega \in \Omega^Z.\) We define 
\[\tilde{A}^{n,\infty}_T:= \int_0^T e^{\gamma X^n(\tilde{Z}_s^{\infty})-\frac{\gamma^2}{2} \E^{X^n}[X^n(\tilde{Z}^{\infty}_s)^2]}ds, \ \ \tilde{A}^{n}_T:= \int_0^T e^{\gamma X^n(\tilde{Z}_s^{n})-\frac{\gamma^2}{2} \E^{X^n}[X^n(\tilde{Z}^{n}_s)^2]}ds \] 
for any \(T>0.\)

\begin{lemma}\label{newlemma}
\(|\tilde{A}^{n}_T-\tilde{A}^{n,\infty}_T|\) converges to \(0\) in \(L^2(\mathbb{P}_x^Z \otimes \mathbb{P}^X)\).
\end{lemma}
\begin{proof}
We remark that, since \(\tilde{Z}^n\) converges to \(\tilde{Z}^{\infty}\) with respect to \(J_1\)-topology \(\mathbb{P}_x^Z\)-almost surely, \(\tilde{Z}^n_t\) converges to \(\tilde{Z}^{\infty}_t\) for almost every \(t\), \(\mathbb{P}_x^Z\)-almost surely.

Let \(\tilde{X}^n(x):=\gamma X^n(x)-\gamma^2\mathbb{E}^X[X^n(x)]/2\), then we have
\begin{align}
\nonumber &\mathbb{E}_x^Z\mathbb{E}^X[|\tilde{A}^{n}_T-\tilde{A}^{n,\infty}_T|^2] = \mathbb{E}_x^Z\mathbb{E}^X\left[\left|\int_0^T(e^{\tilde{X}^n(\tilde{Z}^n_s)}-e^{\tilde{X}^n(\tilde{Z}^{\infty}_s)})ds\right|^2\right] \hspace{20mm}\\
\nonumber =& \mathbb{E}_x^Z\left[\int_0^T\int_0^Te^{\gamma ^2 \pi g_{\lambda}^n(\tilde{Z}_s^n,\tilde{Z}_t^n)}+ e^{\gamma ^2 \pi g_{\lambda}^n(\tilde{Z}_s^{\infty},\tilde{Z}_t^{\infty})} \right. -e^{\gamma ^2 \pi g_{\lambda}^n(\tilde{Z}_s^{\infty},\tilde{Z}_t^n)}-e^{\gamma ^2 \pi g_{\lambda}^n(\tilde{Z}_s^n,\tilde{Z}_t^{\infty})} dsdt\Biggr]\\
= &2\mathbb{E}_x^Z\left[\int_0^T\int_0^t e^{\gamma ^2 \pi g_{\lambda}^n(\tilde{Z}_s^n,\tilde{Z}_t^n)}+ e^{\gamma ^2 \pi g_{\lambda}^n(\tilde{Z}_s^{\infty},\tilde{Z}_t^{\infty})}\right. -e^{\gamma ^2 \pi g_{\lambda}^n(\tilde{Z}_s^{\infty},\tilde{Z}_t^n)}-e^{\gamma ^2 \pi g_{\lambda}^n(\tilde{Z}_s^n,\tilde{Z}_t^{\infty})} dsdt\Biggr] \label{eq:new2}
\end{align}
We consider each terms of (\ref{eq:new2}). For the first term of (\ref{eq:new2}), we split into two terms as following. For any \(\eta >0\), we have
\begin{eqnarray*}
\mathbb{E}_x^Z\left[\int_0^T\int_0^t  e^{\gamma ^2 \pi g_{\lambda}^n(\tilde{Z}_s^n,\tilde{Z}_t^n)}dsdt\right]&=& \mathbb{E}_x^Z\left[\int_0^T\int_0^t  e^{\gamma ^2 \pi g_{\lambda}^n(\tilde{Z}_s^n,\tilde{Z}_t^n)}{\bf 1}_{\{|\tilde{Z}_s^n-\tilde{Z}_t^n|> \eta\}} dsdt \right]\\
 &&+ \mathbb{E}_x^Z\left[\int_0^T\int_0^t  e^{\gamma ^2 \pi g_{\lambda}^n(\tilde{Z}_s^n,\tilde{Z}_t^n)} {\bf 1}_{\{|\tilde{Z}_s^n-\tilde{Z}_t^n|\leq \eta\}} dsdt \right].
\end{eqnarray*}
We set \(B(z;\eta):=\{y\in \R^d\ ;\ |y-z|\leq \eta\}\). Then we have
\begin{eqnarray}
\lefteqn{ \mathbb{E}_x^Z\left[\int_0^T\int_0^t  e^{\gamma ^2 \pi g_{\lambda}^n(\tilde{Z}_s^n,\tilde{Z}_t^n)}{\bf 1}_{\{|\tilde{Z}_s^n-\tilde{Z}_t^n|\leq \eta\}} dsdt \right]} \label{eq:dif3}\\
\nonumber  &=& \int_{\mathbb{R}^d} \int_{B(z;\eta)}  \int_0^T \int_0^t e^{\gamma ^2 \pi g_{\lambda}^n(y,z)}C_n^2 p^n(s,x,y)p^n(t-s,y,z) dsdt dm(y)dm(z).
\end{eqnarray}

By some computation, we have
\begin{eqnarray}
\nonumber \lefteqn{\int_0^T \int_0^t e^{\gamma ^2 \pi g_{\lambda}^n(y,z)}C_n^2 p^n(s,x,y)p^n(t-s,y,z) dsdt}\\
\nonumber &=& \int_0^T \int_s^T e^{\gamma ^2 \pi g_{\lambda}^n(y,z)}C_n^2 p^n(s,x,y)p^n(t-s,y,z) dtds \\
\nonumber &\leq & e^{\lambda T}  e^{\gamma ^2 \pi g_{\lambda}^n(y,z)}\int_0^T C_n e^{-\lambda s} p^n(s,x,y)\int_s^T C_n e^{-\lambda(t-s)}p^n(t-s,y,z) dtds \\
\nonumber &\leq & e^{\lambda T}  e^{\gamma ^2 \pi g_{\lambda}^n(y,z)}\int_0^{\infty} C_n e^{-\lambda s} p^n(s,x,y)\int_s^{\infty} C_n e^{-\lambda(t-s)}p^n(t-s,y,z) dtds \\
&= & e^{\lambda T}e^{\gamma ^2 \pi g_{\lambda}^n(y,z)} g_{\lambda}^n(x,y) g_{\lambda}^n(y,z).\label{eq:dif1}
\end{eqnarray}

For small \(\varepsilon_1>0\) satisfying \(\gamma^2 C^* + \varepsilon_1/\pi <d,\) we set \(\varepsilon := d-\gamma^2 C^* - \varepsilon_1/\pi\). Then, by Assumption \ref{assumption} (2) and Lemma \ref{Green'slemma}, we have 
\begin{eqnarray}
\nonumber \lefteqn{e^{\lambda T}e^{\gamma ^2 \pi g_{\lambda}^n(y,z)} g_{\lambda}^n(x,y) g_{\lambda}^n(y,z)\leq e^{\lambda T}e^{\gamma ^2 C^*\pi g_{\lambda}^{\infty}(y,z)} g_{\lambda}^n(x,y) (C^* g_{\lambda}^{\infty}(y,z) +C)}\\
\nonumber &\leq & C^* e^{\lambda T}\frac{1}{\varepsilon_1} e^{(\gamma ^2 C^*\pi +\varepsilon_1) g_{\lambda}^{\infty}(y,z)} g_{\lambda}^n(x,y)  +C e^{\lambda T}e^{\gamma ^2 C^*\pi g_{\lambda}^{\infty}(y,z)}g_{\lambda}^n(x,y)\\
&\leq & C \frac{1}{|y-z|^{d-\varepsilon}} g_{\lambda}^n(x,y). \label{eq:dif2}
\end{eqnarray}
By (\ref{eq:dif1}), (\ref{eq:dif2}) and (\ref{eq:dif3}), we have
\begin{eqnarray}
\nonumber \mathbb{E}_x^Z\left[\int_0^T\int_0^t  e^{\gamma ^2 \pi g_{\lambda}^n(\tilde{Z}_s^n,\tilde{Z}_t^n)}{\bf 1}_{\{|\tilde{Z}_s^n-\tilde{Z}_t^n|\leq \eta\}} dsdt \right]&\leq &\int_{\mathbb{R}^d} \int_{B(z;\eta)}C \frac{g_{\lambda}^n(x,y)}{|y-z|^{d-\varepsilon}}  dm(y)dm(z)\\
\nonumber &=& \int_{\mathbb{R}^d} \int_{B(y;\eta)}C\frac{g_{\lambda}^n(x,y)}{|y-z|^{d-\varepsilon}}  dm(z)dm(y)\\
\nonumber &=&C \int_{\mathbb{R}^d} \int_0^{\eta} \frac{r^{d-1}}{r^{d-\varepsilon}}dr g_{\lambda}^n(x,y) dm(y)\\
&=&C \eta^{\varepsilon}. \label{eq:new2-2}
\end{eqnarray}

By Assumption \ref{assumption} (2) and Lemma \ref{Green'slemma}, we have 
\begin{eqnarray*}
e^{\gamma ^2 \pi g_{\lambda}^n(\tilde{Z}_s^n,\tilde{Z}_t^n)}{\bf 1}_{\{|\tilde{Z}_s^n-\tilde{Z}_t^n|> \eta\}} & \leq & Ce^{\gamma ^2 \pi C^* g_{\lambda}^{\infty}(\tilde{Z}_s^n, \tilde{Z}_t^n)}{\bf 1}_{\{|\tilde{Z}_s^n-\tilde{Z}_t^n|> \eta\}}\\
&\leq & Ce^{-\gamma ^2 C^* \log_{+}{|\tilde{Z}_s^n- \tilde{Z}_t^n|}}{\bf 1}_{\{|\tilde{Z}_s^n-\tilde{Z}_t^n|> \eta\}}\\
&\leq & C\left(\frac{1}{\eta^{\gamma ^2 C^*}} \vee 1 \right),
\end{eqnarray*}
Hence, by taking first the limit as \(n\) tends to \(\infty\) and then the limit as \(\eta\) tends to \(0\), from (\ref{eq:new2}), (\ref{eq:new2-2}), Lebesgue's convergence theorem and Assumption \ref{assumption} (1), it holds that
\begin{eqnarray}
\lim_{n\to \infty} \mathbb{E}_x^Z\left[\int_0^T\int_0^t  e^{\gamma ^2 \pi g_{\lambda}^n(\tilde{Z}_s^n,\tilde{Z}_t^n)}dsdt \right]= \mathbb{E}_x^Z\left[\int_0^T\int_0^t  e^{\gamma ^2 \pi g_{\lambda}^{\infty}(\tilde{Z}_s^{\infty},\tilde{Z}_t^{\infty})} dsdt \right].\label{eq:new2-3}
\end{eqnarray}

Next, we consider the second term of (\ref{eq:new2}). It holds that
\begin{eqnarray}
\nonumber \lefteqn{\mathbb{E}_x^Z\left[\int_0^T\int_0^t  e^{\gamma ^2 \pi g_{\lambda}^n(\tilde{Z}_s^{\infty},\tilde{Z}_t^{\infty})}dsdt\right]}\\
&=& \int_0^T\int_0^t \int_{\R^d}\int_{\R^d}e^{\gamma ^2 \pi g_{\lambda}^n(y,z)}p^{\infty}(s,x,y)p^{\infty}(t-s,y,z)dsdt dm(y)dm(z) \label{eq:new2-4}
\end{eqnarray}
By Assumption \ref{assumption} (2) and Lemma \ref{Green'slemma}, it holds that
\begin{eqnarray*}
e^{\gamma ^2 \pi g_{\lambda}^n(y,z)}p^{\infty}(s,x,y)p^{\infty}(t-s,y,z) \leq C\left( \frac{1}{|y-z|^{\gamma^2 C^*}} \vee 1 \right) p^{\infty}(s,x,y)p^{\infty}(t-s,y,z)
\end{eqnarray*}
and, by Lemma \ref{selfsimlemma}, we have
\begin{eqnarray*}
\int_0^T \int_0^t \int_{\R^d} \int_{\R^d} C\left( \frac{1}{|y-z|^{\gamma^2 C^*}} \vee 1 \right) p^{\infty}(s,x,y)p^{\infty}(t-s,y,z) dm(y)dm(z)dsdt <\infty.
\end{eqnarray*}
So, by Lebesgue's convergence theorem, we have
\begin{eqnarray}
\lim_{n\to \infty} \mathbb{E}_x^Z\left[\int_0^T\int_0^t  e^{\gamma ^2 \pi g_{\lambda}^n(\tilde{Z}_s^{\infty},\tilde{Z}_t^{\infty})}dsdt\right]= \mathbb{E}_x^Z\left[\int_0^T\int_0^t  e^{\gamma ^2 \pi g_{\lambda}^{\infty}(\tilde{Z}_s^{\infty},\tilde{Z}_t^{\infty})}dsdt\right]. \label{eq:new2-5}
\end{eqnarray}

For the third term of (\ref{eq:new2}), by Fatou's lemma, we have
\begin{eqnarray}
\nonumber \varlimsup_{n\to \infty} \int_0^T\int_0^t -\mathbb{E}_x^Z[e^{\gamma ^2 \pi g_{\lambda}^n(\tilde{Z}_s^{\infty},\tilde{Z}_t^n)}] dsdt &=& -\varliminf_{n\to \infty} \int_0^T\int_0^t \mathbb{E}_x^Z[e^{\gamma ^2 \pi g_{\lambda}^n(\tilde{Z}_s^{\infty},\tilde{Z}_t^n)}] dsdt\\
\nonumber &\leq &- \int_0^T\int_0^t \mathbb{E}_x^Z[\varliminf_{n\to \infty}e^{\gamma ^2 \pi g_{\lambda}^n(\tilde{Z}_s^{\infty},\tilde{Z}_t^n)}] dsdt\\
&=& - \int_0^T\int_0^t \mathbb{E}_x^Z[e^{\gamma ^2 \pi g_{\lambda}^{\infty}(\tilde{Z}_s^{\infty},\tilde{Z}_t^{\infty})}] dsdt. \label{eq:new2-6}
\end{eqnarray}
By the same way as above, for the fourth term of (\ref{eq:new2}), we have
\begin{eqnarray}
\varlimsup_{n\to \infty} \int_0^T\int_0^t -\mathbb{E}_x^Z[e^{\gamma ^2 \pi g_{\lambda}^n(\tilde{Z}_s^n,\tilde{Z}_t^{\infty})}] dsdt \leq  -  \int_0^T\int_0^t \mathbb{E}_x^Z[e^{\gamma ^2 \pi g_{\lambda}^{\infty}(\tilde{Z}_s^{\infty},\tilde{Z}_t^{\infty})}] dsdt. \label{eq:new2-7}
\end{eqnarray}
By (\ref{eq:new2}), (\ref{eq:new2-3}), (\ref{eq:new2-5}), (\ref{eq:new2-6}) and (\ref{eq:new2-7}),
 \(\mathbb{E}_x^Z \mathbb{E}^X[|\tilde{A}^{n}_T-\tilde{A}^{n,\infty}_T|^2]\) converges to \(0\).
\end{proof}

\begin{proposition}\label{PCAF_Tconv}
For any \(T>0\) and \(x\in \R^d\), \(A_T^n\) under \(\mathbb{P}_{x_n}^{Z^n}\otimes \mathbb{P}^{X^n}\) converges weakly to \(\tilde{A}_T\) under \(\mathbb{P}_{x}^{Z^{\infty}}\otimes \mathbb{P}^{X^{\infty}}\).
\end{proposition}
\begin{proof}
This follows from Lemma \ref{PCAFcor}, \ref{newlemma} and the fact that \(Z^n\) has the same distribution as that of \(\tilde{Z}^n.\)
\end{proof}

Next, we check the criteria of Theorem \ref{M1cri2} to prove the tightness of \(\{(A^n)^{-1}\}_{n}\) with respect to \(M_1\)-topology. See Appendix \ref{AppSko} for Skorokhod's topologies.

\begin{lemma}\label{lemma7}
For any \(T, \eta>0\) and \(x\in \R^d\), it holds that
\[\lim_{\delta \searrow 0}\varlimsup_{n \to \infty}\mathbb{P}_{x_n}^{Z^n}\otimes \mathbb{P}^{X^n}\left(v_T((A_T^{n})^{-1},0,\delta) \geq \eta\right)=0,\]
where \(v_T\) is defined at \((\ref{eq:v1})\) in Appendix \ref{AppSko}.
\end{lemma}
\begin{proof}
For any \(\eta\) and \(\varepsilon >0\), we take \(\delta>0\) satisfying \(\mathbb{P}_{x}^{Z^{\infty}}\otimes \mathbb{P}^{X^{\infty}}(\tilde{A}_{\eta}\leq \delta)\leq \varepsilon/2.\)
By Lemma \ref{PCAF_Tconv}, there exists \(n_*\in \mathbb{N}\) such that, for any \(n \geq n_*\), 
\[|\mathbb{P}_{x}^{Z^{\infty}}\otimes \mathbb{P}^{X^{\infty}}(\tilde{A}_{\eta}\leq \delta)-\mathbb{P}_{x_n}^{Z^n}\otimes \mathbb{P}^{X^n}(A^n_{\eta}\leq \delta)|\leq \varepsilon/2.\]
Since \(A^n\) is non-decreasing \(\mathbb{P}_{x_n}^{Z^n}\otimes \mathbb{P}^{X^n}\)-almost surely for each \(n\), we have \[\mathbb{P}_{x_n}^{Z^n}\otimes \mathbb{P}^{X^n} \left(v_T((A_T^{n})^{-1},0,\delta) \geq \eta\right)=\mathbb{P}_{x_n}^{Z^n}\otimes \mathbb{P}^{X^n}\left((A_{\delta}^n)^{-1} \geq \eta \right) = \mathbb{P}_{x_n}^{Z^n}\otimes \mathbb{P}^{X^n}\left(A_{\eta}^n\leq \delta \right).\]
These inequalities yield that \(\mathbb{P}_{x_n}^{Z^n}\otimes \mathbb{P}^{X^n}(A^n_T\leq \delta)\leq \varepsilon\) for any \(n\geq n_*\). Thus the proof is completed. 
\end{proof}

To check the criterion Theorem \ref{M1cri2} (1), we prove the following statement.
\begin{lemma}\label{lemma8}
\(\tilde{A}_T/T\) converges to \(1\) in \(L^2(\mathbb{P}_x^{Z^{\infty}}\otimes \mathbb{P}^X)\) as \(T \to \infty.\)
\end{lemma}
\begin{proof}
By Lemma \ref{PCAFcor}, \(\tilde{A}^{n,\infty}_T\) converges to \(\tilde{A}_T\) in \(L^2(\mathbb{P}_x^{Z^{\infty}} \otimes \mathbb{P}^{X})\) and so \(\mathbb{E}_x^{Z^{\infty}}\mathbb{E}^{X}[(\tilde{A}_T)^2]\) can be expressed as
\[\mathbb{E}_x^{Z^{\infty}}\mathbb{E}^{X}[(\tilde{A}_T)^2]= \int_0^T \int_0^T \mathbb{E}_x^{Z^{\infty}}[\exp{(\gamma^2 \pi g_{\lambda}^{\infty}(Z_s^{\infty},Z_t^{\infty}))}]dsdt.\] Combining this with change of variables, the stationary increments of \(Z^{\infty}\) and the self-similarity of \(Z^{\infty}\), we have
\begin{eqnarray}
\nonumber \lefteqn{\frac{1}{T^2}\mathbb{E}_x^{Z^{\infty}}\mathbb{E}^{X}\left[(\tilde{A}_T)^2\right] = \int_0^1 \int_0^1 \mathbb{E}_x^{Z^{\infty}}[\exp{(\gamma^2 \pi g_{\lambda}^{\infty}(Z_{Ts}^{\infty},Z_{Tt}^{\infty}))}]dsdt} \\
\nonumber &=& \int_0^1 \int_0^1 \mathbb{E}_0^{Z^{\infty}}[\exp{(\gamma^2 \pi g_{\lambda}^{\infty}(0,Z_{T|t-s|}^{\infty}))}]dsdt\\
\nonumber &=& \int_0^1 \int_0^1 \mathbb{E}_0^{Z^{\infty}}[\exp{(\gamma^2 \pi g_{\lambda}^{\infty}(0,T^{1/d}Z_{|t-s|}^{\infty}))}]dsdt\\
&=& \int_0^1 \int_0^1 \int_{\mathbb{R}^d} \exp{(\gamma^2 \pi g_{\lambda}^{\infty}(0,T^{1/d}y))} p^{\infty}(|t-s|,0,y) dm(y) ds dt. \label{eq:8-3}
\end{eqnarray}
For \(T\geq 1\), by Lemma \ref{Green'slemma}, we have
\begin{eqnarray}
e^{\gamma^2 \pi g_{\lambda}^{\infty}(0,T^{1/d}y)} &\leq & e^{\gamma^2 \log_{+}{\frac{1}{T^{1/d}|y|}}+C}\leq \frac{C}{|y|^{\gamma^2}}{\bf 1}_{\{|y|\leq 1\}}+C {\bf 1}_{\{|y|> 1\}}.\label{eq:8-4}
\end{eqnarray}
We remark the last two terms of \((\ref{eq:8-4})\) do not depend on \(T\). By Lemma \ref{selfsimlemma}, we have
\begin{eqnarray}
\nonumber \lefteqn{\int_0^1 \int_0^1 \int_{\mathbb{R}^d} \left(\frac{C}{|y|^{\gamma^2}}{\bf 1}_{\{|y|\leq 1\}}+C {\bf 1}_{\{|y|> 1\}}\right) p^{\infty}(|t-s|,0,y))dm(y)dsdt}\\
\nonumber &\leq & C \int_0^1\int_0^1 \mathbb{E}_x^{Z^{\infty}}\left[ \frac{1}{|Z_s^{\infty}-Z^{\infty}_t|^{\gamma^2}} \right]dsdt +C\\
&=& C \int_0^1\int_0^1 \frac{1}{|t-s|^{\gamma^2/d}}dsdt +C\  <\  \infty. \label{eq:8-5}
\end{eqnarray}
By \((\ref{eq:8-4})\), \((\ref{eq:8-5})\) and applying Lebesgue's convergence theorem to \((\ref{eq:8-3})\), we have

\begin{eqnarray*}
\lim_{T\to \infty} \frac{1}{T^2}\mathbb{E}_x^{Z^{\infty}}\mathbb{E}^{X} \left[(\tilde{A}_T)^2\right]&=& \lim_{T\to \infty} \int_0^1 \int_0^1 \int_{\mathbb{R}^d} e^{\gamma^2 \pi g_{\lambda}^{\infty}(0,T^{1/d}y)} p^{\infty}(|t-s|,0,y) dm(y) ds dt\\
&=& \int_0^1 \int_0^1 \int_{\mathbb{R}^d}  p^{\infty}(|t-s|,0,y) dm(y) ds dt\\
&=&1.
\end{eqnarray*}
Combining this with \(\mathbb{E}_x^{Z^{\infty}}\mathbb{E}^{X}[\tilde{A}_T]=T\), the proof is completed.
\end{proof}

\begin{lemma}\label{lemma9}
For any \(T >0\) and \(x\in \R^d\), it holds that
\[\lim_{c\to \infty}\varlimsup_{n \to \infty}\mathbb{P}_{x_n}^{Z^n}\otimes \mathbb{P}^{X}\left(\sup_{t\leq T}|(A_t^{n})^{-1}|> c\right)=0.\]
\end{lemma}

\begin{proof}
For any \(t >0\) and \(\eta >0\), we have \(\mathbb{E}_{x_n}^{Z^n}\mathbb{E}^X [A_t^n]=t\) and
\begin{eqnarray*}
\nonumber \mathbb{E}_{x_n}^{Z^n}\mathbb{E}^X[A_t^n]&=&\mathbb{E}_{x_n}^{Z^n}\mathbb{E}^X[A_t^n {\bf 1}_{\{A_t^n\leq \eta\}}]+\mathbb{E}_{x_n}^{Z^n}\mathbb{E}^X[A_t^n {\bf 1}_{\{A_t^n> \eta\}}]\\
&\leq & \eta \mathbb{P}_{x_n}^{Z^n}\otimes \mathbb{P}^{X}(A_t^n \leq \eta) + \sqrt{\mathbb{E}_{x_n}^{Z^n}\mathbb{E}^X[(A_t^n)^2]} \sqrt{\mathbb{P}_{x_n}^{Z^n}\otimes \mathbb{P}^{X}(A_t^n > \eta)}
\end{eqnarray*}
By Proposition \ref{PCAF_Tconv}, taking the limit as \(n\) tends to infinity, we have
\begin{equation}
t\leq \eta \mathbb{P}_{x}^{Z^{\infty}} \otimes \mathbb{P}^X
(\tilde{A}_t \leq \eta) + \sqrt{\mathbb{E}_{x}^{Z^{\infty}}\mathbb{E}^X[(\tilde{A}_t)^2]} \sqrt{\mathbb{P}_{x}^{Z^{\infty}} \otimes \mathbb{P}^X(\tilde{A}_t > \eta)}. \label{eq:9-c}
\end{equation}
By \((\ref{eq:9-c})\) and Lemma \ref{lemma8}, we have
\begin{eqnarray}
\nonumber \varlimsup_{t\to \infty} \mathbb{P}_{x}^{Z^{\infty}} \otimes \mathbb{P}^X(\tilde{A}_t\leq \eta )&= & 1- \varliminf_{t\to \infty} \mathbb{P}_{x}^{Z^{\infty}} \otimes \mathbb{P}^X(\tilde{A}_t> \eta )\\
\nonumber &\leq & 1- \varliminf_{t\to \infty} \frac{\left(t-\eta \mathbb{P}_{x}^{Z^{\infty}} \otimes \mathbb{P}^X(\tilde{A}_t\leq \eta)\right)^2}{\mathbb{E}_x^{Z^{\infty}}\mathbb{E}^X[(\tilde{A}_t)^2]} \\
\nonumber &=& 1-  \lim_{t\to \infty} \frac{1}{ \mathbb{E}_x^{Z^{\infty}}\mathbb{E}^X[(\frac{\tilde{A}_t}{t})^2]} \\
&=& 0. \label{eq:9-1}
\end{eqnarray}

On the other hand, by Proposition \ref{PCAF_Tconv}, for any \(\varepsilon >0\) and \(c>0\), there exists \(n^*\) such that, for any \(n\geq n^*\), 
\begin{equation}
|\mathbb{P}_{x}^{Z^{\infty}} \otimes \mathbb{P}^X(\tilde{A}_{c}\leq T)-\mathbb{P}_{x_n}^{Z^n}\otimes \mathbb{P}^{X}(A^{n}_{c}\leq T)|\leq \varepsilon /2. \label{eq:9-2}
\end{equation}
Since \(A^n\) is non-decreasing \(\mathbb{P}_{x_n}^{Z^n}\otimes \mathbb{P}^{X}\)-almost surely for each \(n\), we have 
\[\mathbb{P}_{x_n}^{Z^n}\otimes \mathbb{P}^{X}\left(\sup_{t\leq T}|(A_t^{n})^{-1}|> c\right)=\mathbb{P}_{x_n}^{Z^n}\otimes \mathbb{P}^{X}\left((A_T^{n})^{-1}> c\right)\leq \mathbb{P}_{x_n}^{Z^n}\otimes \mathbb{P}^{X}\left(A_c^{n}\leq T \right).\]
So the proof is completed by combining (\ref{eq:9-2}) with (\ref{eq:9-1}) replaced \(t\) and \(\eta\) with \(c\) and \(T\), respectively.
\end{proof}

\begin{theorem}\label{AinvM1tight}
The family of the distributions of \((A^n)^{-1}\) under \(\mathbb{P}_{x_n}^{Z^n}\otimes \mathbb{P}^X\) is tight with respect to \(M_1\)-topology. 
\end{theorem}
\begin{proof}
Since \((A^n)^{-1}\) is increasing \(\mathbb{P}_{x_n}^{Z^n}\otimes \mathbb{P}^X\)-almost surely for each \(n\), the criterion Theorem \ref{M1cri2} (2) automatically holds. Combining this with Lemma \ref{lemma7} and Lemma \ref{lemma9}, \(\{(A^n)^{-1}\}_n\) is tight with respect to \(M_1\)-topology by Theorem \ref{M1cri2}.
\end{proof}

We show the following lemma in order to prove the convergence with respect to \(U\)-topology, where \(U\)-topology is the local uniform topology. We remark that the convergence with respect to any of Skorokhod's topologies are the same when the limit is continuous.

\begin{lemma}\label{lemma11}
Any subsequential limit \(\tilde{A}^{-1}\) of \(\{(A^{n})^{-1}\}_{n}\) is strictly increasing, \(\mathbb{P}_x^{Z^{\infty}}\otimes \mathbb{P}^X\)-almost surely.
\end{lemma}
\begin{proof}
Let \(\tilde{A}^{-1}\) be a subsequential limit of \(\{(A^{n})^{-1}\}_{n}\). For convenience, by taking a subsequence, we may assume that \((A^n)^{-1}\) converges weakly to \(\tilde{A}^{-1}\) with respect to \(M_1\)-topology. By \cite[Lemma 13.2.3]{W}, the family of non-decreasing c\`{a}dl\`{a}g functions is closed in \(D[0,\infty)\) with respect to \(M_1\)-topology, so \(\tilde{A}^{-1}\) is non-decreasing, \(\mathbb{P}_x^{Z^{\infty}}\otimes \mathbb{P}^X\)-almost surely. It holds that
\begin{eqnarray*}
\nonumber \lefteqn{\mathbb{P}_{x}^{Z^{\infty}}\otimes \mathbb{P}^X(\text{There\ exist}\ s<t\ \text{such\ that\ }\tilde{A}_t^{-1}=\tilde{A}_s^{-1})}\\
\nonumber &=&\lim_{T\to \infty}\mathbb{P}_{x}^{Z^{\infty}}\otimes \mathbb{P}^X(\text{There\ exist}\ s<t\leq T\ \text{such\ that\ }\tilde{A}_t^{-1}=\tilde{A}_s^{-1})\\
\nonumber &=&\lim_{T\to \infty}\lim_{N\to \infty} \mathbb{P}_{x}^{Z^{\infty}}\otimes \mathbb{P}^X(\text{There\ exists}\ k\  \text{such\ that}\ 1\leq k\leq 2^N,\ \tilde{A}_{kT/2^N}^{-1}=\tilde{A}_{(k-1)T/2^N}^{-1})\\
&\leq &\lim_{T\to \infty}\lim_{N\to \infty}\sum_{k=1}^{2^N} \mathbb{P}_{x}^{Z^{\infty}}\otimes \mathbb{P}^X(\tilde{A}_{kT/2^N}^{-1}=\tilde{A}_{(k-1)T/2^N}^{-1}).
\end{eqnarray*}
So, it is enough to show that \(\mathbb{P}_{x}^{Z^{\infty}}\otimes \mathbb{P}^X(\tilde{A}_t^{-1}=\tilde{A}_s^{-1})=0\) for any \(s<t\).
By \cite[Lemma VI.3.12]{JS}, \(\{t>0\ ; \ \mathbb{P}_{x}^{\infty}(\tilde{A}_t^{-1}=\tilde{A}_{t-}^{-1})=1\}\) is dense in \([0,\infty)\). So, without loss of generality, we may assume that \(s\) and \(t\) satisfy \(\mathbb{P}_{x}^{Z^{\infty}}\otimes \mathbb{P}^X(\tilde{A}_t^{-1}=\tilde{A}_{t-}^{-1})=\mathbb{P}_{x}^{Z^{\infty}}\otimes \mathbb{P}^X(\tilde{A}_s^{-1}=\tilde{A}_{s-}^{-1})=1\). By Theorem \ref{whitt} and the definition of PCAF, we have
\begin{eqnarray}
\nonumber \mathbb{P}_{x}^{Z^{\infty}}\otimes \mathbb{P}^X(\tilde{A}_t^{-1}=\tilde{A}_s^{-1})&=&\lim_{N\to \infty}\mathbb{P}_{x}^{Z^{\infty}}\otimes \mathbb{P}^X(\tilde{A}_t^{-1}-\tilde{A}_s^{-1}<\frac{1}{N})\\
\nonumber &\leq & \lim_{N\to \infty}\varliminf_{n\to \infty}\mathbb{P}_{x_n}^{Z^n}\otimes \mathbb{P}^X((A^{n}_t)^{-1}-(A^{n}_s)^{-1}<\frac{1}{N})\\
\nonumber &=&\lim_{N\to \infty}\varliminf_{n\to \infty}\mathbb{P}_{x_n}^{Z^n}\otimes \mathbb{P}^X(t\leq A^{n}_{(A_s^{n})^{-1}+\frac{1}{N}})\\
&=&\lim_{N\to \infty}\varliminf_{n\to \infty}\mathbb{P}_{x_n}^{Z^n}\otimes \mathbb{P}^X(t\leq s+A^{n}_{\frac{1}{N}}\circ \theta_{(A^{n}_s)^{-1}})\label{eq:11-1}
\end{eqnarray}
Since \(\{(A^{n}_s)^{-1}\leq u\}=\{s\leq A_u^{n}\}\in \sigma(\{Z_v^{n}\}_{v\leq u})\), \(\mathbb{P}^X\)-almost surly, by the strong Markov property for \(Z^{n}\), the stationarity of \(X^n\) and \(\mathbb{E}_{x_n}^n[A^n_{1/N}]=1/N\), we have
\begin{eqnarray*}
\mathbb{P}_{x}^{Z^{\infty}}\otimes \mathbb{P}^X(\tilde{A}_t^{-1}=\tilde{A}_s^{-1})&\leq & \lim_{N\to \infty}\varliminf_{n\to \infty}\mathbb{P}_{x_n}^{Z^n}\otimes \mathbb{P}^X(t<s+A^{n}_{\frac{1}{N}})\\
&\leq & \lim_{N\to \infty}\varliminf_{n\to \infty} \frac{\frac{1}{N}}{t-s} \ =\ 0.
\end{eqnarray*}
Thus the proof is complete.
\end{proof}

Although the next theorem is not necessary for the proof of Theorem \ref{maintight}, we state here of independent interest.
\begin{theorem}\label{theorem12}
For any \(x\in \R^d\), the collection of distributions of time-changed processes \(\hat{Z^n}\) under \(\mathbb{P}_{x_n}^{Z^n}\otimes \mathbb{P}^X\) is tight with respect to \(L^1_{loc}\)-topology.
\end{theorem}
\begin{proof}
It follows from Assumption \ref{assumption} (1), Lemmas \ref{AinvM1tight}, \ref{lemma11} and Theorem \ref{J1M1L1}.
\end{proof}

We prove that \(\{A^n\}_n\) is also tight with respect to \(M_1\)-topology and its subsequential limit is strictly increasing as follows. We will use these properties to prove the weak convergence of \(A^n\) and \((A^n)^{-1}\) with respect to \(U\)-topology.
\begin{lemma}\label{AM1tight}
The family of the distributions of \(A^n\) under \(\mathbb{P}_{x_n}^{Z^n}\otimes \mathbb{P}^X\) is tight with respect to \(M_1\)-topology. 
\end{lemma}
\begin{proof}
This follows from the fact that \(\mathbb{P}_{x_n}^{Z^n}\otimes \mathbb{P}^X(A_T^n\geq \eta)\leq T/\eta\) and Theorem \ref{M1cri2}.
\end{proof}

\begin{lemma}\label{lemma14}
Any subsequential limit \(\tilde{A}\) of \(\{A^{n}\}_{n}\) is strictly increasing, \(\mathbb{P}_x^{Z^{\infty}}\otimes \mathbb{P}^X\)-almost surely.
\end{lemma}
\begin{proof}
By Lemma \ref{newlemma}, we may replace \(A^n\) to \(A^{n, \infty}\). By Lemma \ref{PCAFcor}, this follows from Lemma \ref{lemma7} and the same way as the proof of Lemma \ref{lemma11}.
\end{proof}

\begin{proposition}\label{lemma15}
\((1)\)\ The family of the distributions under \(\mathbb{P}_{x_n}^{Z^n}\otimes \mathbb{P}^X\) of \(A^n\) is \(U\)-tight and a subsequential limit of \(\{A^n\}_n\) is continuous \(\mathbb{P}_x^{Z^{\infty}}\otimes \mathbb{P}^X\)-almost surely.\\
\((2)\)\ The family of the distributions under \(\mathbb{P}_{x_n}^{Z^n}\otimes \mathbb{P}^X\) of \((A^n)^{-1}\) is \(U\)-tight and a subsequential limit of \(\{(A^n)^{-1}\}_n\) is continuous \(\mathbb{P}_x^{Z^{\infty}}\otimes \mathbb{P}^X\)-almost surely.
\end{proposition}
\begin{proof}
By Lemmas \ref{AinvM1tight} and \ref{AM1tight}, \(\{A^n\}_n\) and \(\{(A^n)^{-1}\}_n\) are \(M_1\)-tight, respectively. So, for any sequence \(n_i\), there exists a subsequence \(\{n'_i\}\) and \(\tilde{A}, \tilde{B}\) such that \(A^{n'_i}\) converges weakly to \(\tilde{A}\) with respect to \(M_1\)-topology and  \((A^{n'_i})^{-1}\) converges weakly to \(\tilde{B}\) with respect to \(M_1\)-topology. For convenience, we write the subsequence \(\{n_i\}\) as \(\{n'_i\}\).

By Lemma \ref{lemma11} and Lemma \ref{lemma14}, \(\tilde{A}\) and \(\tilde{B}\) are unbounded strictly increasing c\`{a}dl\`{a}g processes. So, by the continuity of an inverse operator at a strictly increasing function \cite[Corollary 13.6.4]{W}, it holds that \(((A^{n_i})^{-1})^{-1}=A^{n_i}\) converges weakly to \(\tilde{B}^{-1}\) with respect to \(U\)-topology. So we have \(\tilde{A}=\tilde{B}^{-1}\) in distribution and \(\tilde{A}\) is continuous because \(C[0,\infty)\) is the closed set of \(D[0,\infty)\) with respect to \(U\)-topology. By \cite[Lemma 13.6.5]{W}, \(\tilde{A}^{-1}\) is also continuous, and so \((A^{n_i})^{-1}\) converges weakly to \(\tilde{A}^{-1}\) with respect to \(U\)-topology.
\end{proof}

The proof of the following Proposition \ref{PCAFlem} is postponed to Section \ref{secPCAFlem}.
\begin{proposition}\label{PCAFlem}
\(\tilde{A}\) and \(A^{\infty}\) have the same distribution under \(\mathbb{P}_x^{Z^{\infty}} \otimes \mathbb{P}^{X}\) for any \(x\in \R^d\).
\end{proposition}

\begin{proof}[Proof of Theorem \ref{PCAFconvergence}]
By Proposition \ref{lemma15}, for any sequence, there exit a subsequence \(\{n_i\}\) and continuous process \(\check{A}\) such that \(A^{n_i}\) converges weakly to \(\check{A}\) with respect to \(U\)-topology. By Theorem \ref{whitt}, \(A^{n_i}\) converges to \(\check{A}\) in finite-dimensional distribution. By Lemma \ref{PCAF_Tconv}, \(A^{n_i}_t\) converges weakly to \(\tilde{A}_t\) for any \(t\). Thus, for any \(t\), it holds that \(\mathbb{P}_x^{Z^{\infty}}\otimes \mathbb{P}^X(\tilde{A}_t=\check{A}_t)=1\) and, by \cite[Problem 1.1.5]{KS}, we have \(\mathbb{P}_x^{Z^{\infty}}\otimes \mathbb{P}^X(\tilde{A}_t=\check{A}_t\ \text{ for\ any\ }t)=1\). Therefore \(\tilde{A}\) and \(\check{A}\) have the same distribution under \(\mathbb{P}_x^{Z^{\infty}}\otimes \mathbb{P}^X\) with respect to \(U\)-topology, and this yields that \(A^n\) converges weakly to \(\tilde{A}\) with respect to \(U\)-topology.

In the same way as the proof of Proposition \ref{lemma15}, \((A^n)^{-1}\) converges weakly to \(\tilde{A}^{-1}\) with respect to \(U\)-topology. 

Combining these with Proposition \ref{PCAFlem} completes the proof.
\end{proof}

\subsection{Convergence of finite-dimensional distributions}\label{seccfdd}
In this subsection, we prove the convergence of the pairs \((Z^n, A^n)\) under \(\mathbb{P}_{x_n}^{Z^n}\otimes \mathbb{P}^X\) to \((Z^{\infty}, A^{\infty})\) under \(\mathbb{P}_{x}^{Z^{\infty}}\otimes \mathbb{P}^X\) in finite-dimensional distribution. Recall that we defined
\[A^{n,\infty}_T:= \int_0^T \exp{(\gamma X^n(Z_s^{\infty})-\frac{\gamma^2}{2} \E^{X^n}[X^n(Z^{\infty}_s)^2])}ds \] 
for \(n\in \mathbb{N}\), and \(\tilde{A}_T\) as the weak limit of \(A^n_T\) for any \(T\geq 0\) in the previous subsection. Remark that \(\tilde{A}\) has the same distribution as that of \(A^{\infty}\) by Proposition \ref{PCAFlem}. Throughout of this subsection, we take  \(\gamma \in (0,\sqrt{d/C^*})\).

\begin{lemma}\label{lemma18}
The pair \((Z^n, A^n)\) under \(\mathbb{P}_{x_n}^{Z^n}\otimes \mathbb{P}^X\) converges to \((Z^{\infty}, A^{\infty})\) under \(\mathbb{P}_{x}^{Z^{\infty}}\otimes \mathbb{P}^X\) in finite-dimensional distribution on the continuity point of \((Z^{\infty}, A^{\infty})\) for any \(x\in \R^d\).
\end{lemma}
\begin{proof}

Take any positive integer \(k\), and \(\{t_i\}_{1\leq i \leq k}\) satisfying \(0\leq t_1<t_2<\cdots <t_k\) and \(\mathbb{P}_{x}^{Z^{\infty}}(Z_{t_i}^{\infty}=Z_{t_{i-}}^{\infty})=1\). For any bounded Lipschitz function \(F:\Pi_{i=1}^k (\R^d \times [0,\infty))\to \R\), by Lemma \ref{PCAFcor}, we have
\begin{eqnarray}
\nonumber &&\varlimsup_{n\to \infty} \left| \mathbb{E}_x^{Z^{\infty}}\mathbb{E}^X[F(Z^{\infty}_{t_1}, A_{t_1}^{n,\infty},Z^{\infty}_{t_2},\cdots ,A_{t_k}^{n,\infty})]-\mathbb{E}_x^{Z^{\infty}}\mathbb{E}^X[F(Z^{\infty}_{t_1}, A_{t_1}^{\infty},Z^{\infty}_{t_2},\cdots ,A_{t_k}^{\infty})]  \right|\\
&\leq &\varlimsup_{n\to \infty} C\sum_{i=1}^k \mathbb{E}_x^{Z^{\infty}}\mathbb{E}^X|A_{t_i}^{n,\infty}-A_{t_i}^{\infty}|\  = \ 0. \label{eq:new3}
\end{eqnarray}

We define the random functions \(G^n:\D \to \Pi_{i=1}^k (\R^d \times [0,\infty))\) by 
\[G^n(z):=\left(z_{t_1}, \int_0^{t_1}e^{\tilde{X}^n(z_s)}ds, \cdots, z_{t_k}, \int_0^{t_k}e ^{\tilde{X}^n(z_s)}ds\right),\]
where \(\tilde{X^n}(y):=\gamma X^n(y)-\frac{\gamma^2}{2} \EX [X^n(y)^2]\) for \(y\in \R^d.\) By Skorokhod's representation theorem, we may assume that \(\{Z^n\}_n\) for all \(m\in \mathbb{N}\cup \{\infty\}\) are on the same probability space \((\Omega, \mathbb{P}_x^{Z})\). For \(\tilde{A}^{n,\infty}_T:= \int_0^T \exp{(\gamma X^n(\tilde{Z}_s^{\infty})-\frac{\gamma^2}{2} \E^{X^n}[X^n(\tilde{Z}^{\infty}_s)^2])}ds\), by Lemma \ref{newlemma}, 
\begin{equation}
\left|\EX\E_x^{Z}[FG^n(Z^n)]-\EX\E_x^{Z}[FG^n(Z^{\infty})]
\right|^2\leq C \EX\E_x^{Z} \left[\sum_{i=1}^k |Z^n_{t_i}-Z^{\infty}_{t_i}|^2 + |A^{n}_{t_i}-A_{t_i}^{n,\infty}  
|^2\right] \label{eq:new4}
\end{equation}
converges to \(0.\) By \((\ref{eq:new3})\), \((\ref{eq:new4})\) and the independence of \(\{Z^n\}_n\) and \(\{X^n\}_n\), we have
\[\lim_{n\to \infty} \left| \mathbb{E}_x^{Z^{\infty}}\mathbb{E}^X[F(Z^{n}_{t_1}, A_{t_1}^{n},Z^{n}_{t_2},\cdots ,A_{t_k}^{n})]-\mathbb{E}_x^{Z^{\infty}}\mathbb{E}^X[F(Z^{\infty}_{t_1}, A_{t_1}^{\infty},Z^{\infty}_{t_2},\cdots ,A_{t_k}^{\infty})]  \right|=0.\]
So the proof is completed.
\end{proof}

\begin{remark}
If \(Z^n\) were independent of \(A^n\), the convergence of the joint distribution of \(Z^n\) and \(A^n\) follows from the convergence of the marginal distributions of \(Z^n\) and \(A^n\) immediately. In our case, we use the property that \(A^n\) can be viewed as a function of \(Z^n\), instead of the independence.
\end{remark}

\subsection{Convergence of \((Z^n, A^n)\)}\label{seccjd}
\begin{proposition}\label{lemma19}
For \(\gamma \in (0,\sqrt{d/C^*})\), the joint distribution of \((Z^n, A^n)\) under \(\mathbb{P}_{x_n}^{Z^n}\otimes \mathbb{P}^{X^n}\) converges to that of \((Z^{\infty}, A^{\infty})\) under \(\mathbb{P}_x^{Z^{\infty}}\otimes \mathbb{P}^{X^{\infty}}\) with respect to \(J_1 \times U\)-topology as \(n \to \infty\) for any \(x\in \R^d\). 
\end{proposition}

\begin{proof}
This follows from Theorem \ref{maintight}, Lemma \ref{lemma18} and Theorem \ref{whitt}.
\end{proof}

\begin{proof}[Proof of Theorem \ref{mainconvergence}]
This follows from Proposition \ref{lemma19}, the continuity of an inverse map \cite[Corollary 13.6.4]{W}, Theorem \ref{J1UJ1}, and the strictly increase and the continuity of \(A^{\infty}\).
\end{proof}

\section{Proof of the equality in distribution of \(\tilde{A}\) and \(A^{\infty}\)}\label{secPCAFlem}
In this section, we prove Proposition \ref{PCAFlem} by proving that there exists a PCAF \(\bar{A}\) having the same distribution as that of \(\tilde{A}\), and this PCAF \(\bar{A}\) corresponds to \(\mu^{\infty}\). See Appendix \ref{AppPCAF} for the definition and properties of a PCAF.

\begin{proposition}\label{PCAFlem2}
There exist \(\Lambda \in \mathscr{F}^{X}\otimes \mathscr{F}_{\infty}^{Z^{\infty}}\) and \(\bar{A}\) such that \(\mathbb{P}_x^{Z^{\infty}}(\Lambda_{\omega})=1\) for any \(x\in \R^d\) and \(\bar{A}(\omega, \cdot)\) is a PCAF on the defining set \(\Lambda_{\omega}\) for \(\mathbb{P}^X\)-a.e. \(\omega\), where \(\Lambda_{\omega}:=\{\omega'\in \Omega^{Z^{\infty}}\ |\  (\omega, \omega')\in \Lambda\}\). Moreover, \(\bar{A}\) has the same distribution as that of \(\tilde{A}\) under \(\mathbb{P}_x^{Z^{\infty}}\otimes \mathbb{P}^X\) for any \(x\in \R^d\).
\end{proposition}

We prove Proposition \ref{PCAFlem2} in a similar way to the proof of \cite[Proposition 2.4]{AK}. In \cite{AK}, it is assumed that the almost sure convergence of PCAFs. Although this assumption does not hold in our case, we can prove Proposition \ref{PCAFlem2} by using the following two lemmas instead of the assumption in \cite[Proposition 2.4]{AK}.

\begin{lemma}\label{PCAFassumption}
For any \(x\in \R^d\) and \(t>0,\) \(\int_0^t |A_s^{n,\infty}-\tilde{A}_s|ds\) converges to \(0\) in \(L^1(\mathbb{P}_x^{Z^{\infty}}\otimes \mathbb{P}^X)\) as \(n\to \infty\).
\end{lemma}
\begin{proof}
By Lemma \ref{PCAFcor}, for \(s\leq t\), \(x\in \R^d\) and \(\mathbb{P}_x^{Z^{\infty}}\)-almost sure \(\omega\), \(A_s^{n,\infty}\) converges to \(\tilde{A}_s\) in \(L^1(\mathbb{P}^X)\). Combining this with 
\[\EX |A^{n,\infty}_s-\tilde{A}_s| \leq \EX [A^{n,{\infty}}_s]+\EX [\tilde{A}_s]=2s,\]
by Lebesgue's convergence theorem, it holds that
\begin{eqnarray*}
\lim_{n\to \infty} \mathbb{E}_x^{Z^{\infty}}\mathbb{E}^{X} \left[\int_0^t |A_s^{n,{\infty}}-\tilde{A}_s|ds\right]= \mathbb{E}^{Z^{\infty}}_x \int_0^t \lim_{n\to \infty} \EX[|A_s^{n,{\infty}}-\tilde{A}_s|]ds=0
\end{eqnarray*}
for any \(t>0\) and \(x\in \R^d\). So the proof is completed.
\end{proof}

\begin{lemma}\label{contiincreconv}
For any \(t>0\) and non-decreasing functions \(f_n, f\in C([0,t])\), if \(\int_0^t|f_n(s)-f(s)|ds\) converges to \(0\) as \(n\to \infty\), then \(f_n\) converges pointwise to \(f\) on \((0,t)\).
\end{lemma}
\begin{proof}
If this lemma does not hold, there exist \(s_0\in (0,t)\), an increasing sequence \(\{n_k\}_k\) satisfying \(n_k\to \infty\) and \(\varepsilon>0\) such that \(|f_{n_k}(s_0)-f(s_0)|>\varepsilon\) holds.

If \(f_{n_k}(s_0)-f(s_0)>\varepsilon\) holds, then, by the continuity of \(f\) at \(s_0\), there exists \(\delta\) such that \(|f(s_0)-f(s)|\leq \varepsilon/2\) for \(s\in [(s_0-\delta)\vee 0,(s_0+\delta)\wedge t]\). Combining this with the monotonicity of \(f_{n_k}\), for \(s\in [s_0,(s_0+\delta)\wedge t]\), we have 
\begin{eqnarray*}
f_{n_k}(s)-f(s)\ \geq \ f_{n_k}(s)-f(s_0)-\varepsilon/2 \ \geq \ f_{n_k}(s_0)-f(s_0)-\varepsilon/2 \ \geq \ \varepsilon/2
\end{eqnarray*}
and so
\[\int_0^t|f_{n_k}(s)-f(s)|ds\ \geq \ \int_{s_0}^{(s_0+\delta)\wedge t}\frac{\varepsilon}{2} ds \ \geq \ \frac{\delta \wedge (t-s_0)}{2}\varepsilon \ >\ 0.\]
This is a contradiction.
If \(f(s_0)-f_{n_k}(s_0)>\varepsilon\) holds, in the same way as above, we also get a contradiction.
\end{proof}

PCAFs \(A^{n,{\infty}}\) and \(\tilde{A}\) are continuous and non-decreasing \(\mathbb{P}_x^{Z^{\infty}}\otimes \mathbb{P}^X\)-almost surely and \(A^{n,{\infty}}_0=\tilde{A}_0=0\). Moreover \(L^1\)-convergence implies the subsequential almost everywhere convergence. So the following holds from Lemmas \ref{PCAFassumption}, \ref{contiincreconv}. 
\begin{corollary}\label{PCAFassu1}
For any \(t>0\) and \(x\in \R^d\), there exist \(\{n_i\}_i\subset \mathbb{N}\) satisfying \(n_i \to \infty\) such that \(A^{n_i, \infty}_s(\omega, \omega')\) converges pointwise to \(\tilde{A}_s(\omega, \omega')\) as \(n_i \to \infty\), for \(\mathbb{P}_x^{Z^{\infty}}\otimes \mathbb{P}^X\)-almost sure \((\omega, \omega')\) and \(s\in[0,t]\).
\end{corollary}

\begin{proof}[Proof of Proposition \ref{PCAFlem2}]
In this proof, for any subset \(\Lambda \subset \Omega^X \times \Omega ^{Z^{\infty}}\), we define its section for \(\omega\in \Omega^X\) by \(\Lambda^{\omega}:=\{\omega' \in \Omega^{Z^{\infty}} \mid (\omega, \omega')\in \Lambda\}.\) For \(T\geq 0\) and \(0\leq t\leq T\), we define \(\Lambda_t^T \) by
\begin{equation*}
\left\{(\omega, \omega')\in \Omega ^X \times \Omega^{Z^{\infty}}\ \middle| \ \begin{split} \text{There\ exist\ }n_i \to \infty \text{\ such\ that,}\text{\ for\ any\ }u\in [t,T],\hspace{10mm}\\
 \tilde{A}_{t,u}(\omega,\omega'):=\lim_{n_i\to \infty}(A^{n_i, \infty}_u(\omega,\omega')-A^{n_i, {\infty}}_t(\omega,\omega'))\text{\ exists},\\
[t,T]\ni u \mapsto \tilde{A}_{t,u}(\omega, \omega')\in [0,\infty)\text{\ is\ continuous\ and}\hspace{10mm}\\
\text{\ strictly\ increasing,\ and\ } \tilde{A}_{t,t}(\omega, \omega')=0.\hspace{27mm}
\end{split} \right\}.
\end{equation*}
For \(t\geq 0,\) we also define 
\[\Lambda_t:=\bigcap_{t\leq T\in \mathbb{N}}\Lambda_t^{T}.\]

Let \(\{\theta_t\}_{t\geq 0}\) be a collection of shift operators of \(Z^{\infty}\). Then the following holds.
\begin{lemma}\label{shiftlem1}
For any \(t>0\) and \(\omega \in \Omega ^X,\) \(\Lambda_t^{\omega}=\theta_t^{-1}(\Lambda_0^{\omega})\) holds.
\end{lemma}
\begin{proof}
At first, we prove \(\theta_t(\Lambda_t^{\omega})\subset \Lambda_0^{\omega}.\)
For \(\omega'\in \Lambda_t^{\omega}\) and \(T\in \mathbb{N}\), take \(\tilde{T}\in \mathbb{N}\) with \(\tilde{T}\geq T+t.\) Since \((\omega, \omega')\in \Lambda_t^{\tilde{T}}\), there are \(n_i\to \infty\) such that \([t,\tilde{T}]\ni u \mapsto \tilde{A}_{t,u}(\omega, \omega') \in [0,\infty)\) is continuous and strictly increasing. Since \(A^{n,\infty}\) are PCAFs having their defining sets as \(\Omega^{Z^{\infty}}\), \(A^{n_i,\infty}_u(\omega, \theta_t\omega')-A^{n_i,\infty}_0(\omega, \theta_t\omega')=A^{n_i,\infty}_{u+t}(\omega, \omega')-A^{n_i,\infty}_t(\omega, \omega')\) converges to \(\tilde{A}_{t,u+t}(\omega, \omega')\) for any \(u\in [0,T]\). Moreover \(\tilde{A}_{0,u}(\omega, \theta_t\omega')=\tilde{A}_{t,u+t}(\omega, \omega')\) is continuous and strictly increasing on \([0,T]\), and \(\tilde{A}_{0,0}(\omega, \theta_t\omega')=0.\) So \((\omega, \theta_t\omega') \in \Lambda_0^{T}\) and \(\theta_t\omega' \in \Lambda ^{\omega}_0.\)

In a similar way to above, we can also prove \(\theta_t(\Lambda_t^{\omega})\supset \Lambda_0^{\omega}.\)
\end{proof}

Since \(A_s^{n, \infty}\) is \(\mathscr{F}^X\otimes \mathscr{F}^{Z^{\infty}}_s\)-measurable, \(\Lambda_t\) is a \(\mathscr{F}^X\otimes \mathscr{F}^{Z^{\infty}}_{\infty}\)-measurable set. By Lemma \ref{lemma11}, \(\tilde{A}\) is continuous and strictly increasing \(\mathbb{P}_x^{Z^{\infty}}\otimes \mathbb{P}^X\)-almost surely. So, by Corollary \ref{PCAFassu1}, we have \(\mathbb{P}_x^{Z^{\infty}}\otimes \mathbb{P}^X(\Lambda_t^T)=1\) for \(0\leq t\leq T\), and \(\mathbb{P}_x^{Z^{\infty}}\otimes \mathbb{P}^X(\Lambda_t)=1\) for \(t\geq 0\).

\begin{lemma}\label{PCAFprob1lem}
For \(\mathbb{P}^X\)-almost sure \(\omega \in \Omega^X\), any \(t\geq 0\) and \(x\in \R^d,\) it holds that \(\mathbb{P}_x^{Z^{\infty}}(\Lambda_t^{\omega})=1\).
\end{lemma}
\begin{proof}
We can prove this lemma in the same way as that of \cite[Lemma A.3]{AK}. However, for the reader's convenience, we describe a proof. By Fubini's theorem, we have \(\mathbb{P}^{X}\otimes \mathbb{P}_x^{Z^{\infty}}(\Lambda_0^{\omega})=\mathbb{P}_x^{Z^{\infty}}\otimes \mathbb{P}^X(\Lambda_0)=1\) and 
\[\EX \int_{\R} \mathbb{P}_y^{Z^{\infty}}(\Lambda_0^{\omega}) p^{\infty}(0,0,y)dy=1.\]
So, for \(\mathbb{P}^X\)-a.s. \(\omega \in \Omega^X\) and almost every \(y\in \R^d\), \(\mathbb{P}_y^{Z^{\infty}}(\Lambda_0^{\omega})=1\) holds. By Lemma \ref{shiftlem1} and the Markov property for \(Z^{\infty}\), for \(\mathbb{P}^X\)-almost sure \(\omega \in \Omega^X\), \(t\geq 0\) and \(x\in \R^d\), we have
\[\mathbb{P}_x^{Z^{\infty}}(\Lambda_t^{\omega})=\mathbb{P}_x^{Z^{\infty}}(\theta_t^{-1}\Lambda_0^{\omega})=\mathbb{P}_x^{Z^{\infty}}\mathbb{P}_{Z_t^{\infty}}^{Z^{\infty}}(\Lambda_0^{\omega})=\int_{\R^d}\mathbb{P}_y^{Z^{\infty}}(\Lambda_0^{\omega})p^{\infty}(t,x,y)dy=1.\]
\end{proof}

We define \(\tilde{\Lambda}\) by
\begin{equation*}
\left\{(\omega,\omega') \in \Omega^X\times \Omega^{Z^{\infty}} \middle|\begin{split} \text{There\ exists\ }\{n_i\}\subset \mathbb{N}\text{\ such\ that\ }n_i\to \infty, \hspace{25mm} \\
\lim_{t\searrow 0}\varliminf_{n_i\to \infty}A^{n_i, \infty}_t(\omega, \omega')=0 
\text{\ and\ } \lim_{t\to \infty}\varlimsup_{n_i\to \infty}A^{n_i, \infty}_t(\omega,\omega')=\infty \end{split} \right\}.
\end{equation*}
By the monotone convergence theorem and Fatou's lemma, we have 
\[0\leq \E_x^{Z^{\infty}}\mathbb{E}^X[\lim_{t\searrow 0}\varliminf_{n\to \infty}A^{n,\infty}_t]\leq \lim_{t\searrow 0}\varliminf_{n \to \infty}t=0,\]
for any \(x\in \R^d.\) By Lemma \ref{lemma9}, we have \(\lim_{t\to \infty}\varlimsup_{n \to \infty}A^{n,\infty}_t=\infty \), \(\mathbb{P}_x^{Z^{\infty}}\otimes \mathbb{P}^X\)-almost surely, for any \(x\in \R^d\). So \(\mathbb{P}_x^{Z^{\infty}}\otimes \mathbb{P}^X(\tilde{\Lambda})=1\) holds for any \(x\in \R^d.\) In the same way as Lemma \ref{PCAFprob1lem}, for \(\mathbb{P}^X\)-a.s. \(\omega \in \Omega^X\), \(\mathbb{P}_x^{Z^{\infty}}(\tilde{\Lambda}^{\omega})=1\) holds for any \(x\in \R^d.\)

We define 
\[\Lambda:=\tilde{\Lambda}\cap \bigcap_{0< q\in \mathbb{Q}}\Lambda_{q}.\]
By Lemma \ref{PCAFprob1lem} and the above, for \(\mathbb{P}^X\)-a.s. \(\omega \in \Omega^X\), \(\mathbb{P}_x^{Z^{\infty}}(\Lambda^{\omega})=1\) holds for any \(x\in \R^d.\)

On \(\Omega\), we will define the PCAF \(\bar{A}\) having the same distribution as that of \(\tilde{A}\).

For any \(0<q\in \mathbb{Q}\) and \(s,t\) satisfying \(q\leq s\leq t,\) take \(T\in \mathbb{N}\) with \(t\leq T.\) For \((\omega, \omega')\in \Lambda,\) since \((\omega, \omega')\in \Lambda_q^T,\) there are \(n_i\to \infty\) such that 
\[\lim_{n_i\to \infty}(A^{n_i, \infty}_t (\omega, \omega')-A^{n_i, \infty}_s(\omega, \omega'))=\tilde{A}_{q,t}(\omega, \omega')-\tilde{A}_{q,s}(\omega, \omega')\] 
exists. We define this limit by \(\tilde{A}_{s,t}^{(T)}(\omega, \omega').\) Since \(\Lambda_0^{T_1}\subset\Lambda_0^{T_2}\) for \(T_2\leq T_1\), we have \(\tilde{A}_{s,t}^{(T_1)}(\omega, \omega')=\tilde{A}_{s,t}^{(T_2)}(\omega, \omega')\) for \(q\leq s \leq t\leq T_2 \leq T_1.\) So we can define 
\[\tilde{A}_{s,t}(\omega, \omega'):=\lim_{\mathbb{N} \ni T\nearrow \infty}\tilde{A}_{s,t}^{(T)}(\omega, \omega')\]
for \(q\leq s\leq t.\) Since \([0,T]\ni t \mapsto \tilde{A}_{0,t}^{(T)}(\omega, \omega')\in [0,\infty)\) is continuous and strictly increasing, so is \([s,\infty)\ni t \mapsto \tilde{A}_{s,t}(\omega, \omega')\in [0,\infty)\).

Moreover, for \(0\leq u\leq s \leq t\), we have
\begin{eqnarray}
\nonumber 0\leq \tilde{A}_{u,t}(\omega, \omega')-\tilde{A}_{s,t}(\omega, \omega')&=&\lim_{n_i\to \infty}(A_s^{n_i,\infty}(\omega, \omega')-A_u^{n_i, \infty}(\omega, \omega'))\\
&\leq &\varliminf_{n_i\to \infty}A_{s}^{n_i, \infty}(\omega, \omega'). \label{eq:PCAFconst1}
\end{eqnarray}

Since \((\omega,\omega')\in \tilde{\Lambda},\) by letting \(s\searrow 0\) at (\ref{eq:PCAFconst1}), \(\tilde{A}_{u,t}(\omega, \omega')-\tilde{A}_{s,t}(\omega, \omega')\) converges to \(0.\) This means that \(\{\tilde{A}_{s,t}(\omega, \omega')\}_{s\in [0,t]}\) is a Cauchy sequence, so there exists 
\begin{equation}\bar{A}_t(\omega, \omega'):=\lim_{s\searrow 0}\tilde{A}_{s,t}(\omega, \omega') \label{eq:barA}
\end{equation}
for \(t>0.\) We have
\begin{eqnarray}
\nonumber \bar{A}_t(\omega,\omega')-\bar{A}_s(\omega,\omega')&=&\lim_{u\searrow 0}(\tilde{A}_{u,t}(\omega, \omega')-\tilde{A}_{u,s}(\omega, \omega'))\\
\nonumber &=& \lim_{u\searrow 0}\lim_{n_i\to \infty}(A_t^{n_i, \infty}(\omega, \omega')-A_s^{n_i, \infty}(\omega, \omega'))\\
\nonumber &=& \tilde{A}_{0,t}(\omega, \omega')-\tilde{A}_{0,s}(\omega, \omega')\\
&=& \tilde{A}_{s,t}(\omega, \omega')
\label{eq:barAconti1}
\end{eqnarray}
for \(t,s>0\), and
\begin{eqnarray*}
0\ \leq \ \varlimsup_{t\searrow 0}\bar{A}_t(\omega,\omega')&=&\varlimsup_{t\searrow 0}\lim_{s\searrow 0} \lim_{n_i\to \infty} (A^{n_i, \infty}_{t}(\omega, \omega')-A^{n_i, \infty}_{s}(\omega, \omega'))\\
&\leq &\varlimsup_{t\searrow 0}\varliminf_{n_i\to \infty}A^{n_i, \infty}_{t}(\omega, \omega')\ =\ 0.
\end{eqnarray*}

So, by defining \(\bar{A}_0(\omega, \omega'):=0\), \(\bar{A}_{\cdot}(\omega, \omega')\) is continuous on \([0,\infty)\) and strictly increasing, and so is \(\bar{A}(\omega, \cdot)\) on \(\Lambda^{\omega}\) for \(\mathbb{P}^X\)-almost sure \(\omega\).

Since \(A_s^{n,\infty}\) is \(\mathscr{F}^X \otimes \mathscr{F}^{Z^{\infty}}_s\)-measurable, for \(\mathbb{P}^X\)-almost sure \(\omega\) and \(s\geq 0\), \(\bar{A}_s|_{\Lambda^{\omega}}\) is \(\mathscr{F}^{Z^{\infty}}_s\)-measurable.

For \(\mathbb{P}^X\)-almost sure \(\omega\), any \(\omega'\in \Lambda^{\omega}\) and \(t,s\geq 0\), there are \(n_i \to \infty\) such that, for any \(u\in [t,s+t],\) \(A^{n_i, \infty}_u(\omega, \omega')-A^{n_i,\infty}_t(\omega, \omega')\) converges to \(\tilde{A}_{t,u}(\omega, \omega')\). Then we have
\begin{equation}
0\leq \varliminf_{n_i\to \infty}A_s^{n_i,\infty}(\omega, \theta_t \omega')=\varliminf_{n_i\to \infty}(A_{s+t}^{n_i,\infty}(\omega, \omega')-A_t^{n_i,\infty}(\omega, \omega'))=\tilde{A}_{t,s+t}(\omega, \omega'). \label{eq:barAshift}
\end{equation}
By the continuity of \(\tilde{A}_{t,\cdot}(\omega, \omega')\) and \(\tilde{A}_{t,t}(\omega, \omega')=0\), by letting \(s \searrow 0\) at (\ref{eq:barAshift}), we have
\[\lim_{s\to 0}\varliminf_{n_i\to \infty}A_s^{n_i,\infty}(\omega, \theta_t \omega')=0.\] In the same way as (\ref{eq:barAshift}), we have 
\[\lim_{s\to \infty}\tilde{A}_s(\omega, \theta_t \omega')=\lim_{s\to \infty}\varlimsup_{n_i\to \infty}(A_{s+t}^{n_i,\infty}(\omega, \omega')-A_t^{n_i,\infty}(\omega, \omega'))=\infty. \]
So \(\theta_t\omega' \in \tilde{\Lambda}^{\omega}\) for any \(t\geq 0.\)

For \(\mathbb{P}^X\)-almost sure \(\omega\), any \(t\geq 0\), \(s\in \mathbb{Q}\) with \(s> 0\), \(\omega'\in \Lambda^{\omega}\), \(s\leq T\in \mathbb{N}\) and \(u\in [s,T]\), we have 
\begin{eqnarray}
\nonumber \lim_{n_i\to \infty}(A_u^{n_i,\infty}(\omega, \theta_t \omega')-A_s^{n_i,\infty}(\omega, \theta_t \omega')) &=& \lim_{n_i\to \infty}(A_{u+t}^{n_i,\infty}(\omega, \omega')-A_{s+t}^{n_i,\infty}(\omega, \omega'))\\
\nonumber &=& \tilde{A}_{s+t,u+t}(\omega, \omega')\\
&=& \bar{A}_{u+t}(\omega, \omega')-\bar{A}_{s+t}(\omega, \omega'), \label{eq:barAshift2}
\end{eqnarray}
so \(\tilde{A}_{s,u}(\omega, \theta_t \omega'):=\lim_{n_i \to \infty}(A_u^{n_i,\infty}(\omega, \theta_t \omega')-A_s^{n_i,\infty}(\omega, \theta_t \omega'))\) exist. Since \(\bar{A}_{\cdot}(\omega, \omega')\) is continuous and strictly increasing, \((\omega, \theta_t \omega')\in \Lambda_s\) holds. Thus we have \(\theta_t \omega' \in \Lambda^{\omega}\) for any \(t\geq 0\).

Furthermore, for \(\mathbb{P}^X\)-almost sure \(\omega\), \(t\geq 0\), \(\omega'\in \Lambda^{\omega}\), and \(u\geq 0\), we take \(T\in \mathbb{N}\) satisfying \(u+t\leq T.\) By letting \(s\searrow 0\) in (\ref{eq:barAshift2}), we have \(\bar{A}_u(\omega, \theta_t \omega')=\bar{A}_{u+t}(\omega, \omega')-\bar{A}_{t}(\omega, \omega').\) So \(\bar{A}(\omega,\cdot)\) is a PCAF on \(\Lambda^{\omega}\) for \(\mathbb{P}^X\)-a.s. \(\omega\).

Moreover, since \(\bar{A}_t=\lim_{s\searrow 0}\tilde{A}_{s,t}=\lim_{s\searrow 0}\lim_{n_i\to \infty}(A_{t}^{n_i,\infty}-A_{s}^{n_i,\infty})\) on \(\Lambda\) and \(\tilde{A}\) is the weak limit of \(A^{n, \infty}\), for any \(t\geq 0,\) \(\bar{A}_t\) has the same distribution as that of \(\tilde{A}_t\) under \(\mathbb{P}_x^{Z^{\infty}}\otimes \mathbb{P}^X\) for any \(x\in \R^d.\) Since \(\bar{A}\) and \(\tilde{A}\) are right-continuous \(\mathbb{P}_x^{Z^{\infty}}\otimes \mathbb{P}^X\)-a.s., \(\bar{A}\) and \(\tilde{A}\) have the same distribution by \cite[Problem 1.1.5]{KS}. Thus the proof of Proposition \ref{PCAFlem2} is completed.
\end{proof}

Next, we will prove that the PCAF \(\bar{A}\) corresponds to \(\mu^{\infty}\) and \(\bar{A}\) has the same distribution as that of \(A^{\infty}\).

The following lemma appeared in \cite[Appendix A]{CoF} without a proof. For the reader's convenience, we give a proof.
\begin{lemma}\label{Dapprox}
For \(f\in D([0,1];\R^d)\), there exist step functions \(f^{(n)}\) such that \(f^{(n)}\) converges to \(f\) uniformly on \([0,1].\)
\end{lemma}
\begin{proof}
We define \(0=t_0^{(n)}<t_1^{(n)}<\cdots < t_{N^{(n)}}^{(n)}=1\) as
\[\{t_i^{(n)}\}_{i=0}^{N^{(n)}}=\{0\leq t\leq 1\ ;\ |f(t)-f(t-)|\geq \frac{1}{n}\}\cup\{\frac{k}{2^n}\ ;\ 0\leq k \leq 2^n\}\]
 and the step function 
\[f^{(n)}:=\sum_{0 \leq i \leq N^{(n)}-1}f(t_i^{(n)}){\bf 1}_{[t_i^{(n)}, t_{i+1}^{(n)})}+f(t_{N^{(n)}}^{(n)}) {\bf 1}_{t_{N^{(n)}}^{(n)}}.\]
Remark that \(N^{(n)}\) is finite, and it holds that \(|t_{i+1}^{(n)}-t_{i}^{(n)}|\leq 1/2^n\) and \(\{t_{i}^{(n)}\}_{i=0}^{N^{(n)}}\subset \{t_{i}^{(n+1)}\}_{i=0}^{N^{(n+1)}}\). By the representation
\begin{eqnarray*}
\sup_{0\leq s\leq 1}|f(s)-f^{(n)}(s)|&=&\sup_{0\leq s\leq 1} \sum_{0 \leq i \leq N^{(n)}-1}|f(s)-f(t_i^{(n)})|{\bf 1}_{[t_i^{(n)}, t_{i+1}^{(n)})}(s),
\end{eqnarray*}
it is enough to prove that, for any \(\varepsilon>0,\) there exists \(n_0\) such that, for any \(n\geq n_0\) and \(t\in [t_{i}^{(n)}t_{i+1}^{(n)})\), \(|f(t)-f(t_{i}^{(n)})|\leq \varepsilon\) holds. Suppose this does not hold, that is, there are \(\varepsilon>0,\) \(n_k\nearrow \infty\), \(i=i(n_k)\) and \(t^{(n_k)}\in [t_{i}^{(n_k)}, t_{i+1}^{(n_k)})\) such that \(|f(t^{(n_k)})-f(t_i^{(n_k)})|>\varepsilon.\) Then we can take subsequence \(\{n_k'\}\subset \{n_k\}\) satisfying \(t^{(n_k')}\) converges to some \(t\in[0,1]\). Since it holds that
\(|t_{i}^{(n_k')}-t|\leq |t_{i}^{(n_k')}-t^{(n_k')}|+|t^{(n_k')}-t|\leq 1/2^{n_k'}+|t^{(n_k')}-t|,\) the sequence \(t_{i}^{(n_k')}\) also converges to \(t\). Furthermore, we can choose a further subsequence \(\{n_k''\}\subset \{n_k'\}\) satisfying \(\{t^{(n_k'')}\}_k\) and \(\{t_i^{(n_k'')}\}_k\) are monotone sequences.

If \(t^{(n_k'')} \nearrow t\) and \(t^{(n_k'')}_i \nearrow t\), we have 
\(\varepsilon< |f(t^{(n_k'')})-f(t^{(n_k'')}_i)|\to |f(t-)-f(t-)|=0.\)

If \(t^{(n_k'')} \searrow t\) and \(t^{(n_k'')}_i \searrow t\), we have \(\varepsilon< |f(t^{(n_k'')})-f(t^{(n_k'')}_i)|\to |f(t)-f(t)|=0.\)

Since \(t_i^{(n_k'')} \leq t^{(n_k'')}\), we do not have to consider the case of \(t^{(n_k'')} \nearrow t\) and \(t^{(n_k'')}_i \searrow t\).

If \(t^{(n_k'')} \searrow t\) and \(t^{(n_k'')}_i \nearrow t\) and there exists \(k_0\) such that \(t_i^{(n_{k_0}'')}=t\), then, for any \(k\geq k_0\), there is \(i\) satisfying \(t_i^{(n_{k}'')}=t\), so it holds that \[\varepsilon< |f(t^{(n_k'')})-f(t^{(n_k'')}_i)|=|f(t^{(n_k'')})-f(t)|\to |f(t)-f(t)|=0.\]

Otherwise, it holds that \(t^{(n_k'')} \searrow t\) and \(t^{(n_k'')}_i \nearrow t\), and, for any \(k\), there are \(\tilde{k}\geq k\)  and \(i\) satisfying \(t_i^{(n_{\tilde{k}}'')}<t\). For large \(n_{\tilde{k}}''\) such that \(|f(s)-f(s-)|\leq 1/n_{\tilde{k}}'' <\varepsilon/2\) for \(s\in (t_i^{(n_{\tilde{k}}'')}, t_{i+1}^{(n_{\tilde{k}}'')})\), we have \(\varepsilon< |f(t^{(n_{\tilde{k}}'')})-f(t^{(n_{\tilde{k}}'')}_i)|\to |f(t)-f(t-)|<\varepsilon/2.\)

These are contradictions.
\end{proof}

\begin{lemma}\label{DIntconv}
For any \(T>0\), \(f\in D([0,T]; \R)\) and non-decreasing right continuous functions \(g_n, g :[0, T]\to \R\), if \(g_n\) converges pointwise to \(g\) , then their Lebesgue-Stieltjes integrals \(\int fdg_n\) converge to \(\int fdg.\)
\end{lemma}

\begin{proof}
Without loss of the generality, we may assume \(T=1.\) By Lemma \ref{Dapprox}, for any \(\varepsilon>0,\) we can take \(k\) such that \(\sup_{0\leq t\leq 1}|f^{(k)}(t)-f(t)|\leq \varepsilon\), where \(f^{(k)}\) is the step function constructed in Lemma \ref{Dapprox}. Then we have
\begin{eqnarray}
\nonumber \lefteqn{\left|\int_0^1 fdg- \int_0^1 fdg_n\right|}\\
\nonumber  &\leq & \left|\int_0^1 fdg- \int_0^1 f^{(k)}dg\right| + \left|\int _0^1f^{(k)}dg- \int_0^1 f^{(k)}dg_n\right| + \left|\int_0^1 f^{(k)}dg_n-\int_0^1 fdg_n\right|\\
\nonumber  &\leq &\varepsilon \ (|g(1)-g(0)| + |g_n(1)-g_n(0)|)\\
&& \hspace{10mm}+ \sum_{i=0}^{N^{(k)}}|f(t_i^{(k)})| |(g(t_{i+1}^{(k)})-g(t_{i}^{(k)}))-(g_n(t_{i+1}^{(k)})-g_n(t_{i}^{(k)}))| \label{eq:DInt1}
\end{eqnarray}
Since \(\{t_i^{(k)}\}_i\) is finite for fixed \(k\), by letting \(n\) tends to infinity at (\ref{eq:DInt1}), we have
\[\left|\int_0^1 fdg- \int_0^1 fdg_n\right|\leq 2\varepsilon |g(1)-g(0)|.\] 
Thus the proof is complete.
\end{proof}

\begin{proposition}\label{PCAFlem3}
The PCAF \(\bar{A}\) corresponds to \(\mu^{\infty}\).
\end{proposition}

\begin{proof}
By Lemma \ref{PCAFlem2}, \(\bar{A}\) is a PCAF whose distribution is same as that of \(\tilde{A}\). So there exists the Revuz measure \(\bar{\mu}(\omega)\) corresponding to \(\bar{A}( \cdot, \omega)\) for \(\mathbb{P}^X\)-almost sure \(\omega\). We will prove that \(\mu^{\infty}(\omega)=\bar{\mu}(\omega)\) for \(\mathbb{P}^X\)-almost sure \(\omega\).

Set \(d\mu^{n,\infty}:=\exp(\gamma X^n(x)-\frac{\gamma^2}{2} \mathbb{E}^X[(X^n(x))^2])dm.\)
Then, \(A^{n,{\infty}}(\omega,\cdot)\) corresponds to \(\mu^{n,\infty}(\omega)\) for \(\mathbb{P}^X\)-almost sure \(\omega\). So, for \(f\in C_c(\R^d)\) with \(f\geq 0\) and \(h\in L^1(\R^d)\cap L^{\infty}(\R^d)\) with \(h\geq 0\), we have
\begin{equation}
\int_{\R^d} \E_{x}^{Z^{\infty}}\left[\int_0^t f(Z_s^{\infty})dA_s^{n,{\infty}}\right] h(x) dm(x) =\int_{\R^d} \int_0^t P_s^{\infty}h(x)ds \cdot f(x)d\mu^n(x) \label{eq:PCAFlem3-1}
\end{equation}
\(\mathbb{P}^X\)-almost surely. Here, we defined \(P_s^{\infty}h(x):=\E_x^{Z^{\infty}}h(Z_s^{\infty}).\) By the strong Feller property of \(Z^{\infty}\), \(\int_0^t P_s^{\infty}hds \in C(\R^d)\) and so \(\int_0^t P_s^{\infty}hds\cdot f \in C_c(\R^d)\). Thus, by Theorem \ref{Shamov} and (\ref{eq:PCAFlem3-1}), we have
\begin{equation}
\lim_{n\to {\infty}}\EX\left| \int_{\R^d} \E_{x}^{Z^{\infty}}\left[\int_0^t f(Z_s^{\infty})dA_s^{n,{\infty}}\right] h(x) dm(x)- \int_{\R^d} \int_0^t P_s^{\infty}hds \cdot fd\mu^{\infty} \right|=0.\label{eq:PCAFlem3-2}
\end{equation}

On the other hand, since \(f(Z^{\infty}_{\cdot})\in D([0,t];\R)\), \(\mathbb{P}_x^{Z^{\infty}}\)-almost surely, and \(A_t^{n,{\infty}}\) converges to \(\tilde{A}_t\) in \(L^1(\mathbb{P}^X)\), \(\mathbb{P}^{Z^{\infty}}_x\)-almost surely, by Lemma \ref{DIntconv} and Lebesgue's convergence theorem, we have
\begin{eqnarray*}
 \EX\left| \int_{\R^d} \E_{x}^{Z^{\infty}}\left[\int_0^t f(Z_s^{\infty})dA_s^{n,{\infty}}\right] h(x) dm(x) - \int_{\R^d} \E_{x}^{Z^{\infty}}\left[\int_0^t f(Z_s^{\infty})d\tilde{A}_s\right] h(x) dm(x)\right|\\
\leq \int_{\R^d} \E_{x}^{Z^{\infty}}\mathbb{E}^X\left|\int_0^t f(Z_s^{\infty})dA_s^{n,{\infty}}- \int_0^t f(Z_s^{\infty})d\tilde{A}_s\right| h(x) dm(x),
\end{eqnarray*}
and this converges to \(0\) as \(n \to \infty\). Combining this with (\ref{eq:PCAFlem3-2}) and Lemma \ref{PCAFlem2}, we have
\begin{equation}
\int_{\R^d} \E_{x}^{Z^{\infty}}\left[\int_0^t f(Z_s^{\infty})d\bar{A}_s\right] h(x) dm(x) =\int_{\R^d} \int_0^t P_s^{\infty}h(x)ds \cdot f(x)d\mu^{\infty}(x),\label{eq:PCAFlem3-4}
\end{equation}
\(\mathbb{P}^X\)-almost surely, for \(f\in C_c(\R^d)\) and \(h\in L^1(\R^d)\cap L^{\infty}(\R^d)\) with \(f,h\geq 0.\)

Equation (\ref{eq:PCAFlem3-4}) also holds \(\mathbb{P}^X\)-almost surely for a non-negative measurable function \(h\) and \(f\in C_c(\R^d)\) with \(h\geq 0\) by using the monotone convergence theorem for \(h_n:=(h\wedge n){\bf 1}_{[0,n]}\).
By taking \(h=1\), it follows from (\ref{eq:PCAFlem3-4}) that
\begin{equation}
\int_{\R^d} f(x)d\mu^{\infty}(x)=\lim_{t\searrow 0}\frac{1}{t}\int_{\R^d} \E_{x}^{Z^{\infty}}\left[\int_0^t f(Z_s^{\infty})d\bar{A}_s \right] dm(x), \label{eq:PCAFlem3-5}
\end{equation}
\(\mathbb{P}^X\)-almost surely, for \(f\in C_c(\R^d)\) with \(f\geq 0.\) 
Combining this with the correspondence between \(\bar{A}\) and \(\bar{\mu}\), 
\[\int fd\mu^{\infty} = \int fd\bar{\mu}\]
 holds \(\mathbb{P}^X\)-almost surely. Thus, for each open set \(O\), there exists \(\Lambda^O\) with \(\mathbb{P}^X(\Lambda^O)=1\) such that \(\mu^{\infty}(\omega)(O) = \bar{\mu}(\omega)(O)\) for \(\omega \in \Lambda^O.\) By taking a countable open basis \(\{O_i\}_{i=1}^{\infty}\), it holds that 
 \[\mu^{\infty}= \bar{\mu} \text{\ \ on\ } \cap_{i=1}^{\infty} \Lambda^{O_i}\]
  since \(\mu^{\infty}\) and \(\bar{\mu}\) are Borel measures. Thus \(\bar{A}(\omega,\cdot)\) corresponds to \(\mu^{\infty}(\omega)\) for \(\mathbb{P}^X\)-almost sure \(\omega\).
\end{proof}

\begin{proof}[Proof of the Proposition \ref{PCAFlem}]
By Lemma \ref{PCAFlem2} and \ref{PCAFlem3}, \(\bar{A}\) has the same distribution as that of \(A^{\infty}\).
\end{proof}

In this section, we did not use specific properties of \(A^{n,\infty}\) for \(Z^n\). Thus, by the same proof as that of Proposition \ref{PCAFlem}, we have the following general statement of independent interest.

\begin{theorem}
Let \((\mathcal{E},\mathcal{F})\) on \(L^2(E;dm)\) be a regular Dirichlet form and \(Z\) be the \(m\)-symmetric Hunt process associated with this Dirichlet form. Denote \(A^n\) by a PCAF corresponding the smooth measure \(\mu^n\) for \(n\in \mathbb{N}\cup\{\infty\}.\) Assume that \(Z\) is the strong Feller process, \(\mu^n\) converges vaguely to \(\mu^{\infty}\) , \(A^n_t\) converges to some \(\tilde{A}_t\) in \(L^1(\mathbb{P}_x^{Z})\) for any \(x\in E\) and \(t\), and there exists \(C_t\in L^1(0,T)\) such that \(\mathbb{E}_x[A_t^n]\leq C_t\) for any \(T\). Then \(\tilde{A}\) has the same distribution as that of \(A^{\infty}\).
\end{theorem}

\section{Examples}\label{examples}
In this section, we give two examples of the main result, Theorem \ref{mainconvergence}. In the first one, we introduce Liouville \(\alpha\)-stable process and prove the convergence of these processes. In the second one, we prove the scaling limit of Liouville simple random walk on \(\mathbb{Z}^2\) is Liouville Brownian motion.

\subsection{Convergence of the Liouville \(\alpha\)-stable processe on \(\R\)}\label{example1}
Let \(d=1\). For \(\alpha \in[1,2)\), denote by \(Z^{\alpha}=(\{Z^{\alpha}_t\}_t, \{\mathbb{P}_x^{Z^{\alpha}}\}_x)\) a symmetric \(\alpha\)-stable process on \(\R\) on a probability space \(\Omega^{Z^{\alpha}}\). Denote by \(p^{\alpha}(t,x,y)\) the continuous transition density function of \(Z^{\alpha}\) and \(g_{\lambda}^{\alpha}(x,y):=\int_0^{\infty} e^{-\lambda t}p^{\alpha}(t,x,y)dt\) a \(\lambda\)-order Green's function of \(Z^{\alpha}\) for \(\lambda>0.\) Let \(X^{\alpha}\) be a Gaussian field on \(\R\) on a probability space \((\Omega^{X^{\alpha}}, \mathbb{P}^{X^{\alpha}})\) whose covariance kernel is \(\pi g_{\lambda}^{\alpha}\). By Lemma \ref{Green'slemma} and \cite{B}, we can construct a non-degenerate GMC \(\mu^{\alpha}\) for any \(\alpha \geq d=1\) and \(\gamma \in [0,\sqrt{2})\). We define \(\hat{Z}^{\alpha}\) as the time-changed process of \(Z^{\alpha}\) by \(\mu^{\alpha}\) for any  \(\alpha \geq 1\). In particular, \(\hat{Z}^1\) is Liouville Cauchy process on \(\R\). See \cite{Bav} for details of the construction of Liouville Cauchy process.

\begin{theorem}\label{example1thm}
For \(\gamma \in [0,1)\) and any \(x\in \R\), time-changed process \(\hat{Z}^{\alpha}\) under \(\mathbb{P}_x^{Z^{\alpha}} \otimes \mathbb{P}^{X^{\alpha}}\) converges weakly to \(\hat{Z}^1\) under \(\mathbb{P}_x^{Z^1} \otimes \mathbb{P}^{X^1}\) with respect to \(J_1\)-topology as \(\alpha \searrow 1\).
\end{theorem}

\begin{remark}
For \(\alpha \in (0,2)\) and positive integer \(d\), we can also consider symmetric \(\alpha\)-stable processes \(Z^{\alpha}\) on \(\R^d\) and Gaussian fields whose covariance kernel are \(\pi\) times \(\lambda\)-order Green's functions \(g_{\lambda}^{\alpha}\) of \(Z^{\alpha}\).  However, we cannot construct non-degenerate GMC \(\mu^{\alpha}\) for any \(\alpha<d\) because \(\pi g_{\lambda}^{\alpha}(x,y)\) cannot be bounded by \(\log_{+}{\frac{1}{|x-y|}}+C\) uniformly. So we only consider cases of \(d=1,2\) and \(\alpha \in [d,2]\). Moreover, if \(d=2\), there is the only case of \(\alpha=2\), so we assumed \(d=1\).
\end{remark}

\begin{remark}
For \(d=\alpha=2\), \(X^{\alpha}\) is the massive Gaussian free field on \(\R^2\). Hence, for \(\alpha \geq d,\) \(X^{\alpha}\), \(\mu^{\alpha}\) and \(\hat{Z}^{\alpha}\) are generalizations to \(\alpha\)-stable cases of the massive Gaussian free field, the Liouville measure and Liouville Brownian motion, respectively. For convenience, we call \(\hat{Z}^{\alpha}\) the \textit{Liouville \(\alpha\)-stable process}. In particular, we call \(\hat{Z}^1\) \textit{Liouville Cauchy process} on \(\R\). Theorem \ref{example1thm} means that Liouville \(\alpha\)-stable process converges weakly to Liouville Cauchy process as \(\alpha \searrow 1\).
\end{remark}

\begin{remark}
For positive integer \(d\) and \(\alpha \in (d,2)\), symmetric \(\alpha\)-stable process \(Z^{\alpha}\) on \(\R^d\) is strongly recurrent, that is, \(Z^{\alpha}\) has positive capacity for a point. Then \(X^{\alpha}\) has a bounded covariance kernel, so \(X^{\alpha}\) is not only a random distribution, but also a random function for \(\alpha>d\). Moreover \(\mu^{\alpha}\) is absolutely continuous with respect to the Lebesgue measure, \(\mathbb{P}^{X^{\alpha}}\)-almost surely for \(\alpha >d\). On the other hand, \(Z^d\) is recurrent simply, so \(X^d\) does not have a bounded covariance kernel, \(X^d\) is not a random function and \(\mu^{\alpha}\) is singular with respect to the Lebesgue measure, \(\mathbb{P}^{X^d}\)-almost surely. Properties change significantly when \(\alpha=d\), so we consider the convergence to \(\alpha=d\).
\end{remark}

\begin{remark}
Since \(Z^{\alpha}\) on \(\R\) is strongly recurrent for \(\alpha>1\), by \cite{CHK}, \(\hat{Z}^{\alpha}\) also converges to \(\hat{Z}^{\alpha_0}\) as \(\alpha \to \alpha_0\) for \(\alpha_0>1\).
\end{remark}

\begin{proof}[Proof of Theorem \ref{example1thm}]
By using the Fourier inversion theorem and Euler's reflection formula, we have \(g_{\lambda}^{\alpha}(0,0)=\frac{1}{2\alpha \sin(\pi/\alpha)}\), so  \(Z^{\alpha}\) is strongly recurrent if and only if \(\alpha>d=1\).

By Theorem \ref{mainconvergence}, it is enough to check Assumption \ref{assumption} (1) and (2) with \(C^*=1\). Fix \(d=1\), \(\gamma \in [0,1)\) and \(x\in \R.\)

(1) By considering the characteristic function, \(Z^{\alpha}_t\) converges weakly to \(Z^1_t\) as \(\alpha \to 1\) for any \(t\geq 0\). Since \(Z^{\alpha}\) are L\'evy processes,  by \cite[VI. 17. Exercise 15.]{Ka}, \(Z^{\alpha}\) converges weakly to \(Z^1\) as \(\alpha \to 1\) with respect to \(J_1\)-topology.

By the Fourier inversion theorem, we have
\[p^{\alpha}(t,x,y)=\frac{1}{2\pi}\int_{\R}e^{i\theta (x-y)}e^{-\frac{t}{2}|\theta|^{\alpha}}d\theta.\]
Since  \(|e^{i\theta (x-y)}e^{-\frac{t}{2}|\theta|^{\alpha}}|\leq e^{-\frac{t}{2}|\theta|}+e^{-\frac{t}{2}|\theta|^2},\) by Lebesgue's convergence theorem, \(p^{\alpha}(t,x,y)\) converges pointwisely to \(p^1(t,x,y)\). By the two sided heat kernel estimates for \(p^{\alpha}\) (see \cite{BG} for example), there exists \(C>0\) such that, for any \(\alpha\in [1,2)\), \(t\geq 0,\) and \(x,y\in \R\), it holds that 
\begin{equation}
p^{\alpha}(t,x,y) \leq C(p^0(t,x,y)+p^{1/2}(t,x,y)). \label{eq:ex1-4}
\end{equation}
The right hand side of \((\ref{eq:ex1-4})\) times \(e^{-\lambda t}\) is integrable on \([0,\infty)\) for each \(x,y\in \R\) with \(x\neq y\). Hence, by Lebesgue's convergence theorem, \(g_{\lambda}^{\alpha}(x,y)\) converges pointwise to \(g_{\lambda}^{1}(x,y)\) as \(\alpha \searrow 1.\)

(2) Since \(g_{\lambda}^{\alpha}(0,0)=\frac{1}{2\alpha \sin(\pi/\alpha)}\) is increasing to \(g_{\lambda}^{1}(0,0)=\infty\) as \(\alpha \searrow 1\) and \(g_{\lambda}^{1}(x,y)\) is also increasing to \(g_{\lambda}^{1}(0,0)\) as \(|x-y|\searrow 0\), by the continuity of \(g_{\lambda}^{\alpha}(x,y)\), there exists small \(\delta>0\) such that \(g_{\lambda}^{\alpha}(x,y)\leq g_{\lambda}^{0}(x,y)\) for \(|\alpha -1| \vee |x-y|\leq \delta .\) While, it holds that
\begin{eqnarray}
\nonumber g_{\lambda}^{\alpha}(x,y) &=&  \int_0^{\infty} \frac{e^{-\lambda t}}{2\pi} \int_{\R} e^{i\theta (x-y)}e^{-\frac{t}{2}|\theta|^{\alpha}} d\theta dt\\
\nonumber &=& \frac{1}{2\pi} \int_{\R} \frac{e^{i\theta (x-y)}}{\lambda + \frac{|\theta|^{\alpha}}{2}} d\theta \\
\nonumber &=& \frac{1}{\pi} \int_{0}^{\infty} \frac{1}{\lambda + \frac{\theta^{\alpha}}{2}} \cos(|x-y|\theta) d\theta\\
\nonumber &=& \frac{1}{\pi}\left[\frac{1}{\lambda + \frac{\theta^{\alpha}}{2}} \frac{\sin(|x-y|\theta)}{|x-y|}\right]_0^{\infty}+\frac{1}{2\pi} \int_{0}^{\infty} \frac{\alpha \theta^{\alpha-1}}{(\lambda + \frac{\theta^{\alpha}}{2})^2} \frac{\sin(|x-y|\theta)}{|x-y|} d\theta \\
\nonumber &=& \frac{1}{2\pi} \int_{0}^{\infty} \frac{\alpha \theta^{\alpha-1}}{(\lambda + \frac{\theta^{\alpha}}{2})^2} \frac{\sin(|x-y|\theta)}{|x-y|} d\theta .
\end{eqnarray}
So we have
\begin{eqnarray}
\nonumber |g_{\lambda}^{\alpha}(x,y)|&\leq &  \frac{1}{2\pi|x-y|} \int_{0}^{\infty} \frac{\alpha \theta^{\alpha-1}}{(\lambda + \frac{|\theta|^{\alpha}}{2})^2} d\theta \\
\nonumber &\leq & \frac{1}{2\pi \lambda^2 |x-y|} \int_{0}^{1} \alpha \theta^{\alpha-1} d\theta +\frac{1}{2\pi|x-y|} \int_{1}^{\infty} \frac{\alpha \theta^{\alpha-1}}{(\frac{\theta^{\alpha}}{2})^2} d\theta\\
&=& \frac{\frac{1}{\lambda^2}+4}{2\pi} \frac{1}{|x-y|}.  \label{eq:ex1-3}
\end{eqnarray}
Thus we have \(|g_{\lambda}^{H}(x,y)|\leq C\) for \(|x-y|\geq \delta\), and \(g_{\lambda}^{\alpha}(x,y) \leq g_{\lambda}^{1}(x,y)+C \) for \(x,y\in \mathbb{R}.\)

\end{proof}

\subsection{Scaling limit of Liouville simple random walks on \(\mathbb{Z}^2\)}\label{example2}
Let \(d=2\), \(Z^{\infty}\) be Brownian motion on \(\R^2\) and \(Z^n\) be a continuous-time simple random walk on \(D_n:=\frac{1}{\sqrt{n}}\mathbb{Z}^2\). More precisely, \(Z^n\) is defined as follows. Let \(\{S_i\}_{i=1}^{\infty}\) be a simple symmetric random walk on \(\mathbb{Z}^2\). For convenience, let \(S_0:=(0,0)\). Denote by \(\{N^n\}_n\) an independent Poisson processes with rates \(n\). We remark that 
\[\mathbb{P}(N_t^n=k)=\frac{(nt)^k}{k!}e^{-nt}\]
for \(t\geq 0\) and \(k\in \mathbb{N}\cup \{0\}.\) For \(n\in \mathbb{N}\), we define the continuous-time simple random walk \(Z^n\) on \((\Omega^{Z^n}, \{\mathbb{P}_{x}^{Z^n}\}_{x\in \R^2})\) by
\[Z^n_t:=\frac{1}{\sqrt{n}}S_{N^n_t} + x_n\]
for \(x\in \R^2\) and any \(x_n \in \frac{\mathbb{Z}^2}{\sqrt{n}}\) satisfying \(|x-x_n|< \sqrt{\frac{2}{n}}\). We remark that \(Z^n\) are L\'evy processes, see \cite[Exercise 2.2.5]{CF} for example. In particular, since L\'evy process with no drift is an \(m\)-symmetric Hunt process for the Lebesgue measure \(m\).

The transition density functions \(p^n\) and the \(\lambda\)-order Green's functions \(g_{\lambda}^n\) of \(Z^n\) are written as \(p^n(t,x,y)=\mathbb{P}_{x_n}(Z_t^n=y_n)\) and \(g_{\lambda}^n(x,y)= n \int_0^{\infty}e^{-\lambda t}p^n(t,x,y)dt\) for \(n\in \mathbb{N}, \lambda>0\) and \(x,y\in \R^2.\)

Denote by \(X^n\) the centred Gaussian field on \(\R^2\) on a probability space \((\Omega^{X^n},\mathbb{P}^{X^n})\) having the covariance kernel \(\pi g_{\lambda}^n\), and \(\mu^n\) by its GMC with \(\gamma\in [0,2)\) for \(n\in \mathbb{N}\cup \{\infty\}\). Let \(\hat{Z}^n\) be the time-changed process of \(Z^n\) by \(\mu^n\). For \(n\in \mathbb{N}\), we call \(\hat{Z}^n\) \textit{Liouville simple random walk} on \(\frac{1}{\sqrt{n}}\mathbb{Z}^2\).

Then the following theorem holds.
\begin{theorem}\label{example2thm}
For any \(x\in \R^2\) and \(\gamma \in [0,\sqrt{2})\), time-changed processes \(\hat{Z}^n\) under \(\mathbb{P}_{x}^{Z^n} \otimes \mathbb{P}^{X^n}\) converge weakly to Liouville Brownian motion \(\hat{Z}^{\infty}\) under \(\mathbb{P}_{x}^{Z^{\infty}} \otimes \mathbb{P}^{X^{\infty}} \) with the local uniform topology.
\end{theorem}

\begin{remark}
By restricting to \(\frac{1}{\sqrt{n}}\mathbb{Z}^2\), \(X^{n}\) is the discrete version on \(\frac{1}{\sqrt{n}}\mathbb{Z}^2\) of the massive Gaussian free field on \(\R^2\). So, Theorem \ref{example2thm} means that the scaling limit of the Liouville simple random walk on \(\mathbb{Z}^2\) is Liouville Brownian motion.
\end{remark}

\begin{proof}[Proof of Theorem \ref{example2thm}]
By the continuity of addition with Skorokhod's topology \cite[Proposition VI. 1. 23]{JS}, without loss of generality, we may use \(x_n \in \frac{\mathbb{Z}^2}{\sqrt{n}}\), which is defined as follows. For any \(x=(x^{(1)},x^{(2)})\in \R^2\), we define \[i_n(x):=x_n:=\left(\frac{\lfloor \sqrt{n} x^{(1)}\rfloor}{\sqrt{n}}, \frac{\lfloor \sqrt{n} x^{(2)}\rfloor}{\sqrt{n}}\right),\]
where \(\lfloor a \rfloor\) is the largest integer less than or equal to \(a.\) Then, \(x_n\) converges to \(x\) and \(C_n=1/m(\{y\in \R^2 : y_n=x\})=n\) for any \(x\in \R^2.\)

We remark that \(J_1\)-topology is the same as the local uniform topology when the limit is continuous. By Theorem \ref{mainconvergence}, it is enough to check Assumption \ref{assumption} (1) and (2) with \(C^*=1\). Suppose \(\gamma \in [0,\sqrt{2})\) and fix \(x\in \R^2.\)

(1) Let \(Z^{\infty}\) be Brownian motion on \(\R^2.\) Then, the weak convergence of \(Z^n_t\) to \(Z^{\infty}_t\) for any \(t\) follows from Donsker's theorem immediately.
Since \(Z^n\) are L\'{e}vy processes, by \cite[VI. 17. Exercise 15.]{K}, \(Z^n\) converges weakly to \(Z^{\infty}\) as \(n \to \infty\) with the local uniform topology.

Next, we prove \(g_{\lambda}^{n}(x,y)\) converges pointwise to \(g_{\lambda}^{\infty}(x,y)\), the \(\lambda\)-order Green's function of \(Z^{\infty}\), as \(n  \to \infty\). By the definition of \(Z^n\), we have
\begin{equation}
p^n(t,x_n,y_n)=\sum_{k=0}^{\infty} \mathbb{P}(S_k=\lfloor \sqrt{n} y \rfloor -\lfloor \sqrt{n}x \rfloor) \frac{(nt)^{k}}{k!}e^{-nt}, \label{eq:HKPoisson}
\end{equation}
where \(\lfloor \sqrt{n} x \rfloor :=(\lfloor \sqrt{n} x^{(1)} \rfloor, \lfloor \sqrt{n} x^{(2)} \rfloor )\) and \(\lfloor \sqrt{n} y \rfloor :=(\lfloor \sqrt{n} y^{(1)} \rfloor, \lfloor \sqrt{n} y^{(2)} \rfloor ).\) For \(x,y\in \R^2\) with \(x\neq y,\) we can take large \(n\in \mathbb{N}\) such that \((x-y)_n\neq 0\). Then, by the local central limit theorem \cite[Theorem 2.3.10, 2.3.11]{LL} and the super-additivity of the floor function, we have
\begin{eqnarray}
\nonumber \lefteqn{|np^n(t,x_n,y_n)- \mathbb{E}[p^{\infty}(N_{t}^n/n,x_n,y_n)] |}\\
\nonumber & \leq &\sum_{k=0}^{\infty} \left|n \mathbb{P}(S_k=\lfloor \sqrt{n} x \rfloor-\lfloor \sqrt{n} y \rfloor)-p^{\infty}(k/n,x_n,y_n)\right|\frac{(nt)^k}{k!}e^{-nt}\\
\nonumber &\leq & \left( \sum_{k=1}^{n}\frac{n}{|\lfloor \sqrt{n} (x-y) \rfloor|^2}\frac{C}{k} +\sum_{k=n+1}^{\infty} \frac{Cn}{|\lfloor \sqrt{n} (x-y) \rfloor|^4} \right) \frac{(nt)^k}{k!}e^{-nt}\\
&\leq & \left( \sum_{k=1}^{n}\frac{1}{|(x-y)_n|^2}\frac{C}{k} +\sum_{k=n+1}^{\infty}\frac{C}{n|(x-y)_n|^4}  \right) \frac{(nt)^k}{k!}e^{-nt} \label{eq:ex2-2-0} \\
&\leq & \frac{C}{nt|(x-y)_n|^2}  +\frac{C}{n|(x-y)_n|^4}\label{eq:ex2-2-1}
\end{eqnarray}
for  \(x,y\in \R^2\) with \(x\neq y\). By checking the convergence of characteristic functions, \(N_{t}^n/n\) converges weakly to \(t\) as \(n\to \infty\) for any \(t\geq 0.\) For \(x,y\in \R^2\) with \(x\neq y\) and large \(n\), we have
\[\sup_{0\leq t}p^{\infty}(t, x_n,y_n)\leq \frac{1}{\pi e |(x-y)_n|^2}\leq \frac{C}{|x-y|^2}.\]
So \(p^{\infty}(\cdot, x_n,y_n)\) is bounded continuous function for each \(x,y\) and large \(n.\) Thus, by letting \(n\) tends to \(\infty\) in \((\ref{eq:ex2-2-1})\), we have
\begin{equation}
\lim_{n\to \infty}np^n(t,x_n,y_n)=p^{\infty}(t,x,y) \label{eq:ex2-2-2}
\end{equation} 
for \(x,y\in \R^2\) with \(x\neq y.\)

For \(x,y \in \R^d\) satisfying \(x\neq y\) and large \(n\), by \((\ref{eq:ex2-2-0})\), we have
\begin{eqnarray}
\nonumber \lefteqn{\left|g_{\lambda}^n(t,x_n,y_n)- \int_0^{\infty}e^{-\lambda t} \mathbb{E}[p^{\infty}(N_{t}^n/n,x_n,y_n)]dt \right|}\\
\nonumber &\leq & \sum_{k=1}^{n}\frac{1}{|(x-y)_n|^2}\frac{C}{k}\int_0^{\infty}e^{-\lambda t} \frac{(nt)^k}{k!}e^{-nt}dt+\sum_{k=n+1}^{\infty}\frac{C}{n|(x-y)_n|^4}\int_0^{\infty}e^{-\lambda t} \frac{(nt)^k}{k!}e^{-nt}dt \\
\nonumber &= &  \sum_{k=1}^{n} \frac{1}{|(x-y)_n|^2}\frac{C}{k}\frac{n^{k}}{(n+\lambda)^{k+1}}+\sum_{k=n+1}^{\infty} \frac{C}{n|(x-y)_n|^4}\frac{n^{k}}{(n+\lambda)^{k+1}}\\
\nonumber &\leq &  \sum_{k=1}^{n} \frac{1}{|\lfloor \sqrt{n} (x-y) \rfloor|^2}\frac{C}{k}+ \frac{C}{|\lfloor \sqrt{n} (x-y) \rfloor|^4}\sum_{k=n+1}^{\infty} \left(\frac{n}{n+\lambda}\right)^{k+1}\\
\nonumber &\leq & C\frac{1+\log n}{|\lfloor \sqrt{n} (x-y) \rfloor|^2} +C \frac{n}{|\lfloor \sqrt{n} (x-y) \rfloor|^4}\\
&\leq & C\frac{1+\log n}{n|x-y|^2} +C \frac{1}{n|x-y|^4}, \label{eq:2-2-3}
\end{eqnarray}
and so 
\begin{equation}
\lim_{n\to \infty} \left|g_{\lambda}^n(t,x_n,y_n)- \int_0^{\infty}e^{-\lambda t} \mathbb{E}[p^{\infty}(N_{t}^n/n,x_n,y_n)]dt \right|=0. \label{eq:2-2-4}
\end{equation}
By the weak convergence of \(N_{t}^n/n\) to \(t\) and the same technique in the proof of (\ref{eq:2-2-3}), \(\mathbb{E}[p^{\infty}(N_{t}^n/n,x_n,y_n)]\) converges to \(p^{\infty}(t,x,y)\). Combining this with Lebesgue's convergence theorem, \(\int_0^{\infty}e^{-\lambda t} \mathbb{E}[p^{\infty}(N_{t}^n/n,x_n,y_n)]dt\) converges to \(g_{\lambda}^{\infty}(x,y).\)

(2) Since \(p^n(t,x_n,y_n)=\sum_{k=0}^{\infty} \mathbb{P}_0(S_k/\sqrt{n}=y_n-x_n)\frac{(nt)^k}{k!}e^{-nt}\), by the local central limit theorem \cite[Theorem 2.3.5]{LL}, we have
\begin{eqnarray}
\nonumber \left| g_{\lambda}^n(x_n,0)-\int_0^{\infty}ne^{-\lambda t}\sum_{k=0}^{\infty}\frac{e^{-\frac{|x_n|^2}{2k/n}}}{2\pi k} \frac{(nt)^k}{k!}e^{-nt} dt \right| &\leq &C \int_0^{\infty} \sum_{k=1}^{\infty} \frac{ne^{-\lambda t} }{k^2} \frac{(nt)^k}{k!}e^{-nt}dt \\
\nonumber &=&\sum_{k=1}^{\infty} \frac{n}{k^2} \frac{n^k}{(n+\lambda)^{k+1}}\\
 & \leq & \sum_{k=1}^{\infty} \frac{1}{k^2} < \infty \label{eq:newex-a}
\end{eqnarray}
Moreover, we have
\begin{eqnarray*}
\int_0^{\infty}ne^{-\lambda t}\sum_{k=0}^{\infty}\frac{e^{-\frac{|x_n|^2}{2k/n}}}{2\pi k} \frac{(nt)^k}{k!}e^{-nt} dt &=& \sum_{k=1}^{\infty} \frac{e^{-\frac{|x_n|^2}{2k/n}}}{2\pi k/n}\frac{n^k}{(n+\lambda)^{k+1}}\\
&=& \int_{1/n}^{\infty} \frac{e^{-\frac{|x_n|^2}{2\lfloor nt \rfloor /n}}}{2\pi \lfloor nt \rfloor /n}\left(1-\frac{\lambda}{n+\lambda} \right)^{\lfloor nt \rfloor +1} dt\\
&=& \int_{1/n}^{\infty} \frac{e^{-\frac{|x_n|^2}{2\lfloor nt \rfloor /n}}}{2\pi \lfloor nt \rfloor /n}\left(1-\frac{\lambda}{n+\lambda} \right)^{\frac{n+\lambda}{\lambda} \frac{n}{n+\lambda}\frac{\lfloor nt \rfloor +1}{n}\lambda} dt.
\end{eqnarray*}
Since \(e\leq (1-1/y)^{-y}\) for any \(y>0\), we have
\begin{eqnarray}
\nonumber \int_0^{\infty}ne^{-\lambda t}\sum_{k=0}^{\infty}\frac{e^{-\frac{|x_n|^2}{2k/n}}}{2\pi k} \frac{(nt)^k}{k!}e^{-nt} dt &\leq & \int_{1/n}^{\infty} \frac{e^{-\frac{|x_n|^2}{2\lfloor nt \rfloor /n}}}{2\pi \lfloor nt \rfloor /n}e^{-\frac{n}{n+\lambda}\frac{\lfloor nt \rfloor +1}{n}\lambda} dt\\
&\leq & \int_{1/n}^{\infty} \frac{e^{-\frac{|x_n|^2}{2t}}}{2\pi \lfloor nt \rfloor /n}e^{-\frac{n}{n+\lambda}\lambda t} dt. \label{eq:newex-b}
\end{eqnarray}
By computation, we have
\begin{eqnarray}
\nonumber \lefteqn{ \left|  \int_{1/n}^{\infty} \frac{e^{-\frac{|x_n|^2}{2\lfloor nt \rfloor /n}}}{2\pi \lfloor nt \rfloor /n}e^{-\frac{n}{n+\lambda}\lambda t} dt- \int_{1/n}^{\infty} \frac{e^{-\frac{|x_n|^2}{2t}}}{2\pi t}e^{-\lambda t} dt \right|}\\
\nonumber &\leq & \sum_{k=1}^{\infty} \int_{\frac{k}{n}}^{\frac{k+1}{n}} \frac{e^{-\frac{|x_n|^2}{2t}}}{2\pi }\left|\frac{1}{k/n}-\frac{1}{t}  \right| e^{-\frac{n}{n+\lambda}\lambda t}dt + \int_{\frac{1}{n}}^{\infty} \frac{e^{-\frac{|x_n|^2}{2t}}}{2\pi t}\left|e^{-\frac{n}{n+\lambda}\lambda t}-e^{-\lambda t}  \right| dt \\
\nonumber &\leq & \sum_{k=1}^{\infty} \int_{\frac{k}{n}}^{\frac{k+1}{n}} \frac{t-k/n}{2\pi (k/n)^2 }e^{-\frac{|x_n|^2}{2t}} e^{-\frac{n}{n+\lambda}\lambda t}dt + \int_{\frac{1}{n}}^{\infty} \frac{e^{-\frac{|x_n|^2}{2t}}e^{-\frac{n}{n+\lambda}\lambda t}}{2\pi t}\left(1-e^{-\frac{\lambda^2}{n+\lambda} t} \right) dt \\
\nonumber &\leq & \sum_{k=1}^{\infty} \frac{(1/n)^2}{4\pi (k/n)^2 } + \int_{\frac{1}{n}}^{\infty} \frac{e^{-\frac{n}{n+\lambda}\lambda t}}{2\pi t}\frac{\lambda^2}{n+\lambda} t dt \\
\nonumber &=& \frac{\pi}{24}+\frac{\lambda}{2\pi n}e^{-\frac{\lambda}{n+\lambda}}\\
&\leq & \frac{\pi}{24}+\frac{\lambda}{2\pi }. \label{eq:newex-c}
\end{eqnarray}
Moreover, if \(|x|>|x_n|\), because of \(||x_n|^2-|x|^2|\leq 2/n\), we have
\begin{eqnarray}
\nonumber \int_{1/n}^{\infty} \frac{e^{-\frac{|x_n|^2}{2t}}}{2\pi t}e^{-\lambda t} dt -\int_{1/n}^{\infty} \frac{e^{-\frac{|x|^2}{2t}}}{2\pi t}e^{-\lambda t} dt &\leq & \int_{1/n}^{\infty} \frac{e^{-\frac{|x_n|^2}{2t}}}{2\pi t}e^{-\lambda t} \frac{|x|^2-|x_n|^2}{2t} dt\\
\nonumber &\leq & C\int_{1/n}^{\infty} \frac{e^{-\lambda t}}{t^2n} dt\\
\nonumber &\leq &\int_{1/n}^1 \frac{C}{nt^2}dt +\int_1^{\infty} Ce^{-\lambda t}dt\\
&\leq & C. \label{eq:newex-d}
\end{eqnarray}
If \(|x|\leq |x_n|\), it is clear that
\begin{eqnarray}
\int_{1/n}^{\infty} \frac{e^{-\frac{|x_n|^2}{2t}}}{2\pi t}e^{-\lambda t} dt \leq \int_{1/n}^{\infty} \frac{e^{-\frac{|x|^2}{2t}}}{2\pi t}e^{-\lambda t} dt \label{eq:newex-d}
\end{eqnarray}
Hence, by \((\ref{eq:newex-a})\), \((\ref{eq:newex-b})\), \((\ref{eq:newex-c})\) and \((\ref{eq:newex-d})\), there exists positive constant \(C\) such that
\[g_{\lambda}^n(x_n,0)\leq g_{\lambda}^{\infty}(x,0)+C\]
for any \(x\in \R^2.\)

Furthermore, in a similar way to above, by using the super-additivity of the floor function, there exists positive constant \(C\) such that
\[g_{\lambda}^n(x_n,y_n)\leq g_{\lambda}^{\infty}(x,y)+C\]
for any \(x,y\in \R^2.\)
\end{proof}

\bmhead{Acknowledgements}
The author would like to express my deepest gratitude to Takashi Kumagai and David A. Croydon for continuous discussions and helpful advice. The author thanks Naotaka Kajino and Seiichiro Kusuoka for careful readings of an earlier version of this paper and pointing out typos. The author also thank Xin Sun for showing the reference \cite{Bav}. This work was supported by JSPS KAKENHI Grant Number JP21J20251 and the Research Institute for Mathematical Sciences, an International Joint Usage/Research Center located in Kyoto University.


\begin{appendices}
\section{Definition and properties of PCAF}\label{AppPCAF}
In this appendix, we recall the definition and properties of positive continuous additive functionals (PCAFs, in abbreviation). See \cite{CF,FOT} and \cite{AK} for details.

Let \(E\) be a locally compact separable metric space, \(\mathcal{B}(E)\) be the collection of all Borel sets of \(E\), and \(m\) be a positive Radon measure with \(\text{ supp}(m)=E\).
Denote by \((\mathcal{E},\mathcal{F})\) a regular Dirichlet form  on \(L^2(E;dm)\) and \(Z=(\{Z_t\}_{t\geq 0}, \Omega, \{\mathcal{F}_t\}_t, \{\mathbb{P}_x\}_{x\in E})\) by an \(m\)-symmetric Hunt process on \(E\) associated with \((\mathcal{E},\mathcal{F})\). Let \(\zeta\) be a life time of \(Z\) and \(\{\theta_t\}_t\) be a shift operator, that is \(\theta_t :\Omega \to \Omega\) satisfying \(Z_s(\theta_t \omega)=Z_{s+t}(\omega)\) for any \(\omega \in \Omega.\)

\begin{definition}
A \([-\infty,\infty]\)-valued stochastic process \(A=\{A_t\}_{t\geq 0}\) is called a \textit{positive continuous additive functional} (PCAF in abbreviation) of \(Z\) if there exists \(\Lambda \in \mathcal{F}_{\infty}\) and an \(m\)-inessential set \(N\subset E\) such that \(\mathbb{P}_x(\Lambda)=1\) for \(x\in E\setminus N\) and \(\theta_t \Lambda \subset \Lambda\) for any \(t>0\), and the following conditions hold.\\
\((A.1)\) For each \(t\geq 0,\) \(A_t|_{\Lambda}\) is \(\mathcal{F}_t|_{\Lambda}\)-measurable.\\
\((A.2)\) For any \(\omega \in \Lambda\), \(A_{\cdot}(\omega)\) is continuous on \([0,\infty)\), \(A_0(\omega)=0,\) \(|A_{t}(\omega)|<\infty\) for \(t<\zeta(\omega)\). Moreover the additivity
\[A_{t+s}(\omega)=A_t(\omega)+A_s(\theta_t\omega)\ \ \text{for\ every\ }t,s\geq 0,\]
is satisfied.
\end{definition}
The set \(\Lambda\) is called the defining set of \(A\). In particular, PCAF \(A\) is called PCAF in the strict sense if \(N\) is empty.

PCAFs \(A\) and \(B\) are called \textit{\(m\)-equivalent} if \(\int_E\mathbb{P}_x(A_t\neq B_t)dm(x)=0\) for any \(t>0\). This condition is equivalent to the condition that there exist a common defining set \(\Lambda\) and a common Borel exceptional set \(N\) such that \(A_t(\omega)=B_t(\omega)\) for every \(t \geq 0\) and \(\omega \in \Lambda.\) Moreover, \(\hat{Z}\) defined by \(\hat{Z}_t:=Z_{A^{-1}_t}\) is called the \textit{time-changed process of \(Z\) by \(A\)}, where \(A^{-1}_t:=\inf\{s>0:A_t>s\}\).

Next, we define some quasi-notations and a smooth measure.

\begin{definition}
(i) An increasing sequence \(\{F_k\}_{k\geq 1}\) of closed sets of \(E\) is \textit{\(\mathcal{E}\)-nest} if \(\cup_{k\geq 1} \{f\in \mathcal{F} : f=0 \ m\text{-a.e.\  on\  }(F_k)^c\}\) is \(\mathcal{E}_1\)-dense in \(\mathcal{F}\), where \(\mathcal{E}_1(\cdot,\cdot)=\mathcal{E}(\cdot,\cdot)+(\cdot,\cdot)_{L^2}\).\\
(ii) A subset \(N\) of \(E\) is \textit{\(\mathcal{E}\)-polar} if there exists an \(\mathcal{E}\)-nest \(\{F_k\}_{k\geq 1}\) such that \(N\subset \cap_{k\geq 1} (F_k)^c\).
\end{definition}

\begin{definition}
A positive Borel measure \(\mu\) on \(E\) is called \textit{smooth} if the following conditions hold.\\
\((S.1)\) \(\mu\) charges no \(\mathcal{E}\)-polar set.\\
\((S.2)\) There exists a nest \(\{F_k\}_k\) such that \(\mu(F_k)<\infty\) for every \(k\geq 1.\)
\end{definition}

The following one-to-one correspondence between PCAFs and smooth measures is called Revuz correspondence. So, a smooth measure is also called \textit{Revuz measure}.
\begin{theorem}
\((i)\)For a PCAF \(A\), there exists a unique smooth measure \(\mu\) such that
\begin{equation}
\int_E fd\mu = \lim_{t\searrow 0}\frac{1}{t}\int_E \mathbb{E}_x[\int_0^t f(Z_s)dA_s] dm(x) \label{eq:AppPCAF-1}
\end{equation}
for any positive Borel function \(f\) on \(E.\)\\
\((ii)\) For any smooth measure \(\mu\), there exists a PCAF satisfying \((\ref{eq:AppPCAF-1})\) up to the \(m\)-equivalence.
\end{theorem}

For example, for a bounded positive Borel function \(f\), we define 
\[A_t:=\int_0^t f(Z_s)ds.\]
Then we have \(A:=\{A_t\}_t\) is a PCAF in the strict sense and the corresponding smooth measure is \(fdm\).

Let \(X\) be a massive Gaussian free field on \(\R^2\) built on a probability space \((\Omega^X, \mathcal{M}^X, \mathbb{P}^X)\). For fixed \(\gamma \in [0,2)\), the Liouville measure \(\mu=\mu_{\gamma}\) is defined, see \cite{K,B,S,GRV,AK} for details. For the definition of Liouville Brownian motion on \(\R^2\), we recall the following proposition from \cite[Proposition 2.4]{AK}.
\begin{proposition}
For Brownian motion \(Z=(\{Z_t\}_t, \{\mathbb{P}_x\}_x, \{\mathcal{F}_t\}_t, \Omega^Z)\) on \(\R^2\) and the Liouville measure \(\mu\), there exists a set \(\Lambda \in \mathcal{M}^X \otimes \mathcal{F}^Z_{\infty}\) such that the following holds.\\
\((i)\) For \(\mathbb{P}^X\)-a.e. \(\omega \in  \Omega^X\), \(\mathbb{P}_x(\Lambda^{\omega})=1\) for any \(x\in \R^2\), where \(\Lambda ^{\omega}:=\{\omega' \in \Omega^Z : (\omega, \omega')\in \Lambda\}\).\\
\((ii)\) For \(\mathbb{P}^X\)-a.e. \(\omega \in  \Omega^X\), there exists a PCAF \(\{A_t(\omega,\cdot)\}_t\) of \(Z\) in the strict sense with defining set \(\Lambda^{\omega}\) such that \(\{A_t(\omega,\cdot)\}_t\) corresponds to \(\mu(\omega)\).
\end{proposition}

For these \(Z\) and \(A\), the time-changed process \(\hat{Z}_t:=Z_{A^{-1}_t}\) of \(Z\) by \(A\) is called \textit{Liouville Brownian motion} on \(\R^2\). See \cite{GRV} for details.

Next, let \(g_{\lambda}\) be the \(\lambda\)-order Green's function of Cauchy process \(C\) on \(\R\). Denotes by \(Y\) a Gaussian field with covariance kernel \(\pi g_{\lambda}\) on \(\R\) built on a probability space \((\Omega^Y, \mathcal{M}^Y, \mathbb{P}^Y)\). For fixed \(\gamma \in [0,1)\), the Gaussian multiplicative chaos \(\mu=\mu_{\gamma}\) for \(Y\) is defined, see \cite{K,B,S,Bav} for details. By \cite{Bav} or the same way as \cite[Proposition 2.4]{AK}, the following proposition holds.
\begin{proposition}
For Cauchy process \(C=(\{C_t\}_t, \{\mathbb{P}_x\}_x, \{\mathcal{F}_t\}_t, \Omega^C)\) on \(\R\) and the Gaussian multiplicative chaos \(\mu\) of \(Y\), there exists a set \(\Lambda \in \mathcal{M}^Y \otimes \mathcal{F}^C_{\infty}\) such that the following holds.\\
\((i)\) For \(\mathbb{P}^Y\)-a.e. \(\omega \in  \Omega^Y\), \(\mathbb{P}_x(\Lambda^{\omega})=1\) for any \(x\in \R\), where \(\Lambda ^{\omega}:=\{\omega' \in \Omega^C : (\omega, \omega')\in \Lambda\}\).\\
\((ii)\) For \(\mathbb{P}^Y\)-a.e. \(\omega \in  \Omega^Y\), there exists a PCAF \(\{B_t(\omega,\cdot)\}_t\) of \(C\) in the strict sense with defining set \(\Lambda^{\omega}\) such that \(\{B_t(\omega,\cdot)\}_t\) corresponds to \(\mu(\omega)\).
\end{proposition}

For these \(C\) and \(B\), the time-changed process \(\hat{C}_t:=C_{B^{-1}_t}\) of \(C\) by \(B\) is called \textit{Liouville Cauchy process} on \(\R\). See \cite{Bav} for details.

\begin{remark}
In \cite{Bav}, Liouville-Cauchy process on the unit circle is defined. However, by Spitzer's theorem \cite{Spi}, essentially, there is no difference between Liouville-Cauchy process on the unit circle and Liouville Cauchy process on \(\R.\)
\end{remark}

\section{Skorokhod's topology}\label{AppSko}
In this appendix, we recall definitions on Skorokhod's topology from \cite{Sk}, \cite{JS} and \cite{W}, and state properties of weak convergence with respect to Skorokhod's topologies, in particular, the criterion of tightness with respect to Skorokhod's \(M_1\)-topology.

Let \(d\) be a positive integer and \(D[0,T]\) be the collection of \(\R^d\)-valued c\`{a}dl\`{a}g function on \([0,T]\). Denote by \(D[0,\infty)\) the collection of \(\R^d\)-valued c\`{a}dl\`{a}g function on \([0,\infty)\). We introduce Skorokhod's topologies to \(D[0,\infty)\). Define \(\Lambda_T\) by the collection of all strictly increasing continuous functions \(\lambda:[0,T] \to [0,T]\) satisfying \(\lambda(0)=0\) and \(\lambda(T)=T.\)
\begin{definition}
For \(f_1,f_2 \in D[0,\infty)\), define the metric \(d_{J_1}(\cdot, \cdot)\) on \(D[0,\infty)\) by
\[d_{J_1}^T(f_1,f_2):= \inf_{\lambda \in \Lambda_T} \{\sup_{0\leq t\leq T}|\lambda(t)-t|  \vee \sup_{0\leq t\leq T} |f_1(\lambda(t))-f_2(t)|\},\]
\[d_{J_1}(f_1,f_2):=\int_0^{\infty} e^{-T}(d_{J_1}^T(f_1,f_2) \wedge 1)dT.\]
The topology induced on \(D[0,\infty)\) by \(d_{J_1}\) is called \textit{Skorokhod's \(J_1\)-topology}.
\end{definition}

For \(d=1\), we define \(M_1\)-topology. For \(f\in D[0,T]\), we define the completed graph \(\Gamma_f(T)\) and the set \(\Pi_f(T)\) by
\[\Gamma_f(T):=\{(s,z)\in [0,T]\times \R : z=\theta f(t-)+(1-\theta)f(t)\text{\ for\ some\ }\theta \in [0,1]\},\]
\[\Pi_f(T):=\{(r,u):[0,1]\to \Gamma_f(T) ; (r,u)\text{\ is\ continuous\ surjection\ } r(0)=0, r(1)=T\}.\]
\begin{definition}
For \(f_1,f_2 \in D[0,\infty)\), define the metric \(d_{M_1}(\cdot, \cdot)\) on \(D[0,\infty)\) by
\[d_{M_1}^T(f_1,f_2):= \inf_{(r_i,u_i)\in \Pi_{f_i}(T)} \{\sup_{0\leq s\leq 1}|r_1(s)-r_2(s)|  \vee \sup_{0\leq s\leq 1}|u_1(s)-u_2(s)|\},\]
\[d_{M_1}(f_1,f_2):=\int_0^{\infty} e^{-T}(d_{M_1}^T(f_1,f_2) \wedge 1)dT.\]
The topology induced on \(D[0,\infty)\) by \(d_{M_1}\) is called \textit{Skorokhod's \(M_1\)-topology}.
\end{definition}
\begin{remark}
When \(d\geq 2\), \(M_1\)-topology also be defined on \(D[0,\infty)\). However there are several kinds of \(M_1\)-topologies for \(d\geq 2\) such as the strong \(M_1\)-topology and the weak \(M_1\)-topology. See \cite[Section 12]{W} for details.
\end{remark}

\subsection{Weak convergence with respect to Skorokhod's topologies}
Since sample paths of Markov processes are c\`{a}dl\`{a}g functions, we may treat Markov processes as \(D[0,\infty)\)-valued random variables. So we can consider the weak convergence of processes in \(D[0,\infty)\) with respect to Skorokhod's topologies. Here, we summary and prove the properties of the weak convergence of processes in \(D[0,\infty)\).

For fixed \(T>0,\) we set \(D_T[0,T]:=\{x\in D[0,T] : x_T=x_{T-}\}.\)
For \(z\in D[0,T]\), we define \begin{equation}
v_T(z,t,\delta):=\sup_{0\vee (t-\delta)\leq t_1\leq t_2\leq (t+\delta)\wedge T}|z(t_1)-z(t_2)|, \label{eq:v1}
\end{equation}
\begin{equation}
w_T(z,\delta):=\sup_{0\leq t\leq T}\sup_{\substack{0\vee (t-\delta)\leq t_1\leq t_2
\leq t_3\leq (t+\delta)\wedge T}}\inf_{0\leq \theta \leq 1} |z(t_2)-\left(\theta z(t_1)+(1-\theta) z(t_3)\right)| \label{eq:w1}
\end{equation}

\begin{theorem}[{\cite[Theorem 12.12.3]{W}}]\label{M1cri1} 
A sequence \(\{\mathbb{P}^n\}_n\) of probability measures on \(D_T[0,T]\) with respect to \(M_1\)-topology is tight if and only if the following conditions hold:\\
\((1)\) For any \(\varepsilon >0,\) there exists \(c>0\) such that, for all \(n\), 
\begin{equation}
\mathbb{P}^n(\{x\in D_T[0,T] : \sup_{0\leq t\leq T}|x|>c\})\leq \varepsilon, \label{eq:S3}
\end{equation}
\((2)\) For any \(\varepsilon, \eta >0,\) there exists \(\delta\in (0,T]\) such that, for all \(n\),
\begin{equation}
\mathbb{P}^n (\{x\in D_T[0,T] : w_T(x,\delta) \geq \eta \})\leq \varepsilon,\label{eq:S2}
\end{equation}
\((3)\)For any \(\varepsilon, \eta >0,\) there exists \(\delta\in (0,T]\) such that, for all \(n\),
\begin{equation}\mathbb{P}^n (\{x\in D_T[0,T] : v_T(x,0,\delta) \geq \eta \})\leq \varepsilon, \label{eq:S3}
\end{equation}
\((4)\) For any \(\varepsilon, \eta >0,\) there exists \(\delta\in (0,T]\) such that, for all \(n\),
\begin{equation}\mathbb{P}^n (\{x\in D_T[0,T] : v_T(x,T,\delta) \geq \eta \})\leq \varepsilon. \label{eq:S4}
\end{equation}\end{theorem}

\begin{remark}
\cite[Section 12.12]{W} is based on \cite[Section 2.7]{Sk} and \(D_T[0,T]\) is used in \cite{Sk}.
\end{remark}

We prove the following theorem on a characterization of a tightness of a sequence of probability measures on \(D[0,\infty)\) with respect to \(M_1\)-topology.

\begin{theorem}\label{M1cri2} 
A sequence \(\{\mathbb{P}^n\}_n\) of probability measures on \(D[0,\infty)\) with respect to \(M_1\)-topology is tight if and only if the following conditions hold:\\
\((1)\) For any \(T, \varepsilon >0,\) there exists \(c>0\) such that, for all \(n\), 
\begin{equation}
\mathbb{P}^n(\{x\in D[0,T] : \sup_{0\leq t\leq T}|x|>c\})\leq \varepsilon, \label{eq:T3}
\end{equation}
\((2)\) For any \(T, \varepsilon, \eta >0,\) there exists \(\delta\in (0,T]\) such that, for all \(n\),
\begin{equation}
\mathbb{P}^n (\{x\in D[0,T] : w_T(x,\delta) \geq \eta \})\leq \varepsilon,\label{eq:T2}
\end{equation}
\((3)\)For any \(T, \varepsilon, \eta >0,\) there exists \(\delta\in (0,T]\) such that, for all \(n\),
\begin{equation}\mathbb{P}^n (\{x\in D[0,T] : v_T(x,0,\delta) \geq \eta \})\leq \varepsilon. \label{eq:T3}
\end{equation}
\end{theorem}
\begin{proof}
It is sufficient to show that \(K\subset D[0,\infty)\) is relatively compact with respect to \(M_1\)-topology if and only if the following hold;

\((1)'\) For any \(T>0,\) \(\sup_{x\in K} \sup_{0\leq t\leq T}|x_t|<\infty.\)

\((2)'\) For any \(T>0,\) \(\lim_{\delta \searrow 0}\sup_{x\in K} w_T(x,\delta)=0.\)

\((3)'\) \(\lim_{\delta \searrow 0}\sup_{x\in K}v_T(x,0,\delta)=0.\)

Suppose \(K\subset D[0,\infty)\) is relatively compact with respect to \(M_1\)-topology. If \((1)'\) does not hold, there exist \(T>0\) and \(\{x^{(n)}\}\subset K\) such that \(\lim_{n\to \infty}\sup_{t\leq T}|x_t^{(n)}|=\infty\). We take a subsequence \(\{x^{(n')}\}\) satisfying \(x^{(n')}\) converges to some \(x\in \bar{K}\) with respect to \(M_1\)-topology. For \(\tilde{T}\geq T\) satisfying \(x_{\tilde{T}}=x_{\tilde{T}-}\), by \cite[Theorem 12.5.1]{W}, \(\tilde{x}^{(n')}\) converges to \(x\) on \([0,\tilde{T}]\) with respect to \(M_1\)-topology, where \(\tilde{x}^{(n')}\in D_{\tilde{T}}[0,\tilde{T}]\) is defined by \(\tilde{x}^{(n')} _t:=x^{(n')} _t\) for \(t< \tilde{T}\) and \(\tilde{x}^{(n')} _{\tilde{T}}:=x^{(n')}_{\tilde{T}-}\). By \cite[2.7.2]{Sk}, \(\sup_{n'}\sup_{0\leq t \leq \tilde{T}}|\tilde{x}^{(n')} _t |<\infty \). This is a contradiction, so \((1)'\) holds. Similarly, \((2)'\) and \((3)'\) hold.

Suppose \((1)'\), \((2)'\) and \((3)'\). We will prove \(K\) is relatively compact in a similar way to the proof of \cite[2.7.2.]{Sk}. For any \(\{x^{(n)}\} \subset K\), by \((1)'\), there are \(\{n'\}\) and \(\{t_i\}\) such that \(\{t_i\}\) is dense in \([0,\infty)\) and \(x^{(n')}_{t_i}\) converges to some \(x_{t_i}\) as \(n'\to \infty .\) By using \((2)'\) and \((3)'\), we can prove the existence of \(\lim_{t_i \searrow t}x_{t_i}\) and \(\lim_{t_i \nearrow t}x_{t_i}\) for \(t\geq 0\) in the same way as \cite[2.3.5.]{Sk}. Setting \(\bar{x}_t:=\lim_{t_i \searrow t}x_{t_i} \in D[0,\infty). \) For any \(T\geq 0\) satisfying \(\bar{x}_T=\bar{x}_{T-}\), \(x^{(n')}\) converges pointwise to \(\bar{x}\) on \(\{t\in [0,T]\ \mid \ \bar{x}_t=\bar{x}_{t-}\}\). Indeed, for \(\varepsilon>0\) and \(t\in [0,T]\) with \(\bar{x}_t=\bar{x}_{t-}\), we can take \(\delta>0\) such that \(w_T(x^{(n')}, \delta)\leq \varepsilon\) by \((2)'\) and \(t', t''\) satisfying \(t-\delta <t'<t-\delta/2, t+\delta /2<t''<t+\delta\) and \(|\bar{x}_{t'}-\bar{x}_{t}|<\varepsilon \text{\ and\ }|\bar{x}_{t''}-\bar{x}_{t}|<\varepsilon.\) 
Then we have 
\[|x^{(n')}_t-\bar{x}_{t}|\leq w_T(x^{(n')}, \delta)+ |\bar{x}_{t'}-\bar{x}_{t}|+w_T(x^{(n')}, \delta)+|\bar{x}_{t''}-\bar{x}_{t}| \leq 4\varepsilon .\] 
By \cite[Theorem 12.5.1.]{W} and \((2)'\), \(x^{(n')}\) converges to \(\bar{x}\) on \(D[0,T]\) with respect to \(M_1\)-topology. So \(x^{(n')}\) converges to \(\bar{x}\) on \(D[0,\infty)\) with respect to \(M_1\)-topology by \cite[Theorem 12.5.1]{W}.
\end{proof}

\begin{remark}
The difference between Theorem \ref{M1cri1} and Theorem \ref{M1cri2} is the condition \((4)\). The condition \((4)\) is used to guarantee a subsequential limit is continuous at the end point of the interval. However, when we consider the convergence with respect to \(M_1\)-topology on \(D[0,\infty)\), by the definition, we only consider the convergence of the restriction to the interval whose end points are continuous points for the limit. This is the intuitive reason why \((4)\) does not appear when we treat the convergence on \([0,\infty).\)
\end{remark}

Let \(Z^n\) be a c\`{a}dl\`{a}g process on \(\R^d\) on a probability space \((\Omega^n, \mathbb{P}^n)\) for each \(n\in \mathbb{N} \cup \{\infty\}.\) Then \(Z^n\) can be viewed as a \(D[0,\infty)\)-valued random variable, so we can consider the weak convergence of c\`{a}dl\`{a}g processes with  respect to Skorokhod's topology.
\begin{theorem}[{\cite[Theorem 11.6.6]{W}}] \label{whitt}
A process \(Z^n\) under \(\mathbb{P}^{Z^n}\) converges weakly to \(Z^{\infty}\) under \(\mathbb{P}^{Z^{\infty}}\) with respect to \(J_1\) \((\)resp. \(M_1\)\()\)-topology if and only if the following conditions hold.\\
\((1)\) \(Z^n\) converges to \(Z^{\infty}\) in finite-dimensional distribution on \(\{t>0; \mathbb{P}^{Z^{\infty}}(Z_{t-}^{\infty} \neq Z_t^{\infty})=0\}\).\\
\((2)\) The collection of distributions of \(Z^n\) under \(\mathbb{P}^{Z^n}\) is tight with respect to \(J_1\) \((\)resp. \(M_1\)\()\)-topology.
\end{theorem}

Time-changed processes can be represented as the compositions of processes and the inverses of PCAFs, so we consider the continuity of composition maps with  respect to Skorokhod's topologies.

Let \(D(E;F):=\{f:E\to F : f\text{ is\ c\`{a}dl\`{a}g} \}\) for subsets \(E\) of \([0,\infty)\) and \(F\) of \(\R^d\). We define 
\(D(E):=D(E;\R)\), \(D^d:=D([0,\infty);\R^d)\), \(D^+:=D([0,\infty);[0,\infty))\), \(D_{\uparrow}:=\{f\in D^+ : f\text{\ is\ non-decreasing}\}\) and \(D_{\upuparrows}:=\{f\in D_{\uparrow} : f\text{\ is\ strictly\  increasing}\}.\) Moreover, we define spaces \(C(E;F), C(E),\) and so on, replaced c\`{a}dl\`{a}g functions by continuous functions above.

\begin{theorem}[{\cite[Theorem 13.2.2]{W}}]\label{J1UJ1}
The composition map from \(D^d\times D_{\uparrow}\) with respect to \(J_1\times J_1\)-topology taking \((x,y)\) into \(x\circ y\) with respect to \(J_1\)-topology is continuous at \((x,y)\in (C^d \times D_{\uparrow}) \cup (D^d \times C_{\upuparrows})\).
\end{theorem}

\begin{theorem}\label{J1M1L1}
The composition map from \(D^d\times D_{\uparrow}\) with respect to \(J_1\times M_1\)-topology  taking \((x,y)\) into \(x\circ y\) with respect to \(L^1_{loc}\)-topology is continuous at \((x,y)\in (C^d \times D_{\uparrow}) \cup (D^d \times D_{\upuparrows})\).
\end{theorem}

\begin{proof}
The continuity of this composition map at the point \((x,y)\in C^d \times D_{\uparrow}\) is proved in \cite[Lemma A.6]{CM}, so we prove the continuity at \((x,y)\in D^d \times D_{\upuparrows}\).
We assume that \(x^n\in D^d\) converges to \(x\in D^d\) with respect to \(J_1\)-topology and \(y^n\in D_{\uparrow}\) converges to \(y\in  D_{\upuparrows}\) with respect to \(M_1\)-topology. For any \(T\), there exists \(\tilde{T}\) such that \(\sup_{t\leq T}|y_t^n|\leq \tilde{T}\) holds. Then, there exists strictly increasing homeomorphism \(\lambda^n : [0,\tilde{T}]\to [0,\tilde{T}]\) such that
\begin{equation}
\lim_{n\to \infty}(\sup_{t\leq \tilde{T}}|x^n_t-x_{\lambda^n(t)}|+\sup_{t\leq \tilde{T}}|\lambda^n(t)-t|)=0.\label{eq:sk3-1}
\end{equation}

So we have
\begin{equation}
\int_0^T |x^n_{y_t^n}-x_{y_t}|dt \leq \int_0^T |x^n_{y_t^n}-x_{\lambda^n(y^n_t)}|dt  +\int_0^T |x_{\lambda^n(y^n_t)}-x_{y_t}|dt \label{eq:J1M1L1-1}
\end{equation}
The first term of (\ref{eq:J1M1L1-1}) is bounded by \(T \sup_{t\leq \tilde{T}}|x^n_t-x_{\lambda^n(t)}|\) and this converges to \(0\) as \(n\to \infty\) by (\ref{eq:sk3-1}). The space \(\{t\leq T : y_t=y_{t-}\}\) is dense in \([0,T]\), and the convergence of \(y^n\) with respect to \(M_1\)-topology yields that \(\lim_{n\to \infty}|y_t^n-y_t|=0\) for \(t\) satisfying \(y_t=y_{t-}\) by \cite[Theorem 12.5.1]{W}. So we have 
\[\varlimsup_{n\to \infty}|\lambda^n(y_t^n)-y_t| \leq \varlimsup_{n\to \infty}\sup_{s\leq \tilde{T}}|\lambda^n(s)-s|+\varlimsup_{n\to \infty}|y_t^n-y_t|=0\]
for almost every \(t>0.\)
Since \(x\) is strictly increasing, there exist at most countable \(t>0\) satisfying \(y_t=y_{t-}\) and \(x_{y_t}\neq x_{y_{t-}}\). Thus the second term of (\ref{eq:J1M1L1-1}) converges to \(0\) as \(n\to \infty\). This completes the proof.
\end{proof}

\begin{remark}
In the above theorem, the strict increase of \(y\) is necessary when \(x \in D^d\setminus C^d\). Indeed, we can construct the counterexample as follows:

Define \(x^n,x\) and \(y \in D[0,1]\) by \(x_t^n:={\bf 1}_{[\frac{1}{2}+\frac{1}{n},1]}(t),\ \  x_t:={\bf 1}_{[\frac{1}{2},1]}(t)\) and \(y_t:=\frac{3t}{2}{\bf 1}_{[0,\frac{1}{3})}(t)+\frac{1}{2}{\bf 1}_{[\frac{1}{3},\frac{2}{3})}(t)+(\frac{3t}{2}-\frac{1}{2}){\bf 1}_{[\frac{2}{3},1]}(t).\) Then, for the homeomorphism \(\lambda^n_t:=(1+2/n)t{\bf 1}_{[0,1/2)}+((1-2/n)t+2/n){\bf 1}_{[1/2,1]}\), we have
\(x^n_{\lambda^n_t}-x_t=0\) and \(|\lambda^n_t-t|\leq 2/n\), so \(x^n\) converges to \(x\) with respect to \(J_1\)-topology.
On the other hand, \(x^n_{y_t}={\bf 1}_{[2/3+2/3n,1]}\) and \(x_{y_t}={\bf 1}_{[1/3,1]}\) and it holds that
\[\lim_{n\to \infty}\int_0^1 |x_{y_t}^n-x_{y_t}|dt =\lim_{n\to \infty} \int_{1/3}^{2/3+2/3n} dt=\lim_{n\to \infty}\left(\frac{1}{3}+\frac{2}{3n}\right)=\frac{1}{3}.\] So \(x^n_{y}\) does not converge to \(x_y\) with respect to \(L^1_{loc}\)-topology.
\end{remark}




\end{appendices}



\begin{thebibliography}{99}
	\bibitem{AK} Andres, S., Kajino, N.: Continuity and estimates of the Liouville heat kernel with applications to spectral dimensions. Probab. Theory Related Fields \textbf{166}, 713-752 (2016)
	\bibitem{Bav} Baverez, G.: On Gaussian multiplicative chaos
and conformal field theory. PhD thesis, University of Cambridge, 2021. \url{https://doi.org/10.17863/CAM.83226}
	\bibitem{B2} Berestycki, N.: Diffusion in planar Liouville quantum gravity. Ann. Inst. Henri Poincar\'{e} Probab. Stat. \textbf{51}, 947-964 (2015) 
	\bibitem{B} Berestycki, N.: An elementary approach to Gaussian multiplicative chaos. Electron. Commun. Probab. \textbf{22}, Paper No. 27, 12 pp. (2017) 
	\bibitem{BeG} Berestycki, N., Gwynne, E.: Random walks on mated-CRT planar maps and Liouville Brownian motion. Comm. Math. Phys. \textbf{395}, 773-857 (2022)  
	\bibitem{BP} Berestycki, N., Powell, E.: Gaussian free field, Liouville quantum gravity and Gaussian multiplicative chaos, \url{https://homepage.univie.ac.at/nathanael.berestycki/}
	\bibitem{BG} Blumenthal, R. M., Getoor, R. K.: Some theorems on stable processes.Trans. Amer. Math. Soc. \textbf{95}, 263-273 (1960)
	\bibitem{CF} Chen, Z.-Q., Fukushima, M.: Symmetric Markov processes, time change, and boundary theory. Princeton University Press, Princeton, NJ, 2012. xvi+479 pp.
	\bibitem{CoF} Cont, R., Fourni\'{e}, D. A.: Change of variable formulas for non-anticipative functionals on path space. J. Funct. Anal. \textbf{259}, 1043-1072 (2010)
	\bibitem{CHK} Croydon, D., Hambly, B., Kumagai, T.: Time-changes of stochastic processes associated with resistance forms. Electron. J. Probab. \textbf{22}, Paper No. 82, 41 pp. (2017)
	\bibitem{CM} Croydon, D., Muirhead, S.: Functional limit theorems for the Bouchaud trap model with slowly varying traps. Stochastic Process. Appl. \textbf{125}, 1980-2009 (2015)
	\bibitem{D} Dudley, R.M.: Real analysis and probability. Revised reprint of the 1989 original, Cambridge Studies in Advanced Mathematics, 74. Cambridge University Press, Cambridge, 2002. x+555 pp.
	\bibitem{DS} Duplantier, B., Sheffield, S.: Liouville quantum gravity and KPZ. Invent. Math. \textbf{185}, 333-393 (2011)
	\bibitem{FO} Fukushima, M., Oshima, Y.: Gaussian fields, equilibrium potentials and multiplicative chaos for Dirichlet forms, Potential Anal. \textbf{55}, 285-337 (2021)
	\bibitem{FOT} Fukushima, M., Oshima, Y., Takeda, M.: Dirichlet forms and symmetric Markov processes. 2nd rev. and ext. ed. Walter de Gruyter \& Co., Berlin, 2011. x+489 pp.
	\bibitem{GRV} Garban, C., Rhodes, R., Vargas, V.: Liouville Brownian motion. Ann. Probab. \textbf{44}, 3076-3110 (2016)
		\bibitem{GN} Gin\'{e}, E., Nickl, R.: Mathematical foundations of infinite-dimensional statistical models. Cambridge University Press, New York, 2016. xiv+690 pp. 
		\bibitem{GMS} Gwynne, E., Miller, J., Sheffield, S.: The Tutte embedding of the mated-CRT map converges to Liouville quantum gravity. Ann. Probab. \textbf{49}, 1677?-1717 (2021) 
	\bibitem{HN} Hager, P., Neuman, E.: The multiplicative chaos of \(H=0\) fractional Brownian fields. Ann. Appl. Probab. \textbf{32}, 2139-2179 (2022)
	\bibitem{JS} Jacod, J., Shiryaev, A. N.: Limit theorems for stochastic processes. Second edition. Springer-Verlag, Berlin, 2003. xx+661 pp. 
	\bibitem{K} Kahane, J.-P.: Sur le chaos multiplicatif. Ann. Sci. Math. Qu\'{e}bec \textbf{9},  105-150 (1985)
	\bibitem{Ka} Kallenberg, O.: Foundations of modern probability. Second edition. Springer-Verlag, New York, 2002. xx+638 pp.
	\bibitem{KS} Karatzas, I., Shreve, S. E.: Brownian motion and stochastic calculus. Second edition. Springer-Verlag, New York, 1991. xxiv+470 pp. 
	\bibitem{LL} Lawler, G. F., Limic, V.: Random walk: a modern introduction. Cambridge University Press, Cambridge, 2010. xii+364 pp.
	\bibitem{LSSW} Lodhia, A., Sheffield, S., Sun, S., Watson, S.S.: Fractional Gaussian fields: a survey. Probab. Surv. \textbf{13}, 1-56 (2016)
	\bibitem{MP} M\"{o}rters, P., Peres, Y., Brownian motion. With an appendix by Oded Schramm and Wendelin Werner. Cambridge Series in Statistical and Probabilistic Mathematics, \textbf{30}. Cambridge University Press, Cambridge, 2010. xii+403 pp.
	\bibitem{O} Ooi, T.: Markov properties for Gaussian fields associated with Dirichlet forms, Osaka J. Math. \textbf{60}, 579-595 (2023) 
	\bibitem{O2} Ooi, T.: Convergence of processes time-changed by Gaussian multiplicative chaos, extended version, arXiv:2305.00734
	\bibitem{Ro} R\"{o}ckner, M.: Generalized Markov fields and Dirichlet forms, \emph{Acta. Appl. Math} \textbf{3}, 285-311 (1985)
	\bibitem{R} Rudin, E.: Real and complex analysis. Third edition. McGraw-Hill Book Co., New York, 1987. xiv+416 pp. 
	\bibitem{S} Shamov, A.: On Gaussian multiplicative chaos. J. Funct. Anal. \textbf{270}, 3224-3261 (2016),
	\bibitem{Sk} Skorohod, A.V.: Limit theorems for stochastic processes, Teor. Veroyatnost. i Primenen. \textbf{1}, 289-319 (1956)
\bibitem{Spi} Spitzer, F.: Some Theorems Concerning 2-Dimensional Brownian Motion, Transactions of the American Mathematical Society, \textbf{87}, 187-197 (1958)
	\bibitem{W} Whitt, W.: Stochastic-process limits. An introduction to stochastic-process limits and their application to queues. Springer-Verlag, New York, 2002. xxiv+602 pp. 

\end{thebibliography}
\end{document}